\documentclass[amstex,12pt,reqno]{amsart}

\usepackage{fancybox,color,epsf,slidesec,epic,eepic,amsmath,amsfonts,amssymb,mathrsfs}
\usepackage{graphicx}
\usepackage{epsfig}
\usepackage[all,cmtip]{xy}
\usepackage[small,nohug,heads=vee]{diagram1}
\usepackage[bottom=1.5cm, right=1.5cm, left=1.5cm, top=1.5cm]{geometry}
\usepackage[pdftex]{pict2e}
\usepackage{mathtools}

\catcode`\@=11

\@addtoreset{equation}{section}
\@addtoreset{equation}{subsection}
\def\theequation{\ifnum\value{subsection}>0\relax
\thesubsection.\arabic{equation}\relax
\else\ifnum\value{section}>0\relax
\thesection.\arabic{equation}\relax
\else\arabic{equation}\fi\fi}

\newtheorem{thm}[equation]{Theorem}
\newtheorem{lem}[equation]{Lemma}
\newtheorem{prop}[equation]{Proposition}
\newtheorem{cor}[equation]{Corollary}
\newtheorem{rem}[equation]{Remark}

\newtheorem{defn}[equation]{Definition}

\def\blk#1{{}}

\newcommand{\wt}{\widetilde}

\newcommand{\cover}[1]{\widetilde{#1}}

\newcommand{\bfG}{{\mathbf G}}
\newcommand{\bfK}{{\mathbf K}}

\newcommand{\bfV}{{\mathbf V}}

\newcommand{\bfe}{{\mathbf e}}

\newcommand{\bfs}{{\mathbf s}}
\newcommand{\bft}{{\mathbf t}}
\newcommand{\bftheta}{{\boldsymbol \theta}}

\newcommand{\g}{{\mathfrak g}}

\newcommand{\kk}{{\mathfrak k}}
\renewcommand{\u}{{\mathfrak u}}
\newcommand{\m}{{\mathfrak m}}
\newcommand{\n}{{\mathfrak n}}
\newcommand{\p}{{\mathfrak p}}
\newcommand{\q}{{\mathfrak q}}

\renewcommand{\sl}{{\mathfrak {sl}}}
\renewcommand{\sp}{{\mathfrak {sp}}}

\newcommand{\R}{{\mathbb R}}
\newcommand{\C}{{\mathbb C}}
\newcommand{\Z}{{\mathbb Z}}

\newcommand{\HH}{{\mathbb H}}
\newcommand{\NN}{{\mathbb N}}
\newcommand{\VV}{{\mathbb V}}
\newcommand{\WW}{{\mathbb W}}

\newcommand{\Sp}{{\mathrm{Sp}}}
\newcommand{\GL}{\mathrm{GL}}

\newcommand{\CE}{{\mathcal E}}

\newcommand{\CF}{{\mathcal F}}

\newcommand{\CO}{{\mathcal O}}

\newcommand{\CY}{{\mathcal Y}}
\newcommand{\CM}{{\mathcal M}}
\newcommand{\CK}{{\mathcal K}}

\newcommand{\CR}{{\mathcal R}}
\newcommand{\CQ}{{\mathcal Q}}

\newcommand{\CU}{{\mathcal U}}

\newcommand{\Orb}{{\mathcal O}}
\newcommand{\bCO}{{\overline{\mathcal O}}}

\def\SE{\mathscr{E}}
\newcommand{\SP}{\mathscr{P}}
\newcommand{\SH}{\mathscr{H}}
\newcommand{\SL}{\mathscr{L}}
\newcommand{\SM}{\mathscr{M}}
\newcommand{\SN}{\mathscr{N}}
\newcommand{\SX}{\mathscr{X}}
\newcommand{\SY}{\mathscr{Y}}
\newcommand{\SZ}{\mathscr{Z}}
\newcommand{\ST}{\mathscr{T}}
\renewcommand{\SS}{\mathscr{S}}
\newcommand{\SV}{\mathscr{V}}
\newcommand{\SA}{\mathscr{A}}
\newcommand{\SO}{\mathscr{O}}

\def\SW{\mathscr{W}}

\newcommand{\sfs}{{\mathsf s}}

\renewcommand{\rm}{{\mathrm m}}
\newcommand{\rn}{{\mathrm n}}
\newcommand{\rD}{{\mathrm D}}
\newcommand{\rG}{{\mathrm G}}
\newcommand{\rF}{{\mathrm F}}
\newcommand{\rH}{{\mathrm H}}
\newcommand{\rJ}{{\mathrm J}}
\newcommand{\rR}{{\mathrm R}}
\newcommand{\rS}{{\mathrm S}}
\newcommand{\rd}{{\mathrm d}}
\newcommand{\rO}{{\mathrm O}}
\newcommand{\rM}{{\mathrm M}}
\newcommand{\rU}{{\mathrm U}}

\newcommand{\tV}{{\tilde{V}}}
\newcommand{\tX}{{\tilde{X}}}
\newcommand{\tY}{{\tilde{Y}}}

\newcommand{\tB}{{\tilde{B}}}
\newcommand{\tv}{{\tilde{v}}}
\newcommand{\tw}{{\tilde{w}}}
\newcommand{\tgamma}{{\tilde{\gamma}}}
\newcommand{\tphi}{{\tilde{\varphi}}}
\newcommand{\tepsilon}{{\tilde{\epsilon}}}

\newcommand{\Ma}{{\mathrm M_{p+q,n}(\mathbb R)}}
\newcommand{\SMa}{{{\mathscr S}(\Ma)}}
\newcommand{\TMa}{{{\mathscr S}^{*}(\Ma)}}

\newcommand{\Irr}{\operatorname{Irr}}
\newcommand{\Hom}{\operatorname{Hom}}
\newcommand{\GNM}{\operatorname{GNM}}
\newcommand{\Tr}{\operatorname{Tr}}

\newcommand{\Ad}{\operatorname{Ad}}
\newcommand{\ad}{\operatorname{ad}}
\newcommand{\Span}{\operatorname{Span}}

\newcommand{\Ker}{\operatorname{Ker}}
\renewcommand{\Im}{\operatorname{Im}}

\newcommand{\Ann}{\operatorname{Ann}}
\newcommand{\AV}{\operatorname{AV}}
\newcommand{\AC}{\operatorname{AC}}
\newcommand{\Nil}{\operatorname{Nil}}
\newcommand{\Wh}{\operatorname{Wh}}
\newcommand{\Ind}{\operatorname{Ind}}

\newcommand{\sgn}{\operatorname{sgn}}

\newcommand{\one} {{1\!\! 1}}

\newcommand{\set}[2]{ \left\{ {#1} \, \left| \, {#2} \right\}\right.}

\newcommand{\vsp}{{\vspace{0.2in}}}

\newcommand{\la}{{\langle}}
\newcommand{\ra}{{\rangle}}

\def\abs#1{\left|{#1}\right|}

\def\inn#1#2{\left\langle
      \def\ta{#1}\def\tb{#2}
      \ifx\ta\@empty{\;} \else {\ta}\fi ,
      \ifx\tb\@empty{\;} \else {\tb}\fi
      \right\rangle}

\def\Thetab{\bar{\Theta}}

\begin{document}

\title[Local Theta Correspondence]{Lectures on Local Theta Correspondence}

\begin{abstract} This set of lecture notes on local theta correspondence is the written version of a mini-course the author gave in March of 2025 for the program ``Representation Theory and Noncommutative Geometry" at the Institut Henri Poincar\'e, Paris. The emphasis is on the Archimedean theory, which concerns representations of classical Lie groups. Section 1 is about the basic theory, including Howe's Duality Theorem, and the conservation relations. Section 2 highlights the invariant-theoretic nature of local theta correspondence via the proof of the conservation relations. Sections 3 and 4 explain how two fundamental invariants of representations behave under local theta correspondence. The final section discusses applications to unitary representation theory.

\end{abstract}

\author{Chen-Bo Zhu }
\thanks{Author's note: The choices of contents of these notes reflect interests and biases of the author. A good part of the notes are based on the work of the author and his collaborators. For complementary reading of the main topics of these notes, see Li's article \cite{Li00}. For the non-Archimedean theory, see Prasad's article \cite{Pr} and the forthcoming book of Gan, Kudla and Takeda \cite{GKT}.
}

\address{Department of Mathematics\\
National University of Singapore\\
10 Lower Kent Ridge Road\\
Singapore 119076}
\email{matzhucb@nus.edu.sg}

\maketitle

\tableofcontents

\section{The basic theory}

The basic references for this section are two papers of Howe \cite{Ho79,Ho89}. See also Howe's unpublished notes on the oscillator representation \cite{HoPre1,HoPre2}. For conservation relations, see \cite{KR2, SZJou}.

In what follows, $\rF$ will always denote a local field of characteristic $0$. In addition we will fix a nontrivial unitary character $\psi$ of $\rF$.

\subsection{Reductive dual pairs and classification}\label{dual pair}

Let $W$ be a finite-dimensional symplectic vector space over $\rF$ with symplectic form $\la, \ra_W: W\times W\rightarrow \rF$. Let $\Sp(W)$ be the isometry group of $\la, \ra_W$.

\begin{defn} [Howe] A reductive dual pair in $\Sp(W)$ is a pair of closed reductive subgroups $G,G'$ of $\Sp(W)$ such that $G$ and $G'$ are mutual centralizers.
\end{defn}

If $(G,G')$ is a reductive dual pair in $\Sp(W)$, and if $W=W_1\oplus W_2$ is an orthogonal direct sum decomposition where $W_1$ and $W_2$ are invariant by $G\cdot G'$, then we say $(G,G')$ is reducible. The restrictions $(G_i,G_i')$ of $(G,G')$ to the $W_i$ define reductive dual pairs in the $\Sp(W_i)$. Then we say that $(G,G')$ is the direct product of the $(G_i,G'_i)$. If $(G,G')$ is not reducible, it is then called irreducible. Any reductive dual pair is a direct product in an essentially unique way of irreducible reductive dual pairs.

Here are the alternative definitions. Denote by $\sigma$ the anti-involution of $\mathrm{End}_\rF(W)$ specified by
\[
   \la x\cdot u, v\ra_W=\la u, x^\sigma\cdot v\ra_W, \qquad u,v\in W,\,x\in \mathrm{End}_\rF(W).
\]
Then $\Sp(W)=\{x\in \mathrm{End}_\rF(W)\mid x^\sigma x=1\}$.
Let $(A, A')$ be a pair of $\sigma$-stable semisimple $\rF$-subalgebras of $\mathrm{End}_\rF(W)$ that are mutual centralizers. Set
\[G:=A\cap \Sp(W),\,\,\textrm{and}\,\,\,G':=A'\cap \Sp(W),\] which are closed subgroups of $\Sp(W)$. The pair of groups $(G,G')$ is then a reductive dual pair in $\Sp(W)$. The dual pair $(G, G')$ is irreducible if the algebra $A$ (or equivalently, $A'$) is either simple or the product of two simple algebras that are exchanged by $\sigma$.

Complete classification of irreducible reductive dual pairs has been given by Howe (\cite[Section 6]{HoPre1}), as described in what follows. For simplicity, we shall describe this classification for the case $\rF=\R$.

The irreducible reductive dual pairs fall into two classes, which are called type I and type II. Those of type II correspond to division algebras, of which there are three containing $\R$, namely $\R$ itself, the complex numbers $\C$ and the quaternions $\HH$. Those of type I corresponding to division algebras with involution (i.e., involutory antiautomorphism). There are four of these containing $\R$, namely, $\R$ with the trivial involution, $\C$ with the trivial involution, $\C$ with the complex conjugation, and $\HH$ with the quaternionic conjugation.

For a uniform description,
let $(\rD, \natural)$ (where $\natural: a\rightarrow a^{\natural}$ is an involution of $\rD$) be one of the following seven pairs:
\begin{equation}
 (\R, \textrm{id}), \quad (\C, \textrm{id}),\quad (\C, \overline{\phantom a}), \quad (\HH,  \overline{\phantom a}),\tag{I}
 \end{equation}
 \begin{equation}
 (\R\times \R, ((x,y)\mapsto (y,x)),\quad (\C\times \C, ((x,y)\mapsto (y,x)),\quad (\HH\times \HH, ((x,y)\mapsto (\bar y,\bar x)),\tag{II}
\end{equation}
where $\textrm{id}$ denotes the identity map, and $\overline{\phantom a}\,$ indicates the complex conjugation or the quaternionic conjugation.

Let $\epsilon=\pm 1$. Let $V$ be an $\epsilon$-Hermitian $\rD$-space, namely a free right $\rD$-module of finite rank,  equipped with a non-degenerate $\R$-bilinear map
\[
\la \cdot , \cdot  \ra_{V}: V\times V\rightarrow \rD
\]
such that
\[
  \la u a, v\ra_V=\la u,v\ra_V a, \quad \la u,v\ra_V=\epsilon \la v,u\ra_V^{\natural}, \quad \textrm{for all }u,v\in V, \, a\in \rD.
\]
This $\R$-bilinear map is called the $\epsilon$-Hermitian form on $V$, whose isometry group is denoted by $\rG(V)$. (We will also refer to $V$ as the standard module of the classical group $\rG(V)$.)

Let $V'$ be an $\epsilon'$-Hermitian $\rD$-space, equipped with an $\epsilon'$-Hermitian form
$\la \cdot , \cdot  \ra_{V'}$, where $\epsilon \epsilon '=-1$. 
Let $W:=\Hom_{\rD}(V,V')$, equipped with the symplectic form $\langle\cdot , \cdot  \rangle _W$ given by
\[
\la T, S\ra_{W}:=\Tr_{\R} (T^{\ast}S), \qquad \mbox{$T$, $S\in \Hom_{\rD}(V,V')$,}
\]
where $\Tr_{\R} (T^{\ast}S)$ is the trace of $T^{\ast}S$ as a $\R$-linear transformation, and $T^{\ast}\in \Hom_{\rD}(V',V)$ is the adjoint of $T$ defined by
\begin{equation}\label{adj}
\langle Tv,v'\rangle_{V'}=\langle v,T^{\ast}v'\rangle_{V}, \qquad \mbox{for all $\, v\in V$, $v'\in V'$.}
\end{equation}

There is a natural homomorphism: $\rG(V) \times \rG(V')\longrightarrow \Sp(W)$ given by
\[
 (g,g')\cdot T = g' T g^{-1} \qquad \mbox{for $T \in \Hom_{\rD}(V,{V}')$, $g\in G$, $g'\in G'$}.
\]
If both $V$ and $V'$ are nonzero, then $\rG(V)$ and $\rG(V')$ are both identified with subgroups of $\Sp(W)$, and $(\rG(V),\rG(V'))$ is an irreducible reductive dual pair in $\Sp(W)$. Moreover, all irreducible reductive dual pairs arise in this way.

When $(\rD, \natural)$ consists of a division algebra with involution, i.e. from (I), the irreducible dual pair $(\rG(V),\rG(V'))$ is called type I.
When $(\rD, \natural)$ is from (II), the irreducible dual pair $(\rG(V),\rG(V'))$ is called type II. A list of irreducible dual pairs of type I and type II is as follows.

Type I:
\[(\rO_{p,q}, \Sp_{2n}(\R))\subseteq \Sp_{2(p+q)n}(\R)\]
\[(\rO_{p}(\C), \Sp_{2n}(\C))\subseteq \Sp_{4pn}(\R)\]
\[(\rU_{p,q}, \rU_{r,s})\subseteq \Sp_{2(p+q)(r+s)}(\R)\]
\[(\Sp_{p,q}, \rO^*_{2n})\subseteq \Sp_{4(p+q)n}(\R)\]

Type II:
\[(\GL_m(\R), \GL_n(\R))\subseteq \Sp_{2mn}(\R)\]
\[(\GL_m(\C), \GL_n(\C))\subseteq \Sp_{4mn}(\R)\]
\[(\GL_m(\HH), \GL_n(\HH))\subseteq \Sp_{8mn}(\R)\]

\subsection{Oscillator representations and Howe's Duality Theorem}\label{duality}

Write  $\rH(W):=W \times \rF$ for  the Heisenberg group with group multiplication
\[
  (u,t)\cdot (u',t')=(u+u', t+t'+\frac{1}{2}\la u, u'\ra_{W}), \qquad u,u'\in W, \ t, t'\in \rF.
\]
Its center is obviously identified with $\rF$.
Recall the Stone-von
Neumann Theorem which asserts that up to isomorphism, there exists a unique irreducible
unitary representation of
$\rH(W)$ with central character $\psi$. See \cite{Ca,HoHei}, for example.

Let $\cover{\Sp}(W)$ be the metaplectic group: it is the unique topological central extension of the symplectic group $\Sp(W)$ by $\{\pm 1\}$ which does not split unless $W=0$ or $\rF \cong \C$.
According to Shale \cite{Sh} and Weil \cite{Wei}, there is a unitary representation $\widehat \omega$ of $\wt{\Sp}(W)\ltimes \rH(W)$ whose restriction $\widehat \omega|_{\rH(W)}$ to $\rH(W)$ is irreducible with central character $\psi$. We call this representation the oscillator representation (attached to the central character $\psi$).

We are given $(G,G')$, an arbitrary reductive dual pair in $\Sp(W)$. We consider the following general setting. Let $\wt{G}$ and $\wt{G'}$ be a pair of reductive groups together with continuous surjective group homomorphisms $\wt{G}\rightarrow G$ and $\wt{G'}\rightarrow G'$.  The
group $\wt{G}\times \wt{G'}$ acts on the Heisenberg group $\rH(W)$ as group automorphisms through its obvious action on $W$.
Using this action, we define the Jacobi group
\begin{equation}\label{Jacobi}
  J:=(\wt{G}\times \wt{G'})\ltimes \rH(W).
\end{equation}

Suppose that $J$ has a unitary representation whose restriction to $\rH(W)$ is irreducible with central character $\psi$. (For example, this is the case if we take $(\wt{G}, \wt{G'})$ to be the inverse image of $(G,G')$ in $\wt{\Sp}(W)$.)
By the Stone-von Neumann Theorem, all such representations, if they exist, are isomorphic to each other up to twisting by unitary characters (which are trivial on
$\rH(W)$). We fix one such $\widehat \omega$ and write $\omega$ for the space of smooth vectors of $\widehat \omega|_{\rH(W)}$, which is $J$-stable and is a smooth representation of $J$. We will refer to $\omega$ as a smooth oscillator representation.

The basic problem is to understand $\omega|_{\cover{G}\times \cover{G'}}$, and this is what the theory of local theta correspondence addresses. The main assertion of the theory is the Howe duality conjecture \cite[Conjecture (local duality)]{Ho79}. We shall give the precise statement for $\rF=\R$, established by Howe in \cite{Ho89}.

The results of \cite{Ho89} are stated in terms of infinitesimal equivalence classes of continuous irreducible admissible representations. With the benefit of the Casselman–Wallach theorem \cite[11.6.8]{Wa2}, it is rather natural to formulate the results in the category of Casselman–Wallach representations. We refer the reader to \cite[Chapter 11]{Wa2} for generalities on Casselman-Wallach representations.

For a reductive Lie group $E$, denote by $\Irr (E)$ the set of equivalence classes of irreducible Casselman-Wallach representations of $E$.  Note that $\Irr  (\wt{G}\times \wt{G'}) \simeq \Irr  (\wt{G}) \times \Irr  (\wt{G'})$. The identification associates to $\pi \in \Irr  (\wt{G})$ and $\pi' \in \Irr  (\wt{G'})$ the completed projective tensor product $\pi\widehat \otimes \pi '$ (see \cite{Tr}).

We fix a realization of the smooth oscillator representation $\omega$, on a space $\SY$. Denote by $\Irr (\cover{G}, \omega )$ the set of elements of
$\Irr (\cover{G})$ which are realized as quotients by $\omega (\cover{G})$-invariant closed subspaces of $\CY$. Namely,
\[\Irr (\cover{G}, \omega ):=\{\pi\in \Irr(\cover{G})\mid \Hom_{\cover{G}}(\omega,\pi)\neq 0\}.\]
Define $\Irr (\cover{G'}, \omega )$ and $\Irr  (\wt{G}\times \wt{G'}, \omega)$, likewise.

\begin{thm} [Howe, \cite{Ho89}] \label{HD1} The set $\Irr  (\wt{G}\times \wt{G'}, \omega)$ is the graph of a bijection between (all of) $\Irr (\cover{G}, \omega )$ and (all of) $\Irr (\cover{G'}, \omega )$. Moreover, an element $\pi\widehat \otimes \pi '$ of $\Irr  (\wt{G}\times \wt{G'}, \omega)$ occurs as a quotient of $\omega$ in a unique way.
\end{thm}

\begin{rem}
For a nonarchimedean local field, the Howe duality conjecture is established by Waldspurger (\cite{Wal}, residue characteristic not $2$), and completed by Minguez \cite{Mi}, Gan-Takeda \cite{GT} and Gan-Sun \cite{GS}, without any restriction on the residue characteristics. Thus the Howe duality conjecture is now a theorem in its full generality.
\end{rem}

It is convenient to introduce the following version of Theorem \ref{HD1}. We are given $\pi \in \Irr (\cover{G})$. If $\pi \in \Irr (\cover{G}, \omega)$, define $\SY_{\pi}$ to be
the intersection of (all) closed $\cover{G}$-invariant subspaces $\SY_1\subseteq \SY$ such that $\SY/\SY_1\simeq \pi$. We shall call $\SY/\SY_{\pi}$ the maximal $\pi$-isotypic quotient of $\SY$, or Howe's maximal quotient of $\pi $. We have
\begin{equation}\label{maxQ}
\SY/\SY_{\pi}= \pi \widehat \otimes \Theta(\pi),
\end{equation}
where $\Theta(\pi)$ is a representation of $\cover{G'}$.
The representation $\Theta(\pi)$ is often referred to as the
full theta lift (or the big theta lift) of $\pi$. If $\pi$ is not in $\Irr (\cover{G}, \omega)$, it is also convenient to define $\Theta(\pi)=0$.

Recall that for a reductive Lie group $E$, a representation of $E$ is called quasisimple if $\CU (\CE)^E$ act on the representation space as scalar operators. Here $\CU (\CE)$ denotes the universal enveloping algebra of the Lie algebra $\CE$ of $E$, and the superscript $E$ indicates the invariant space under the adjoint action $\text{Ad}E$.

\begin{thm} [Howe, \cite{Ho89}] \label{HD2} For any $\pi \in \Irr (\cover{G}, \omega )$, the representation $\Theta(\pi)$ is a quasisimple Casselman-Wallach representation of $\cover{G'}$. Furthermore, $\Theta(\pi)$ has a unique irreducible $\cover{G'}$ quotient $\pi '$ so that $\pi\widehat \otimes \pi ' \in \Irr  (\wt{G}\times \wt{G'}, \omega)$.
\end{thm}

We shall refer to Theorems \ref{HD1} and \ref{HD2} as Howe's Duality Theorem, and the correspondence of representations $\pi$ and $\pi '$  in Theorems \ref{HD1} and \ref{HD2} as local theta correspondence or Howe correspondence. In \cite{Ho89}, Howe proves an algebraic version of Theorems \ref{HD1} and \ref{HD2} and shows that Theorems \ref{HD1} and \ref{HD2} follow from this algebraic version. It remains an interesting problem to give a suitably distribution theoretic proof of the Howe's Duality Theorem, as envisaged in \cite[Section 11]{Ho79}.

To get a sense of the algebraic formulation of the Howe's Duality Theorem and what's involved in its proof, we shall consider the Harish-Chandra module of $\omega$. Write $\Sp =\Sp (W)$, and $\sp $ its Lie algebra. Fix a maximal compact subgroup $U\simeq \rU_N$ (the unitary group in $N$ variables), where $N=\frac{\dim W}{2}$. We then have the Cartan decomposition:
\begin{equation}\label{Cartan}
\sp =\u \oplus \q,
\end{equation}
where $\u$ is the Lie algebra of $U$. (We shall refer to $\q$ as the symmetric part of $\sp$, and the same terminology applies in other settings.) The $(\sp, \tilde U)$ module associated to $\omega$ is then realized in the space $\SP=\SP_N$ of polynomials on $\C^N$, known as the Fock model \cite{Ba, HoPre2}. In this model, we have a decomposition of the complex Lie algebra $\sp_{\C}$ (the complexfication of $\sp$):
\begin{equation}\label{Fock}
  \omega (\sp_{\C})=\sp^{(1,1)}\oplus \sp^{(2,0)}\oplus \sp^{(0,2)}, \quad \text{where}
\end{equation}
\[
  \begin{array}{rcl}
    \sp^{(1,1)}&=&\text{span}\{z_{i}\frac{\partial}{\partial z_{j}} +\frac{1}{2}\delta_{i,j} \}_{1\leq i,j\leq N},\\
    \sp^{(2,0)}&=& \text{span}\{z_{i}z_{j}\}_{1\leq i,j\leq N}, \\
    \sp^{(0,2)}&=& \text{span}\{\frac{\partial^2}{\partial z_{i}\partial z_{j}}\}_{1\leq i,j\leq N}.
 \end{array}
\]
The relation between decompositions \eqref{Cartan} and \eqref{Fock} is
\[\omega (\u_{\C})= \sp^{(1,1)}, \qquad
\omega (\q_{\C})= \sp^{(2,0)}\oplus \sp^{(0,2)}.\]

We may assume that $G$ and $G'$ are embedded in $\Sp$  in such a way that the Cartan decomposition of $\sp$ also induces Cartan decompositions of $\g$ and $\g'$. Naturally $\SP$ is a $(\g, \cover{K})$ module and a $(\g', \cover{K'})$ module. Note that the action of $\cover{U}$, and therefore $\cover{K}$ and $\cover{K'}$, preserves the degree of a polynomial in $\SP$. In this algebraic context, we will adopt analogous notations such as $\Irr(\g, \cover{K}, \omega)$ and $\Irr(\g', \cover{K'}, \omega)$ without further explanation.

When one of the members ($G$ or $G'$) is compact, in which case we say that the dual pair $(G,G')$ is compact, the duality theorem is known earlier (see \cite{KV, HoRem}) and amounts essentially to a version of Classical Invariant Theory \cite{Wey,KV,HoRem}. Note that the other member of the dual pair may or may not be compact. We will state the duality result for a compact dual pair, which has additional structure and which also serves as a bridge to the general case. Let us write a compact dual pair as $(K, G')$ (to signal that $G=K$ is compact.) We may assume that $K$ is contained in $U$, which we have fixed. The decomposition \eqref{Fock} induces a decomposition for $ \omega (\g'_{\C})$:
\[
 \omega (\g'_{\C})=\g'^{(1,1)}\oplus \g'^{(2,0)}\oplus \g'^{(0,2)}, \quad \text{ where}\quad \g'^{(i,j)}=\omega (\g'_{\C})\cap \sp^{(i,j)}.
\]


Define the space of $K$-harmonics in $\SP$:
\[\SH (K)= \{P\in \SP: \Delta (P)=0 \text{ for all }\Delta \in \g'^{(0,2)}\},
\]
which is a $\cover{K}\times \g'^{(1,1)}$-module. Consider the isotypic decomposition of $\SP$ as an $\cover{K}$-module:
\[\SP =\sum _{\sigma \in \Irr(\cover{K},\omega)}\SP_{\sigma}.\]

Generalizing theory of spherical harmonics, the following result may be considered as the main result of Classical Invariant Theory. This is explained in \cite{HoRem} and appears in \cite{Ho89}. See also \cite{KV}.

\begin{thm} [{\cite{Ho89,KV}}] \label{thm:compact} \begin{enumerate}
\item [(a).] The joint action of $\cover{K}\times \g' $ on $\SP_{\sigma}$ is irreducible for each $\sigma \in \Irr (\cover{K}, \omega)$.
\item [(b).] The space $\SH(K)_{\sigma}:=\SH(K)\cap \SP_{\sigma}$ consists of the space of polynomials of lowest degree in $\SP_{\sigma}$.
\item [(c).] One has $\SP_{\sigma}=\CU (\g'^{(2,0)})\cdot \SH(K)_{\sigma}$, where $\CU (\g'^{(2,0)})$ is the universal enveloping algebra of
$\g'^{(2,0)}$.
\item [(d).] The group $\cover{K}$ and the Lie algebra $\g'^{(1,1)}$ generate mutual commutants on $\SH(K)$. Equivalently, each $\SH(K)_{\sigma}$ is irreducible under the joint action of $\cover{K}\times \g'^{(1,1)}$, and if we write $\SH(K)_{\sigma}\simeq \sigma \otimes \tau$ for $\tau \in \Irr(\g'^{(1,1)},\omega)$, then $\sigma$ determines $\tau$ and vice versa, so that $\sigma \rightarrow \tau$ is an injection from $\Irr(\cover{K},\omega)$ into $\Irr(\g'^{(1,1)},\omega)$.
\end{enumerate}
\end{thm}

The above result for a compact dual pair is algebraic in nature and serves as a stepping stone to the duality theorem in the general setting. For the general setting, the key lies in some remarkable relationship between $(G, G')$ and three other dual pairs in $\Sp (W)$. The three pairs are $(K, M')$, $(K',M)$, and $(M^{(1,1)}, M'^{(1,1)})$, where $M'$ (resp. $M$) is the centralizer of $K$ (resp. $K'$), and $M^{(1,1)}$ (resp. $M'^{(1,1)}$) is a  maximal compact subgroup of $M$ (resp. $M'$). Note that all three dual pairs are compact dual pairs. For $(G,G')=(\rO_{p,q},\Sp_{2n}(\R))$, the relevant dual pairs are depicted in Figure \ref{fig:diamond.pair} in two diamonds: the pairs of Lie groups similarly placed in the two diamonds are dual pairs.

\begin{figure}[htbp]
\caption{{Diamond dual pairs}}
\label{fig:diamond.pair}

\hfil

\begin{picture}(400,100)


\put(131,80){\makebox(0,0){$ \rU_p\times \rU_q $}} \put(301,80){\makebox(0,0){$
\rU_n \times \rU_n $}}

\put(95,50){\makebox(0,0){$ \rO_p \times \rO_q $}} \put(165,50){\makebox(0,0){$
\rU_{p,q} $}} \put(255,50){\makebox(0,0){$ \Sp_{2n}(\R) \times \Sp_{2n}(\R) $}}
\put(336,50){\makebox(0,0){$ \rU_n $}}

\put(131,15){\makebox(0,0){$ \rO_{p,q} $}} \put(301,15){\makebox(0,0){$
\Sp_{2n}(\R)$}}


\put(101,  54){\line(1, 1){20}} \put(139, 76){\line(1,-1){20}}
\put(262, 54){\line(1, 1){20}} \put(310, 76){\line(1,-1){20}}

\put(101,  42){\line(1,-1){20}} \put(139,  22){\line(1, 1){20}}
\put(262,  42){\line(1,-1){20}} \put(310,  22){\line(1, 1){20}}

\end{picture}
\hfil
\end{figure}
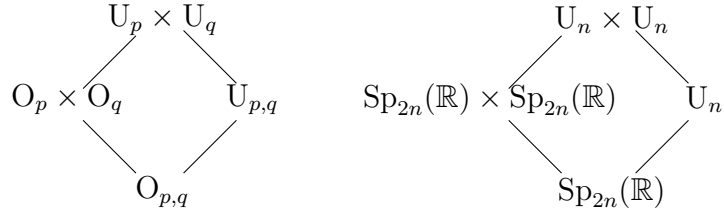

Consider the seesaw dual pairs $(G,G')$ and $(K,M')$ (i.e., we have inclusions $G\supset K$ and $G'\subset M'$), and the containment of $\g'_{\C}$ in $\m'_{\C}$:
\[\g'_{\C}=\kk'_{\C}\oplus \p'_{\C} \subset \m'_{\C}=\m'^{(1,1)}\oplus \m'^{(2,0)}\oplus \m'^{(0,2)}.\]
The symmetric part $\p'_{\C}$ of $\g'_{\C}$ is transversal in the symmetric part $\m'^{(2,0)}\oplus \m'^{(0,2)}$ of $\m'_{\C}$ in the following sense:
\begin{equation}\label{transversal} \m'^{(2,0)}\oplus \m'^{(0,2)}
= \p'_{\C}\oplus \m'^{(0,2)}= \m'^{(2,0)}\oplus \p'_{\C}. \end{equation}
Likewise for the seesaw dual pairs $(G,G')$ and $(M,K')$, we have the transversality of $\p_{\C}$ in $\m^{(2,0)}\oplus \m^{(0,2)}$.

Applying Theorem \ref{thm:compact} for the compact dual pair $(K, M')$ and using \eqref{transversal}, one immediately obtains the following

\begin{prop} [\cite{Ho89}]
The space $\SH(K)$ generates $\SP$ as a $\g'$-module, and similarly for $K'$. That is,
\[\SP=\CU (\g')\cdot \SH(K)= \CU (\g)\cdot \SH(K'),\]
where $\CU (\g')$ is the universal enveloping algebra of $\g'$. Thus, for any $\sigma \in \Irr(\cover{K}, \omega)$ or
$\sigma' \in \Irr(\cover{K'}, \omega)$, one has
\begin{equation}\label{eq:Psigma}\SP_{\sigma}=\CU (\g')\cdot \SH(K)_{\sigma}, \qquad
\SP_{\sigma'}= \CU (\g)\cdot \SH(K')_{\sigma}.\end{equation}
\end{prop}

Howe then proves a few facts about the structure of $\SP$ as a $(\g\oplus \g', \wt{K}\cdot \wt{K'})$ module, which are mainly encoded in the structure of the so-called space of joint harmonics.

\begin{defn}
The space of joint harmonics $\SH_{K,K'}$ is the space of polynomials which are both $K$-harmonics and $K'$-harmonics, namely
\[\SH_{K,K'}=\SH(K)\cap \SH(K').\]
\end{defn}

 For the proof of the next proposition, Howe uses the dual pair $(M^{(1,1)}, M'^{(1,1)})$ and the fact that both members are compact, in addition to the dual pairs $(K, M')$ and $(K',M)$.
 Define $\Irr (\cover{K}, \SH_{K,K'})$ and $\Irr (\cover{K'},\SH_{K,K'})$ in the obvious way.

\begin{prop} [\cite{Ho89}]\label{Jharmonics}
\begin{enumerate}
    \item [(a).] There is a bijective correspondence
    \begin{equation}\label{JointH} \sigma \in \Irr (\cover{K}, \SH_{K,K'}) \longleftrightarrow  \Irr (\cover{K'},\SH_{K,K'})\ni \sigma', \end{equation}
    defined by the condition that $\sigma \otimes \sigma' $ is a direct summand of $\SH_{K,K'}$.
\item [(b).] $\SH_{K,K'}$ generates $\SP$ as a representation of $\g\oplus \g'$, i.e., $\CU (\g\oplus \g')\cdot\SH_{K,K'}=\SP$.
\end{enumerate}
\end{prop}

Write $\SH_{\sigma, \sigma'}:=\SH(K)_{\sigma}\cap \SH(K')_{\sigma'}$. Fix a $\Pi \in \Irr (\g, \cover{K},\omega)$, and consider the Howe's maximal quotient $\SP/\SN_{\Pi}$ of $\Pi$, where $\SN_{\Pi}$ is the intersection of all $\g$ and $\cover{K}$-invariant subspaces $\SN\subseteq \SP$ such that $\SP/\SN$ is a $(\g,\cover{K})$ module isomorphic to $\Pi$. In $\SP/\SN_{\Pi}$, Howe picks a $\cover{K}$-type $\sigma$ of lowest degree, which he shows to have the property $\SH(K)_{\sigma}\cap \SH(K')\ne 0$, and therefore $\sigma$ corresponds to a $\cover{K'}$-type $\sigma'$ in $\SH(K')$, according to \eqref{JointH}. Furthermore $\sigma'$ is a $\cover{K'}$-type in $\SP/\SN_{\Pi}$, and no other representation of $\cover{K'}$ occurring in $\SH(K)_{\sigma}$ occurs in $\SP/\SN_{\Pi}$. Since $\Pi$ is irreducible, $\Pi$ is generated by its $\sigma$-isotypic component as a $\g$ module, and since $\SP_{\sigma}=\CU (\g')\cdot \SH(K)_{\sigma}$, it then follows that $\SP/\SN_{\Pi}$ is generated by $\SH_{\sigma, \sigma'}$ as a $\g \oplus \g'$ module.


From the proof of Proposition \ref{Jharmonics}, we also know there exits $\tau \otimes \tau' \in \Irr (M^{(1,1)} \times M'^{(1,1)})$ such that $\SH_{\sigma,\sigma'}=\SP_{\tau, \tau'} \cap \SH(K)\cap\SH(K')$, where $\SP_{\tau, \tau'}$ is the $\tau \otimes \tau'$ isotypic component, and it contains $\SH_{\sigma,\sigma'}$ with multiplicity one. Let $\SY^{\tau,\tau'}$ be the $\g\oplus \g'$ module generated by $\SP_{\tau,\tau'}$. Howe then defines a $(\g\oplus \g', \wt{K}\times \wt{K'})$ module
\[\SZ =(\SY^{\tau, \tau'}+\SY^{(d)})/\SY^{(d)},\]
where $\SY^{(d)}$ is the $(\g\oplus \g',\cover{K}\times \cover{K'})$ module generated by the space of polynomials of degree at most $d:=\deg (\sigma)-1$. It is clear that $\SZ$ contains the $\wt{K}\times \wt{K'}$ module $\SH_{\sigma,\sigma'}$ with multiplicity one. Since $d$ is the largest integer such that $\SY^{(d)}$ is contained in $\SN_{\Pi}$ \cite[Lemma 4.1]{Ho89}, we see that $\SP/\SN_{\Pi}$ is in fact a quotient of $\SZ$.

Write $\SZ_{\sigma}$ for the $\sigma$-component of $\SZ$, and similarly for $\SZ_{\sigma'}$. Then, similar to \eqref{eq:Psigma}, we have
\[\SZ_{\sigma}=\CU (\g')\cdot \SH_{\sigma, \sigma'}, \qquad \SZ_{\sigma'}=\CU (\g)\cdot \SH_{\sigma, \sigma'}.\]
Consider the simultaneous $(\sigma, \sigma')$-component:
\[\SZ_{\sigma, \sigma'}=(\SZ_{\sigma'})_{\sigma}=(\SZ_{\sigma})_{\sigma'}.\]
By a well-known general fact on the action of centralizer algebra on an isotypic component, we then have
\begin{equation}\label{Zss'}
\SZ_{\sigma, \sigma'}=\CU (\g)^{\cover{K}}\cdot \SH_{\sigma, \sigma'}= \CU (\g')^{\cover{K'}}\cdot \SH_{\sigma, \sigma'}.
\end{equation}
Similar to \eqref{maxQ}, write $\SP/\SN_{\Pi}=\Pi \otimes \Theta (\Pi)$, where $\Theta (\Pi)$ is a $(\g', \cover{K'})$ module. Since $\SP/\SN_{\Pi}$ is a quotient of $\SZ$, Howe then argues that the equality in \eqref{Zss'} will compel $\Theta (\Pi)$ to have a unique irreducible quotient! See \cite[Pages 546--547]{Ho89}.

To the author, the concluding argument of Howe amounts (essentially) to the fundamental insight of Harish-Chandra that $\CU(\g)^K$ action on any $K$-type determines an irreducible representation uniquely! See \cite[Theorem 2]{HC} or \cite[Theorem 4.9]{LeM}.

As a by-product, the proof of the algebraic version of the Howe's Duality Theorem reveals that the correspondence of $\pi$ and $\pi'$ in Theorem \ref{HD2} comes equipped with a correspondence of $\cover{K}$ and $\cover{K'}$-types in the space of joint harmonics.

\begin{thm} [\cite{Ho89}]
Suppose $\pi \in \Irr(\g, \cover{K}, \omega)$, and $\sigma$ is a $\cover{K}$-type of $\pi$ of lowest degree. Then $\sigma \in \Irr (\cover{K}, \SH_{K,K'})$. Let $\sigma' \in \Irr (\cover{K'}, \SH_{K,K'})$ be according to \eqref{JointH}. Then $\sigma'$ is a $\cover{K'}$-type of $\pi'$ of lowest degree.
\end{thm}

For the explicit description of the correspondence in the space of joint harmonics, see \cite[Proposition 7.5]{Ad}.

\vsp

We are back in the setting of Howe's Duality Theorem. We shall write $\pi ':=\theta (\pi)$, called the (local) theta lift of $\pi$. Clearly this depends on the groups $\cover{G}$ and $\cover{G'}$. To emphasize this dependence, we shall also write $\Theta(\pi)$ as $\Theta_{\widetilde G}^{\widetilde G'}(\pi)$, and $\theta(\pi)$ as $\theta_{\widetilde G}^{\widetilde G'}(\pi)$.

\begin{rem} It is sometimes convenient to define $\Theta_{\widetilde G}^{\widetilde G'}(\pi)$ for any (genuine) Casselman-Wallach representation $\pi$, not necessarily irreducible. This may be done as follows. Denote by $\pi^\vee$ the contragredient representation of $\pi$. Note that $\pi ^\vee$ is the unique Casselman–Wallach representation whose underlying Harish-Chandra module is the contragredient of that of $\pi$. Then we define
\[
  \Theta_{\widetilde G}^{\widetilde G'}(\pi):=(\omega\widehat \otimes \pi^\vee)_{G},
\]
where the subscript $G$ indicates the (Hausdorff) $G$-coinvariant space (namely, $(\omega\widehat \otimes \pi^\vee)_{G}$ is the maximal quotient (by $G$-invariant closed subspaces) of $\omega\widehat \otimes \pi^\vee$ on which $G$ acts trivially). Then $\Theta_{\widetilde G}^{\widetilde G'}(\pi)$ is a Casselman-Wallach representation of $\widetilde G'$. We will also denote \[
  \check \Theta_{\widetilde G}^{\widetilde G'}(\pi):=(\omega\widehat \otimes \pi)_{G}.
\]
\end{rem}

\subsection{Conservation relations} \label{subsec:CR}
In this subsection, we discuss another basic assertion in the theory of local theta correspondence, Kudla-Rallis conservation relations conjecture \cite{KR2}. This is established in its full generality in \cite{SZJou}. For simplicity,  we only discuss the orthogonal-symplectic dual pairs, and follow the exposition in \cite{SZCon}.

Denote by $\curlyvee$ the set of isomorphism classes of finite dimensional non-degenerate quadratic spaces over $\rF$.
Also denote by $\curlywedge$ the set of isomorphism classes of
finite dimensional symplectic spaces over $\rF$. By abuse of notation, we do not distinguish an element of $\curlyvee$ with a quadratic space which represents it. Likewise for an element of $\curlywedge$
and a symplectic space which represents it. Throughout this section, $V$ always refers to a non-degenerate quadratic space and $V'$ a symplectic space.

We shall fix a parity $\epsilon\in \Z/2\Z$ and consider quadratic spaces $V$ such that $\dim V$ has parity $\epsilon$.
Write
\begin{equation*}\label{meta}
   1\rightarrow \{\pm 1\}\rightarrow
   {\Sp}_\epsilon(V')\rightarrow
   \Sp(V')\rightarrow 1
\end{equation*}
for the unique topological central extension of the symplectic group
$\Sp(V')$ by $\{\pm 1\}$ such that it splits if $\epsilon$ is even,
or $V'=0$, or $\rF$ is isomorphic to $\C$, and it does not split
otherwise.

As discussed in Section \ref{dual pair}, $(\rO(V), \Sp(V'))$ is an irreducible reductive dual pair in $\Sp(W)$, where $W:=\Hom_{\rF}(V,V')$. Recall from Section \ref{duality} the Jacobi group:
\begin{equation*}
\label{Jacobi}
  \rJ_{V,V'}:=(\rO(V)\times \Sp_{\epsilon}(V')) \ltimes \rH (W).
\end{equation*}

We introduce one general terminology: throughout this article, a smooth representation of a reductive group over
$\rF$ will mean a smooth representation in the usual sense when $\rF$
non-Archimedean, namely it is locally constant, and a smooth Fr\'{e}chet representation of moderate growth when $\rF$ is
Archimedean.  This notion of smooth representations can be extended to the setting of Jacobi groups. See \cite[Section 2]{Su1}.

We will fix a smooth oscillator representation
$\omega_{V,V'}$ of the Jacobi group $\rJ_{V,V'}$. Up to
isomorphism, $\omega_{V,V'}$ is the unique smooth representation with the following
properties: \cite[Section 1]{HoPre2}, \cite[Chapter 2]{MVW}
\begin{enumerate}
  \item[(a).] as a representation of $\rH (W)$, it is irreducible with
central character $\psi$;
   \item[(b).] for every Lagrangian subspace $L$ of $V'$,  the unique (up to scalar) nonzero
(continuous in the archimedean case) linear functional on $\omega_{V,V'}$ which is invariant under $\Hom_{\rF}(V,L) \subset \rH (W)$ is $\rO(V)$-invariant;
 \item[(c).]it is genuine as a representation of $\Sp_{\epsilon}(V')$,
namely, the central element $-1\in {\Sp}_\epsilon(V')$ acts through
the scalar multiplication by $-1\in \C$.
\end{enumerate}
Note that condition (b) means that we have normalized the oscillator representation so that the action of $O(V)$ is induced by its obvious linear action on $\Hom_{\rF}(V,L)$ in the Schr{\"o}dinger model \cite{Ge79, HoPre2} associated to the Lagrangian subspace $\Hom_{\rF}(V,L)$ of $W$.

Denote by $\Irr(\rO(V))$ the set of isomorphism classes of irreducible
admissible smooth representations of $\rO(V)$,
and by $\Irr({\Sp}_\epsilon(V'))$ the isomorphism
classes of irreducible admissible genuine smooth representations of
${\Sp}_\epsilon(V')$.
In this
article, $\pi$ denotes a representation in $\Irr(\rO(V))$ and $\pi '$
denotes a representation in $\Irr(\Sp_\epsilon(V'))$.
We are
interested in the occurrence of $\pi$ and $\pi '$ in the local theta
correspondence, namely as a quotient of $\omega_{V,V'}$.

We recall two well-known facts on local theta correspondence: Kudla's persistence principle \cite[Section III.4]{KuNot}
and non-vanishing of theta liftings in stable range (\cite[Section III.4]{KuNot} and \cite[Theorem 1]{PP}).
See also \cite[Theorems 14.3 and 14.4]{GKT}.

We first consider an orthogonal group $\rO(V)$. Let $\pi \in \Irr(\rO(V))$. Kudla's persistence
principle says that if $V_1',V_2'\in \curlywedge$ and $\dim V_1'\leq
\dim V_2'$, then
\[
 \Hom_{\rO(V)}(\omega_{V,V_1'},
\pi)\neq 0\quad\textrm{implies} \quad \Hom_{\rO(V)}(\omega_{V,V_2'},
\pi)\neq 0.
\]
Non-vanishing of theta liftings in stable range says that if $\dim V\leq \frac{1}{2}\dim V'$, then
\[
  \Hom_{\rO(V)}(\omega_{V,V'},
\pi)\neq 0.
\]

Define the first occurrence index
\begin{equation}
\label{Orth-index}
  \operatorname n(\pi):=\min\{\frac{1}{2}\dim V'\mid V'\in \curlywedge, \Hom_{\rO(V)}(\omega_{V,V'},
\pi)\neq 0\}.
\end{equation}

The conservation relation for orthogonal groups is the following

\begin{thm} [\cite{SZJou}] \label{conso}
Let $\pi\in \Irr(\rO(V))$. Then
\[
  \operatorname n(\pi)+ \operatorname n(\pi\otimes \sgn)=\dim V,
\]
where ``$\sgn$" stands for the sign character of $\rO(V)$.
\end{thm}

\begin{rem} In the non-Archimedean case and for $\pi$ irreducible cuspidal,
Theorem \ref{conso} was proved in \cite[Theorem 2]{Mi2}. For $\rF=\C$, it was proved in \cite{AB}, from the explicit description of the Howe correspondence proved in the paper.
\end{rem}

Now we consider the case of symplectic/metaplectic groups. Let $\pi' \in \Irr(\Sp_\epsilon(V'))$. For any quadratic space $U$, denote by $U^-$ the
space $U$ equipped with the form scaled by $-1$.
Two quadratic spaces $V_1, V_2$ are
said to be in the same Witt tower if the quadratic space $V_1\oplus
V_2^-$ splits. This defines an equivalence relation on $\curlyvee$. An equivalence class of this relation is called
an (orthogonal) Witt tower.  Denote by $\ST$ the set of Witt towers. It is clear that any Witt tower $\bft \in \ST$ is of the form
\[\bft =\{V_{(0)}\oplus V_{r,r}\}_{r\in \Z_{\geq 0}},\]
where $V_{(0)}$ is anisotropic, and $V_{r,r}$ is the split quadratic space of dimension $2r$. The quadratic space $V_0\oplus V_{r,r}$ has the Witt index $r$, which by definition is the dimension of a maximal totally isotropic subspace.

Kudla's persistence principle says that for any given Witt tower $\bft $, if $V_1,V_2\in
\bft$ and $\dim V_1\leq \dim V_2$, then
\[
 \Hom_{\Sp_\epsilon(V')}(\omega_{V_1,V'},
\pi ')\neq 0\quad\textrm{implies} \quad
\Hom_{\Sp_\epsilon(V')}(\omega_{V_2,V'}, \pi')\neq 0.
\]
Non-vanishing of stable range theta liftings says that if $V\in \bft$
and the Witt index of $V$ is at least $\dim V'$, then
\[
  \Hom_{\Sp_\epsilon(V')}(\omega_{V,V'},
\pi')\neq 0.
\]

Define the first occurrence index
\begin{equation}
\label{Symp-index}
   \operatorname m_{\bft}(\pi'):=\min\{\dim V\mid V\in \bft,\,\Hom_{\Sp_\epsilon(V')}(\omega_{V,V'},
\pi')\neq 0\}.
\end{equation}

We now fix a quadratic (i.e. order $2$) character $\chi:
\rF^\times \rightarrow \{\pm 1\}$.
Denote by $\curlyvee_{\epsilon, \chi} \subset \curlyvee $ the subset consisting of those quadratic spaces $V$ such that
\begin{itemize}
  \item $m:= \dim V$ has parity $\epsilon$, and
  \item the discriminant character $\chi_V$ of $V$ equals $\chi$.
\end{itemize}
Recall that the discriminant character $\chi_V$ is given by
\[
   \chi_V(x):=\left(x, \,(-1)^{\frac{m(m-1)}{2}}\det (V)\right)_2,\quad x\in \rF^\times,
\]
where $\det(V)$ is the determinant of the matrix of the symmetric bilinear form on $V$, and
$(\,,\,)_2$ is the quadratic Hilbert symbol for $\rF$.

It is clear that the quadratic character $\chi_V$ (and obviously the parity of $m$) is the same for all spaces V in a Witt tower.
Denote by $\ST_{\epsilon,\chi}$
the set of Witt towers in $\curlyvee_{\epsilon,\chi}$. By the
classification of quadratic spaces over a local field (see \cite{OM}), we know that
\begin{equation}
\label{towernumber}
  \sharp (\ST_{\epsilon,\chi})=\left\{
                                     \begin{array}{ll}
                                       2, & \hbox{if $\rF$ is non-Archimedean;} \\
                                       1, & \hbox{if $\rF$ is isomorphic to $\C$;} \\
                                       \infty, & \hbox{if $\rF=\R$.}
                                     \end{array}
                                   \right.
\end{equation}

{\bf Explicit description of Witt towers}: \cite{OM,KR2}. Assume that we are in the non-Archimedean case, so we have two Witt towers in
$\ST_{\epsilon,\chi}$ (to be labeled as $\bft^+$ and $\bft^-$, which carry different meanings in different cases):
\[\ST_{\epsilon,\chi}=\{\bft^+,\bft^-\}, \]
whose description is given below. For $\epsilon =0$ and $\chi = 1$, we have the split tower $\bft^+$ and the quaternionic tower $\bft^-$:
\[\bft ^+=\{V_{r,r}\}_{r\in \Z_{\geq 0}}, \quad \bft^-=\{B\oplus V_{r,r}\}_{r\in \Z_{\geq 0}},\]
where $B$ is the division quaternion algebra over F with the norm form.
For $\epsilon =0$ and $\chi \ne 1$, the two Witt towers are
\[\bft^{\pm}=\{V^{\pm}_2\oplus V_{r,r}\}_{r\in \Z_{\geq 0}},
\]
where $V^{\pm}_2$ represent the two (non-isomorphic) anisotropic quadratic spaces of dimension $2$. For $\epsilon =1$, the two Witt towers are
\[\bft ^+=\{V^+_1\oplus V_{r,r}\}_{r\in \Z_{\geq 0}}, \quad \bft^-=\{V^-_3\oplus V_{r,r}\}_{r\in \Z_{\geq 0}},\]
where $V^+_1$ (resp. $V^-_3$) represents an anisotropic quadratic space of dimension $1$ (resp. $3$).

The conservation relation for non-Archimedean symplectic groups is
the following

\begin{thm} [\cite{SZJou}] \label{padicsp}
Assume that $\rF$ is non-Archimedean and let $\pi' \in \Irr(\Sp_\epsilon(V'))$. Then
\[
  \operatorname m_{\bft ^+}(\pi ')+ \operatorname m_{\bft ^-}(\pi ')
  =4n+4, \qquad (2n=\dim V').
\]
\end{thm}

\begin{rem} For $\pi '$ irreducible cuspidal, Theorem \ref{padicsp} was proved in \cite[Corollary 3]{KR2}.
\end{rem}

{\bf Witt towers for $\rF=\R$}. A non-degenerate quadratic space $V$ of dimension $m$ is determined by its signature $(p,q)$, where $p+q=m$. We denote it by $\R^{p,q}$, for brevity. The set $\ST $ of Witt towers is then
\[\ST =\{\bft^{(k)} \mid k \in \Z \},\qquad \text{where } \bft^{(k)}=\{\R^{p,q} \mid p-q=k\}.\]

For the quadratic space $V=\R^{p,q}$, the quadratic character $\chi _V$ equals
\[
   \chi_{\alpha}(x):=\left(x, \,(-1)^{\frac{\alpha (\alpha -1)}{2}}\right)_2,\quad x\in \R^\times,
\]
where
\[\alpha :\equiv p-q \, (\text{mod } 4).\]
Note that $\alpha \equiv \epsilon$ (mod $2$). For $\alpha \in \Z/4\Z$, we thus define $\curlyvee_{\epsilon,\alpha}$ to consist of quadratic spaces $V$ such that $\dim V$ is of parity $\epsilon$ and $\chi _V=\chi _{\alpha}$, and $\ST_{\epsilon,\alpha}$ the set of Witt towers in $\curlyvee_{\epsilon,\alpha}$. Explicitly we have
\begin{equation}\label{eq:RealWitt}
\curlyvee_{\epsilon,\alpha}= \{\R^{p,q} \mid p-q \equiv \alpha \,(\text{mod } 4)\}, \quad
\ST_{\epsilon,\alpha}=\{\bft^{(k)} \mid k \in \Z, \, k\equiv \alpha \,(\text{mod } 4)\}.
\end{equation}

We introduce one more terminology. Observe that if $\bft_1, \bft_2\in \ST_{\epsilon,\chi}$ are two different Witt towers
and $V_i\in \bft_i$ ($i=1,2$), then $V_1\oplus V_2^-$ has even
dimension, trivial discriminant
character, and does not split. Therefore we must have
\begin{equation}\label{srank}
  \textrm{the Witt index of $(V_1\oplus V_2^-)$}\leq \frac{\dim V_1+\dim
V_2-4}{2}.
\end{equation}
We say that $\bft_1$ and $\bft_2$ are adjacent if the equality holds in
\eqref{srank}. (This notion does not dependent on the choice of $V_i$ in $\bft_i$.)
When $\rF$ is non-Archimedean, the two Witt towers in
$\ST_{\epsilon,\chi}$ are adjacent. When $\rF=\R$, two Witt
towers $\bft^{(k)}$ and $\bft^{(l)}$ in $\ST_{\epsilon,\alpha}$ are adjacent if and only if $|k-l|=4$.

The conservation relation for real symplectic groups is the following

\begin{thm} [\cite{SZJou}] \label{realsp}
Assume that $\rF=\R$, and let $\pi'\in \Irr(\Sp_\epsilon(V'))$. For $\alpha \in \Z/4\Z$ with $\alpha = \epsilon$ (mod $2$), we have
\begin{enumerate}
\item[(a).]
\[
 \min\{\operatorname m_{\bft_1}(\pi')+ \operatorname m_{\bft_2}(\pi')\mid \bft_1, \bft_2\in \ST_{\epsilon,\alpha }, \, \bft_1\neq
\bft_2\}=4n+4, \qquad (2n=\dim V').
\]

\item[(b).] Any $\bft_1, \bft_2\in \ST_{\epsilon,\alpha }$, $\bft_1\neq
\bft_2$ such that
\[\operatorname m_{\bft_1}(\pi')+ \operatorname m_{\bft_2}(\pi')=4n+4\]
are adjacent.
\end{enumerate}
\end{thm}

\begin{rem} For complex symplectic groups, there are only two Witt towers ${\bft}_{\epsilon}$ ($\epsilon =0,1$ is the parity of the dimension), and all irreducible representations have ``early'' occurrence in both of them.  More precisely, if $\rF=\C$, then for any $\pi'\in \Irr(\Sp_\epsilon(V'))$, one has
\[
  \operatorname m_{\bft_{\epsilon}}(\pi')\leq \left\{
                                   \begin{array}{ll}
                                     \dim V', & \hbox{if $\epsilon=0$;} \\
                                     \dim V'+1, & \hbox{if $\epsilon=1$.}
                                   \end{array}
                                 \right.
\]
We may view this as a replacement of the conservation relation for complex symplectic groups. For a discussion of relevant issues for all classical groups, see Section 7 of \cite{SZJou}.
\end{rem}

\section{The invariant-theoretic nature of local theta correspondence}

The invariant-theoretic (and distribution-theoretic) nature of the theory of local theta correspondence is emphasized by Howe in the beginning \cite{Ho79} as well as in his fundamental paper \cite{Ho89} establishing the duality theorem (for the Archimedean case).

In this section, we will illustrate this viewpoint via the proof of conservation relations \cite{SZJou}. We first discuss
the so-called doubling method (in the context of local theta correspondence), in which invariant theoretical considerations play a principal role. Then we discuss the two complementary aspects of the conservation relations, non-occurrence (or vanishing) and occurrence (or non-vanishing), providing the necessary inequalities leading to an equality.

\subsection{The doubling method and the invariant distribution theorem}

A very useful idea in the theory of local theta correspondence, which appears in \cite{Ho79} and \cite {Ra1}, is to double the variables. Suppose that we are given a reductive dual pair $(G,G') \subset \Sp(W)$, with a smooth oscillator representation $\omega$. We may consider $2W:=W\oplus W^-$, referred to as the double of $W$. There is a group $2G'$ (which contains $G'\times G'$) such that $(G_{\Delta}, 2G')$ forms a reductive dual pair in $\Sp(2W)$, with a suitably compatible smooth oscillator representation $\Omega$. Here $G_{\Delta}$ is an isomorphic copy of $G$. A key role in the doubling method is played by $\Omega^{2G'}_{G_{\Delta}}(\one)$ (or $\Omega (\one)$ in short), the maximal quotient of $\Omega $ on which $G_{\Delta}$ acts by the trivial representation $\one$. (We shall also refer to $\Omega (\one)$ as the Rallis quotient of $G_{\Delta}$.) This is a module for $\cover{2G'}$, and the relevant problem is to analyze $\Omega (\one)$ as a $\cover{G'}\times \cover{G'}$-module, and in particular to understand its irreducible quotients. We refer the reader to \cite[Section 11]{Ho79} for a more detailed discussion.

We will consider the orthogonal-symplectic dual pair:
\[(G,G')=(\rO(V), \Sp(V'))\subset \Sp(W), \quad W=\Hom_{\rF}(V,V').\]

Set $\VV' :=V'\oplus V'^-$, the double of $V'$. We have the dual pair:
\[(G,2G')=(\rO(V), \Sp(\VV'))\subset \Sp(\WW), \quad \WW= \Hom_{\rF}(V,\VV').\]

As in Section \ref{subsec:CR}, we fix a parity $\epsilon$ (of the dimension of the quadratic space $V$). Note that there is a unique continuous homomorphism
\[
   \Sp_\epsilon(V')\times \Sp_\epsilon(V'^-)\rightarrow
\Sp_\epsilon(\VV')
\]
which makes the diagrams in
\[
  \xymatrix{
   1 \ar[r] & \{\pm 1\}\times \{\pm 1\}\ar[r] \ar[d] &\Sp_\epsilon(V')\times \Sp_\epsilon(V'^-)\ar@{->}[r] \ar[d] &\Sp(V')\times
    \Sp(V'^-) \ar[r] \ar[d]  & 1 \\
   1 \ar[r] & \{\pm 1\}\ar@{->}[r] &\Sp_\epsilon(\VV')\ar@{->}[r] &\Sp(\VV') \ar[r] & 1
   }
\]
commutative, where the first vertical arrow is the multiplication
map. Therefore a representation of $\Sp_\epsilon(\VV')$ may be viewed as a representation of $\Sp_\epsilon(V')\times \Sp_\epsilon(V'^-)$
through the pull-back.  We identify $\Sp_\epsilon(V'^-)$ with $\Sp_\epsilon(V')$ in the obvious way. Then we have
\[\omega _{V,\VV'}|_{\Sp_\epsilon(V')\times \Sp_\epsilon(V'^-)}= \omega_{V,V'} \widehat{\otimes} \omega_{V,V'}^{\vee}.
\]
Here and as before, $\widehat \otimes$ stands for the completed
projective tensor product if $\rF$ is Archimedean, and the algebraic
tensor product if $\rF$ is non-Archimedean. In addition $\omega_{V,V'}^{\vee}$ denotes the (smooth) congragredient of $\omega_{V,V'}$.

Let $\Omega^{\VV'}_V(\one)$ be the Rallis quotient of $\rO(V)$, namely the maximal quotient of $\Omega =\omega_{V, \VV'}$ on which $\rO(V)$ acts trivially. This is a representation of $\Sp_\epsilon(\VV')$. We have the following criterion for non-vanishing of theta lifting (due to Kudla).
See \cite[Proposition 17.9 and Remark 17.11]{GKT}, which is stated for the non-Archimedean case. The same proof works for the Archimedean case, due to the fact that MVW-involutions exist for classical groups and metaplectic groups in the Archimedean case \cite{MVW, Pr1, Su2, LST}.

\begin{lem}
\label{nonvanish} For $\pi '\in \Irr(\Sp_\epsilon(V'))$, we have
\[
  \Hom_{\Sp_\epsilon(V')}(\omega_{V,V'},
\pi')\neq 0
\]
if and only if
\[
 \Hom_{\Sp_\epsilon(V')\times \Sp_\epsilon(V'^-)}(\Omega^{\VV'}_V(\one),
\pi '\widehat{\otimes} (\pi ')^\vee)\neq 0.
\]
\end{lem}

We examine the fine structure of $\Omega ^{V'}_V(\one)$ (defined likewise), in the case of real reductive dual pair:
\[(\rG(V), \rG(V'))=(\rO_{p,q}, \Sp_{2n}(\R))\subseteq \Sp_{2N} (\R), \quad N=(p+q)n.\]
(Here $p,q,n$ are arbitrary.)

Let
\begin{align*}\label{HoweQ}
\Omega _{p,q}^n(\one)= \, &\text{the Rallis quotient of $\rO_{p,q}$}.\notag
\end{align*}

We will temporarily shift notation and denote $G={\widetilde{\mathrm{Sp}}_{2n}(\R)}$.
We shall identify
${\widetilde{\mathrm{Sp}}_{2n}(\mathbb R})$ as a set with
\[{\mathrm{Sp}_{2n}({\mathbb R})}\times{\mathbb Z}_2=\{(g,\varepsilon):\ g\in{\mathrm{Sp}_{2n}(\mathbb R}),\ \varepsilon=\pm 1\}.\]
For $a\in \mathrm{GL}_n(\mathbb R)$ and $b\in \rM_{n,n}({\mathbb R})$ such that $b=b^t$, we let
\begin{equation*}
m_a=\left(\begin{array}{cc}
a&0\\
0&(a^{-1})^t
\end{array}\right), \qquad u_b=\left(\begin{array}{cc} I_n&b\\
0&I_n
\end{array}\right).
\end{equation*}

Let
\[M=\{(m_a,\varepsilon):\ a\in \mathrm{GL}_n(\mathbb R),\ \varepsilon=\pm 1\}\]
and
\[U=\{(u_b,1):\ b\in \rM_{n,n}({\mathbb R}),\ b=b^t\}.\]
Then $P=MU$ is a maximal parabolic subgroup of $G$, called the Siegel parabolic subgroup.

\vspace{0.2in} Let $\chi:M\longrightarrow{\mathbb C}^\times$ be given by
\[\chi(m_a,\varepsilon)=\varepsilon\cdot\left\{\begin{array}{ll}
1&\mbox{if}\ \det a>0,\\
i&\mbox{if}\ \det a<0.
\end{array}\right.\]
This is a character of $M$ and it is of order $4$. For
$\alpha=0,1,2,3$ and $s\in {\mathbb C}$, let $\chi^\alpha_s$ be the
character of $P$ given by
\begin{equation}
\chi^\alpha_s[(m_a,\varepsilon)(n_b,1)]=|\det a|^s\chi(m_a,\varepsilon)^\alpha.
\end{equation}
Let $I_n^\alpha(s)$ be the normalized induced representation:
\[I_n^\alpha(s)=\mbox{Ind}_P^G(\chi^\alpha_s) .\]
The representation space of ${I_n^\alpha(s)}$ is
\[\{f\in C^\infty(G):\ f(pg)=\delta^{\frac{1}{2}}(p)
\chi^\alpha_s(p)f(g),\ \forall g\in G, \, p\in P\},\] and $G$ acts
by right translation:
\[g\cdot f(h)=f(hg),\ \ \ \ (g,h\in G).\]
Here $\delta$ denotes the modular function of $P$, and is given by
\[\delta[(m_a,\varepsilon)(n_b,1)]=|\det a|^{2\rho_n}, \qquad \rho_n=\frac{n+1}{2}.\]
When $\alpha=0$ or $2$, the representation ${I_n^\alpha(\sigma)}$ descends to a representation of the
linear group ${\mathrm{Sp}_{2n}(\mathbb R)}$.

\vspace{0.2in}
Recall \cite{Ge79,HoPre2} that the smooth oscillator representation $\omega $ is realized on $\SMa$, the space of Schwartz functions on $\Ma $. This is known as the Schr{\"o}dinger model.

Define the map $\psi _{p,q}: \SMa \mapsto C^\infty(G)$ by
\[\psi _{p,q}(f)(g)=(\omega (g)f)(0), \ \ f\in \SMa, \ g\in G.\]
From the well-known description of $\omega$ in the Schr{\"o}dinger model \cite{Ge79,HoPre2}, we see that
\[\psi _{p,q}: \SMa \longmapsto I_n^\alpha(s_0),\]
where
\begin{equation}
\label{br}
s_0 = \frac{p+q}{2}-\rho_n, \ \ \text{and} \ \ \alpha \equiv p-q \ (\mathrm{mod}\  4).
\end{equation}

\begin{thm} [Kudla-Rallis \cite{KR1}] \label{embedding} The map $\psi _{p,q}$ induces a topological embedding with closed image:
\[\Omega _{p,q}^n(\one)\hookrightarrow I_n^\alpha(s_0).\]
\end{thm}

The continuous dual of $\SMa$, called the space of tempered distributions on $\Ma$, is denoted by $\TMa$. We denote the dualized action of $\cover{\Sp}_{2N}(\R)$ ($N=(p+q)n$) on $\TMa$ by $\omega ^*$. Note that the continuous dual of $\Omega _{p,q}^n(\one)$ is identified with $\TMa^{\rO_{p,q}}$, the space of $\rO(p,q)$-invariant tempered distributions.
Theorem \ref{embedding} implies (and is essentially equivalent to) the following result on invariant tempered distributions.

\begin{thm} [Kudla-Rallis \cite{KR1}] \label{distribution}
\[\TMa^{\rO_{p,q}} =\overline {\Span \{\omega ^*(g)\delta | \, g\in \widetilde{\mathrm{Sp}}_{2n}(\R)\}},
\]
where $\delta$ is the Dirac distribution at the origin of $\Ma$, and for a vector subspace $D$ of $\TMa$, $\overline {D}$ denotes closure of $D$ in the standard Frechet topology of $\TMa$.
\end{thm}

\begin{rem} Analogues of Theorems \ref{embedding} and \ref{distribution} for other real reductive dual pairs were proved by the author in \cite{Zh1}. Note that the analogous result in the non-Archimedean case was proved earlier by Rallis \cite[Theorem II.1.1]{Ra1}.
\end{rem}

\subsection{Non-occurrence} We examine the orthogonal group $\rO(V)$ first. The key fact is the following

\begin{lem} [\cite{Ra1,Ra2, Pr2}] \label{sign} The first occurrence index of the sign character is
\[
\operatorname n(\sgn) =\dim V.
\]
\end{lem}

For the real reductive dual pair $(G,G')=(\rO_{p,q},\Sp_{2n}(\R))$, this means that the sign character participates in Howe correspondence if and only if $p+q\leq n$, i.e., in the stable range with $\rO_{p,q}$ the smaller member. Lemma \ref{sign} is equivalent to
\[\TMa^{\rO_{p,q},\sgn }\ne 0 \text{ if and only if } p+q\leq n.
\]
(Here the superscript $\rO_{p,q},\sgn$ indicates the subspace on which $\rO_{p,q}$ acts by the sign character.)
This can be seen as follows. Recall that Howe correspondence comes equipped with a correspondence of $\cover{K}$, $\cover{K'}$-types in the space of joint harmonics. See Theorem \ref{JointH} for the precise statement, and \cite[Proposition 7.5]{Ad} for the explicit correspondence. It is then easy to see that for $\sgn$ to occur in Howe correspondence for the dual pair $(\rO_{p,q},\Sp_{2n}(\R))$, one must have $p+q\leq n$. Conversely it is also clear that
\[\TMa^{\rO_{p,q},\sgn }\ne 0 \text{ when } p+q=n.\]
(The determinant function of $\rM_{n,n}(\R)$ is obviously one such element.)

{\bf Consequence of Lemma \eqref{sign}}. Let $V_1'$ (respt. $V_2'$) be the symplectic space of dimension $2\rn(\pi)$ (resp. $2\rn(\pi \otimes \sgn)$). We have
\[\Hom _{\rO(V)}(\omega _{V,V_1'}, \pi )\ne 0, \text{ and }\Hom _{\rO(V)}(\omega _{V,V_2'}, \pi \otimes \sgn)\ne 0.\]

Note that $\pi^{\vee} \simeq \pi$, by \cite[pp. 91-92]{MVW}. Together, this implies that
\[\Hom _{\rO(V)}(\omega _{V,V_1'\oplus V_2'}, \sgn )\ne 0.\]
By Lemma \eqref{sign},  we conclude that $\dim V \leq \frac{1}{2}(\dim V_1'+\dim V_2')$,
proving one direction of the conservation relations:
\begin{equation}\label{geqor}
  \operatorname n(\pi)+\operatorname n(\pi\otimes \sgn)\geq
\dim V.
\end{equation}

\vsp

We now examine the symplectic/metaplectic group $\Sp_\epsilon(V')$. Recall that a quadratic space $V$ is called quasi-split if
\begin{equation}\label{quasi}
  \textrm{the Witt index of $V$}\geq \frac{\dim V-2}{2}.
\end{equation}
The key fact is the following

\begin{lem} [\cite{KR2,Lo}] \label{stable}
Assume that $\epsilon$ (the parity of $\dim V$) is even. If $V$ is not quasi-split, then
\[
 \Hom_{\Sp_\epsilon(V')}(\omega_{V,V'},\C)\neq 0
\]
implies that $V$ has Witt index $\geq 2n$, in particular $\dim V\geq
4n+4$. Here $\C$ stands for the unique one-dimensional genuine
representation of $\Sp_\epsilon(V')$, and $\dim V'=2n$.
\end{lem}

For the real reductive dual pair $(G,G')=(\rO_{p,q},\Sp_{2n}(\R))$, with $p+q$ even, this means that unless $|p-q|=0$ or $2$, the trivial character of $\Sp_{2n}(\R)$ participates in Howe correspondence only when $(\rO_{p,q},\Sp_{2n}(\R))$ is in the stable range with
$\Sp_{2n}(\R)$ the small member. As in the orthogonal group case, this fact also follows immediately from the explicit correspondence of $\cover{K}$ and $\cover{K'}$ types in the space of joint harmonics (\cite[Proposition 7.5]{Ad}).

\begin{rem} For the quasi-split Witt tower $\bft^{(2)}=\{\R^{p,q}\mid p-q=2 \}$, the first occurrence of the trivial representation $\one \in \Irr{(\Sp_{2n}(\R)})$ is with the dual pair $(\rO_{n+2,n}, \Sp_{2n}(\R))$, with $\one ^{+,-}$ the corresponding representation of $\rO_{n+2,n}$. Here
$\one ^{+,-}$ is the character of $\rO_{n+2,n}$ whose restriction to $\rO_{n+2,0}$ is trivial and whose restriction to $\rO_{0,n}$ is $sgn$. Thus the first occurrence index of $\one$ in $\bft^{(2)}$ is $2n+2$. Similarly for the quasi-split Witt tower $\bft^{(-2)}=\{\R^{p,q} \mid p-q=-2 \}$, the first occurrence of $\one \in \Irr{(\Sp_{2n}(\R)})$ is with the dual pair $(\rO_{n,n+2}, \Sp_{2n}(\R))$, with $\one ^{-,+}$ the corresponding representation of $\rO_{n,n+2}$. See \cite[Section 5]{HZ}.
\end{rem}

{\bf Consequence of Lemma \eqref{stable}}. 
Let $\bft_1,\bft_2\in \ST_{\epsilon, \chi}$ (see \eqref{eq:RealWitt} for its definition) be two different Witt towers.
For $i=1,2$, let $V_i\in \bft_i$ be such that $\operatorname m_{\bft_i}(\pi' ) =\dim V_i$ and
\begin{equation}\label{homnv}
  \Hom_{\Sp_\epsilon(V')}(\omega_{V_i,V'},\pi')\neq 0.
\end{equation}

Recall \cite[pp. 91-92]{MVW} that \eqref{homnv} for $i=2$ is equivalent to
\begin{equation*}\label{homnv2}
  \Hom_{\Sp_\epsilon(V')}(\omega_{V_2^-,V'},(\pi')^{\vee})\neq 0.
\end{equation*}
Combining the above,
we get
\begin{equation}
\label{homnv3}
  \Hom_{\Sp_{\epsilon_0}(V')}(\omega_{V_1\oplus V_2^-,V'},\C)\neq 0.
\end{equation}
Here $\epsilon _0:=0\in \Z/2\Z$. Note that $V_1\oplus V_2^{-}$ has even dimension, trivial
discriminant character, and does not split, and so it is not quasi-split.
By Lemma \ref{stable}, we conclude that the Witt index of $V_1\oplus
V_2^-$ is at least $2n$, and $\dim V_1+\dim V_2\geq 4n+4$, proving one direction of the conservation relations:
\begin{equation}\label{geqsp}
  \operatorname m_{\bft_1}(\pi')+ \operatorname m_{\bft_2}(\pi')\geq 4n+4.
\end{equation}

\subsection{Occurrence}

We now explain the opposite inequality (on the upper bound of the sum of first occurrence indices) of the conservation relations, which concerns occurrence. The key result leading to the conservation relations is the following

\begin{prop}\label{leq}
\begin{enumerate}
    \item[(a).] For any $\pi\in \Irr(\rO(V))$, we have
\[
  \operatorname n(\pi)+ \operatorname n(\pi\otimes \sgn)\leq \dim V.
\]
\item[(b).] Assume that $\rF$ is not isomorphic to $\C$. Then for any $\pi' \in \Irr(\Sp_\epsilon(V'))$, there are two
different $\bft_1,\bft_2\in \ST_{\epsilon,\chi}$ such that
\[
  \operatorname m_{\bft_1}(\pi ')+\operatorname m_{\bft_2}(\pi ')\leq 4n+4.
\]
\end{enumerate}
\end{prop}

The idea of the proof of the above two inequalities is similar, and it involves the theory of local Zeta integrals, the doubling method and structures of degenerate principal series of orthogonal and symplectic/metaplectic groups.
In the rest of this subsection, we will furnish a proof for the case of a real symplectic/metaplectic group, namely Proposition \ref{leq} (b) for $\rF=\R$. In a sense (in terms of complication of the structure of degenerate principal series), this case is also the most technically involved.

Recall from \eqref{embedding} the imbedding $\Omega _{p,q}^n(\one)\hookrightarrow I_n^\alpha(s_0)$.
Denote by $\rR_{p,q}^n$ the image of $\Omega _{p,q}^n(\one)$ in $I_n^{\alpha}(s_o)$. Set
\[\CQ_{m,\alpha}:= \{(p,q) \mid p+q=m, p-q\equiv \alpha \, (\text{mod } 4) \} \subset \Z^{\geq 0}\times \Z^{\geq 0}.
\]
Then \[
  \sum_{(p,q)\in \CQ_{m,\alpha}} \operatorname
R_{p,q}^n\subset I_n^{\alpha}(s_0), \qquad s_0=\frac{m}{2}-\frac{n+1}{2}.
\]
The precise relationship of $\rR_{p,q}^n$'s with the composition structure of $I_n^{\alpha}(s_o)$ is the subject matter of \cite{LZ1} and \cite{LZ2}.

Recall the parity $\epsilon\in \Z/2\Z$, and $\alpha = \epsilon$ (mod $2$). The key observation to prove the opposite inequality is the following proposition, which can be read off from \cite{LZ1} and \cite{LZ2}.

\begin{prop} [{\cite[Proposition 7.4]{SZJou}}] \label{image}
Let $m\geq n+1$ (i.e., $s_0\geq 0$) be an integer with parity $\epsilon$. Then
\begin{enumerate}
\item[(a).]
\[
  \sum_{(p,q)\in \CQ_{m,\alpha}} \operatorname
R_{p,q}^n=I_n^{\alpha}(s_0).
\]
\item[(b).] Let $(p_1,q_1)\in \CQ_{m,\alpha}$. Then as $\widetilde{\mathrm{Sp}}_{2n}(\R)$-representations,
 \begin{eqnarray*}
  &&\frac{I_n^{\alpha}(s_0)} {\sum_{(p,q)\in \CQ_{m,\alpha}\backslash \{(p_1,q_1)\}} \rR_{p,q}^n} \\
 &\cong& \left\{
           \begin{array}{ll}
            \rR_{p',q'}^n, & \hbox{if there exists a quadratic space $\R^{p',q'}$ of dimension $2n+2-m$} \\
              & \hbox{which lies in the same Witt tower as the quadratic space $\R^{p_1,q_1}$;}\\
            0, & \hbox{otherwise.}
           \end{array}
         \right.
 \end{eqnarray*}
\end{enumerate}
\end{prop}

Note that the condition that there exists a quadratic space $\R^{p',q'}$ of dimension $2n+2-m$ which lies in the same Witt tower as the quadratic space $\R^{p_1,q_1}$ amounts to
\begin{equation}\label{cond}
p_1,q_1\leq n+1, \text{ and } (p',q')=(n+1-q_1, n+1-p_1).\end{equation}

\vsp

Finally, the theory of local Zeta integrals \cite{PSR, LR, Ga1} enters in the following way: it gives an explicit construction of a nonzero element in
\[
 \Hom_{\Sp_\epsilon(V')\times \Sp_\epsilon(V'^-)}(I_{2n}^{\alpha}(s),
\pi'\widehat{\otimes} (\pi')^\vee),
\]
where $\pi' \in \Irr(\Sp_\epsilon(V'))$, $s\in \C$ and $\alpha \in \Z/4\Z$. (Here $\dim V'=2n$.)
Briefly the construction goes as follows. For $f_s\in I_{2n}^{\alpha}(s)$ with $f_s|_{\cover{\rU}_{2n}}$ independent of $s$, and $u \in \pi'$, $u'\in (\pi')^\vee$, set
\[
Z(s,f_s, u^{\vee},u)=\int_{\Sp(V')}f_s(g',1)\la (g')^{-1}u, u^{\vee}\ra \rd\!g',\]
where we have written $(g', 1)\in \Sp(V')\times \Sp(V'^-)\subset \Sp(\VV')$.
The function $Z(s,f_s, u^{\vee},u)$, initially defined for $s$ in a half-plane, admits a meromorphic continuation to all of $\C$. Taking the leading term of the pole at each point $s=s_0$ will then define a nonzero
$\Sp_\epsilon(V')\times \Sp_\epsilon(V'^-)$-invariant continuous linear form on $ I_{2n}^{\alpha}(s_0) \widehat{\otimes} ((\pi')^{\vee}\widehat{\otimes} \pi')$, and therefore defines a nonzero element in
\[
 \Hom_{\Sp_\epsilon(V')\times \Sp_\epsilon(V'^-)}(I_{2n}^{\alpha}(s_0),
\pi'\widehat{\otimes} (\pi')^\vee).
\]

We record the conclusion as the following

\begin{lem}\label{zetaint}
For $\pi' \in \Irr(\Sp_\epsilon(V'))$, $s\in \C$ and $\alpha \in \Z/4\Z$, we have
\[
 \Hom_{\Sp_\epsilon(V')\times \Sp_\epsilon(V'^-)}(I_{2n}^{\alpha}(s),
\pi'\widehat{\otimes} (\pi')^\vee)\neq 0.
\]
\end{lem}

\vsp

\begin{proof} [Proof of Proposition \ref{leq} (b) for $\rF=\R$]
Let
\[\rm (\pi'):
 =\min\{\rm_{\bft}(\pi')\mid \bft \in  \ST_{\epsilon, \alpha}\}.
\]
This depends on $\epsilon,\alpha$, which we have fixed. First we show \[\rm(\pi' )\leq \begin{cases} 2n+1, \ \ \epsilon =1,\\
                               2n+2, \ \ \epsilon =0.\end{cases}
                               \]
By Lemma \ref{zetaint}, we have (for any $s\in \C$),
\[
 \Hom_{\Sp_\epsilon(V')\times \Sp_\epsilon(V'^-)}(I_{2n}^{\alpha}(s),
\pi' \widehat{\otimes} (\pi')^\vee)\neq 0.
\]

Let \[m = \begin{cases} 2n+1, \ \ \epsilon =1,\\
                               2n+2, \ \ \epsilon =0.\end{cases}
                               \]
and $s_0=\frac{m}{2}-\frac{2n+1}{2}$. Part (a) of Proposition \ref{image} implies that there exists $(p,q)\in \CQ_{m,\alpha}$ such that
\[
 \Hom_{\Sp_\epsilon(V')\times \Sp_\epsilon(V'^-)}(\rR _{p,q}^{2n},
\pi'\widehat{\otimes} (\pi')^\vee)\neq 0.
\]
In turn, by Lemma \ref{nonvanish}, this implies that for the quadratic space $V=\R^{p,q}$, we have
\[
  \Hom_{\Sp_\epsilon(V')}(\omega_{V,V'},
\pi')\neq 0.
\]
Therefore $\rm(\pi') \leq m$.

Now pick a quadratic space $V_0\in \curlyvee_{\epsilon, \alpha}$ (see \eqref{eq:RealWitt} for its definition) so that
\[
   \dim V_0=\operatorname m(\pi') \quad (\leq 2n+2) \quad\textrm{and}\quad \Hom_{\Sp_\epsilon(V')}(\omega_{V_0,V'},\pi')\neq
0.
\]

Set $m':=4n+4-\rm (\pi')$. Again by Lemma \ref{zetaint}, we may pick a nonzero element
\[
 \mu \in \Hom_{\Sp_\epsilon(V')\times \Sp_\epsilon(V'^-)}(I^{\alpha}_{2n}(s'),
\pi'\widehat{\otimes} (\pi')^\vee), \quad s'=\frac{m'}{2}-\frac{2n+1}{2}.
\]
Denote by $V_1=\R^{p_1,q_1}$ the (unique) quadratic space of dimension $m'$ which belongs to the same Witt tower as
$V_0$. Note that $V_1$ exists because $\dim V_0=\operatorname m(\pi')\leq 2n+2$.
In view of Lemma \ref{nonvanish}, it will suffice to show that there is a pair $(p,q)\in \CQ_{m',\alpha}\backslash \{(p_1,q_1)\}$ such that $\mu$ does not vanish on $\operatorname \rR_{p,q}^{2n}$. Suppose this is not the case, then $\mu$ factors to a nonzero linear
map on
\[
  \frac{I^{\alpha}_{2n}(s')}{\sum_{(p,q)\in \CQ_{m',\alpha}\backslash\{(p_1,q_1)\}}\rR_{p,q}^{2n}}.
\]
Part (b) of Proposition \ref{image} implies that $p_1, q_1\leq 2n+1$ (see \eqref{cond}), and $\mu $ yields a nonzero linear map in
\[\Hom_{\Sp_\epsilon(V')\times \Sp_\epsilon(V'^-)}(\rR _{p',q'}^{2n}, \pi'\widehat{\otimes} (\pi')^\vee),\] where
$p'=2n+1-q_1$, $q'=2n+1-p_1$. By Lemma \ref{nonvanish}, this implies that for the quadratic space $V=\R^{p',q'}$,
\[
  \Hom_{\Sp_\epsilon(V')}(\omega_{V,V'},
\pi')\neq 0.
\]
Note that $p'+q'=4n+2-m'=\rm (\pi')-2$, contradicting to the minimality in
the definition of $\operatorname m(\pi')$.
\end{proof}

The inequality in \eqref{geqsp}
and the reverse inequality in Proposition \ref{leq} imply Part (a) of Theorem \ref{realsp}. It remains to prove Part (b) of Theorem \ref{realsp}. Suppose that $\bft_1,\bft_2\in \ST_{\epsilon, \alpha}$ are two different Witt
towers so that
\[
  \operatorname m_{\bft_1}(\pi ')+\operatorname m_{\bft_2}(\pi')=4n+4.
\]

Let $V_1\in \bft_1$ and $V_2\in \bft_2$ be quadratic spaces such that
\[
   \dim V_1+\dim V_2=4n+4,
\]
and
\[
   \Hom_{\Sp_\epsilon(W)}(\omega_{V_i,V'},\pi')\neq 0, \quad i=1,2.
\]
As noted earlier in \eqref{homnv3}, Lemma \ref{stable} implies that the Witt index of $V_1\oplus
V_2^-$ is at least $2n$. This implies that the inequality in \eqref{srank} must be an equality, namely
\[
  \textrm{the Witt index of $(V_1\oplus V_2^-)$}=\frac{\dim V_1+\dim
V_2-4}{2} \qquad (=2n).\]
Therefore $\bft_1$ and $\bft_2$ are adjacent.
\qed

\section{Correspondence of invariants (I): generalized Whittaker models}
\label{sec:GWM}

In the representation theory of (real or $p$-adic) reductive groups, an important invariant of smooth representations is the generalized Whittaker models (attached to nilpotent orbits in the Lie algebra). This is the local version of the Fourier coefficients of automorphic forms. (See \cite{Ji} for discussions on a general framework for the study of periods of automorphic forms.)

In this section we will discuss the main result of \cite{Zh2} on the transition of generalized Whittaker models, more precisely how generalized Whittaker models of the full theta lift of a representation $\pi$ relate to those of $\pi$, valid for any local field of characteristic $0$. (A global analogue was recently established by B. Wang \cite{Wa}.). This is a refinement of the previous result of Gomez-Zhu in \cite{GZ}. We include comprehensive details in these notes, partly because the key technical proposition of (\cite[Section 3.3]{Zh2}), which computes the twisted Jaquet modules of the oscillator representation, should be made more precise. For instance, the isomorphism $\Psi _T$ there was not made explicit, and there were some unexplained notations and actions. Historically there were many earlier results on transition of models, including those of Furusawa \cite{Fu95}, M\oe{}glin \cite{Mo98} and Ginzburg-Jiang-Soudry \cite{GJS}. See also Gan's article \cite{Ga2}, which offers another perspective on relating certain models of two groups in the dual pair setting.

For both this section and the next section, a certain double fibration of moment maps plays a special role. We remark that moment maps in the theta correspondence setting were first exploited in the work of Przebinda \cite{PrUnip,PrUnit}, following the work of
Kraft-Procesi \cite{KP} in the geometric setting.

\subsection{Generalized Whittaker models associated to an $\sl_{2}$-triple }
\label{subsec:GWM}
We recall some basics of generalized Whittaker models \cite{Yamash,MW87}. (See also the earlier work of Kawanaka \cite{Ka} in the finite field case.) Recall that $\rF$ is a local field of characteristic $0$, on which we fix a nontrivial unitary character $\psi$.
Let $G$ be a reductive group
over $\rF$, and $\g$ be its Lie algebra. We also fix a $\Ad\,G$ invariant nondegenerate ($\rF$-valued) bilinear form $\kappa$ on $\g$.
Let $\Orb \subset \g$ be a nilpotent orbit, and $X\in \Orb$. We complete $X$ to an $\sl_{2}$-triple $\gamma=\{X,H,Y\}\subset \g$, namely
\[
[H,X]=2X,\qquad [H,Y]=-2Y, \qquad [X,Y]=H.
\]
Set $\g_{j}=\set{Z\in \g}{\ad(H)Z=jZ}$, for $j\in \Z$. Then, from standard
$\sl_{2}$-theory, we have a finite direct sum $\g=\oplus_{j\in \Z}\g_{j}$.
Define the Lie subalgebras $\u=\oplus_{j\leq -2}\g_{j}$, $\n=\oplus_{j\leq -1}\g_{j}$, $\p=\oplus_{j\leq 0} \g_{j}$
and $\m=\g_{0}$. Let $U$, $N$, $P$, and $M$ be the corresponding
subgroups of $G$. Thus $U=\exp \u$, $N=\exp \n$, $P=\set{p\in G}{\Ad(p)\p
\subset \p}$ and $M=\set{m\in G}{\mbox{$\Ad(m)H=H$}}$. Define the (non-degenerate) character
$\chi_{\gamma}$ of $U$ by
\begin{equation}
\label{defchi}
\chi_{\gamma}(\exp Z)=\psi (\kappa(X,Z)), \ \ \ \ \forall \ Z\in \u.
\end{equation}
Let $M_{X}=\set{m\in M}{\mbox{$\Ad(m)X=X$}}$, the stabilizer group of $X$ in $M$. Then it is well-known
\cite[Section 3.4]{CM} that
\[
M_{X}=G_{\gamma}:=\set{g\in G}{\mbox{$\Ad(g)X=X$, $\Ad(g)Y=Y$, $\Ad(g)H=H$}}.
\]
In particular $M_{X}$ is reductive.

For the moment assume that $\g_{-1}\neq 0$. In this case
$\ad(X)|_{\g_{-1}}:\g_{-1}\longrightarrow \g_{1}$ is an isomorphism,
and we may define a symplectic structure on $\g_{-1}$ by setting
\begin{equation}
\label{defsymg-1}
\kappa_{-1}(S,T)=\kappa(X,[S,T]), \qquad \mbox{for $S$, $T\in \g_{-1}$}.
\end{equation}
Denote by $\rH(\g_{-1})$ the associated Heisenberg group. Then there is a canonical surjective group homomorphism from $N$ to $\rH (\g_{-1})$ which maps $\exp Z$ to $\kappa(X,Z)$ in the center of $\rH (\g_{-1})$, for $Z\in \u$:
\[
\exp R \cdot \exp Z \mapsto (R,\kappa (X,Z)) \in \rH(\g_{-1}), \ \ \mbox{for $R\in \g_{-1}$, $Z\in \u$.}
\]
Then, according to the Stone-von
Neumann theorem, there exists a unique, up to equivalence, smooth
irreducible (unitarizable) representation $\mathscr{S}_{\gamma}$ of $N$ such that $U$
acts by the character $\chi_{\gamma}$. Since $M_{X}$ preserves
$\gamma $ and thus the symplectic form $\kappa_{-1}$, it is a subgroup of $\Sp (\g_{-1})$.
One may thus extend the representation $\mathscr{S}_{\gamma}$ of $N$ to a representation of the semi-direct product $\cover{M_X}\ltimes N$, via the smooth oscillator representation $\omega _{\gamma}$ of $\cover{\Sp}(\g_{-1})\ltimes \rH (\g_{-1})$ (attached to the central character $\psi$).
Here $\cover{M_X}$ denotes the inverse image of $M_X$ in $\cover{\Sp}(W)$ (and similarly for other subgroups of $\Sp (\g_{-1})$). 

If $\g_{-1}=0$, we will use the same notation $\SS_{{\gamma}}$ to denote the $1$-dimensional
representation of $N=U$ given by the character $\chi_{\gamma}$, and $\cover{M_X}:=M_{X}$ acts on it trivially.

\begin{defn} Let $(\pi,\mathscr{V})$ be a smooth representation of $G$.
\begin{enumerate}
\item [(a)]
 Define the following $\cover{M_{X}}$-module
 \[
\Wh_{\Orb}(\pi)=\Hom_{N}(\mathscr{V},\SS_{{\gamma}}),
\]
called the space of generalized Whittaker models of $\pi$ associated to the nilpotent orbit $\Orb$.
\item [(b)] Let $L$ be a reductive subgroup of $M_{X}$, and $(\tau,\mathscr{V}_{\tau})$ be a smooth genuine representation of $\cover{L}$.  Define
\[
\Wh_{\Orb,\tau}(\pi)=\Hom_{LN}(\mathscr{V},\mathscr{V}_{\tau}\widehat \otimes
\SS_{{\gamma}}).
\]
Here $\mathscr{V}_{\tau}$ is viewed as a representation of $\cover{L}\ltimes N$ with the trivial $N$ action, and $\SS_{{\gamma}}$ a representation of $\cover{L}\ltimes N$ via the smooth oscillator representation $\omega_{\gamma}$.
\end{enumerate}
\end{defn}

\begin{rem} In the representation theory literature, elements of $\Wh_{\Orb,\tau}(\pi)$ are also referred to as Kawanaka or generalized Gelfand-Graev models \cite{Ka}, which are generalizations of Bessel and Fourier-Jacobi models. (Typically one takes $L$ to be $M_{X}$ or a ``large'' (e.g. spherical) subgroup of $M_X$.) We refer the reader to \cite{GGP} for more information on these special models such as uniqueness.
\end{rem}

We now assume that $G$ is a type I classical group defined by an $\epsilon$-Hermitian $\rD$-space $(V,B)$, where $B$ is the $\epsilon$-Hermitian form. We will describe a parametrization of nilpotent orbits in $\g$ and the explicit description of the stabilizer group $M_X$ in this parametrization. For details, see \cite[Section 3]{GZ}.

Set $\g_{\gamma}=\Span_{\rF}\{X,H,Y\}\subset \g$. Under the action of $\g_{\gamma}$, we have the isotypic decomposition:
\begin{equation*}
 V=\bigoplus_{j=1}^{l} V^{\gamma,t_{j}},
\end{equation*}
where $V^{\gamma,t_{j}}$ is a direct sum of irreducible $t_{j}$-dimensional $\g_{\gamma}$-modules, and $t_{1}> t_{2}> \ldots > t_{l}>0$.

We need one more notation. Denote by $(\rS_{m},\rF^{m})$ the irreducible representation of $\g_{\gamma}\simeq \sl_{2}$ of dimension $m$. It is well-known that there is a non-degenerate invariant bilinear form on $\rF^{m}$, which is unique up to scalar and is $(-1)^{m-1}$-symmetric. See \cite[Lemmas 5.1.10 and 5.1.14]{CM}. (The lemmas are stated for $\rF=\C$; they work for arbitrary $\rF$.)  We fix one such form $(\cdot,\cdot)_{m}$, and denote the corresponding formed space by $(\rS_m, (\cdot,\cdot)_{m})$.

Let $V^{\gamma,t_{j}}_{t_j-1}$ be the space of highest weight vectors in $V^{\gamma,t_j}$ (the subscript refers to the $H$-weight), and let $i_j=\dim V^{\gamma,t_j}_{t_j-1}$, which is the multiplicity of $(\rS_{t_j},\rF^{t_j})$ in $V$. The isotypic decomposition gives rise to a partition or equivalently a Young diagram $\mathbf{d}^{\gamma}=[t_{1}^{i_{1}},\ldots,t_{l}^{i_{l}}]$ of size $\dim V$. Using $\sl_{2}$ theory and the $\epsilon$-Hermitian form $B$ on $V$, one may define a non-degenerate $\epsilon_j$-Hermitian form $B^{\gamma,t_{j}}_{t_j-1}$ on $V^{\gamma,t_{j}}_{t_j-1}$, where $\epsilon_j=(-1)^{t_j-1}\epsilon$.
As $\epsilon$-Hermitian spaces, we have
\begin{equation}
\label{eq:adm}
(V,B)\simeq \bigoplus _{j}(V^{\gamma,t_{j}}_{t_{j}-1},B^{\gamma,t_{j}}_{t_{j}-1}) \otimes (\rS_{t_j},(\cdot,\cdot)_{t_j}).
\end{equation}
We may then describe the stabilizer group $M_X$ as follows. Since $M_{X}$ acts on $V^{\gamma,t_{j}}_{t_{j}-1}$ preserving $B^{\gamma,t_{j}}_{t_{j}-1}$, there is an embedding of $G(V^{\gamma,t_{j}}_{t_{j}-1})$ into $M_{X}$, and we identify $G(V^{\gamma,t_{j}}_{t_{j}-1})$ with its image in $M_X$. Then we have
\begin{equation}\label{MX}
M_{X}\cong \prod_{j=1}^{l} G(V^{\gamma,t_{j}}_{t_{j}-1}).
\end{equation}

It is well-known \cite{CM} that the pair, consisting of
\begin{itemize}
\item[(i)] the Young diagram $\mathbf{d}^{\gamma}=[t_{1}^{i_{1}},\ldots,t_{l}^{i_{l}}]$, and the collection of
\item[(ii)] the formed spaces $(V^{\gamma,t_{j}}_{t_{j}-1},B^{\gamma,t_{j}}_{t_{j}-1})$ of dimension $i_j$ for $1\leq j\leq l$,
\end{itemize}
uniquely determines the $\sl_{2}$-triple $\gamma$ up to the Adjoint action of $G$. Such a pair is called an admissible $\epsilon$-Hermitian Young tableaux (or
admissible Young tableaux in short; admissibility refers to the requirement in \eqref{eq:adm}). With an obvious notion of equivalence, nilpotent orbits in $\g$ are then parameterized by equivalence classes of admissible $\epsilon$-Hermitian Young tableaux. For a more detailed description of the parametrization, see \cite{DKPR1,DKPR2}.

\subsection{Descent of $\sl_{2}$-triples}
\label{subsec:GD}

We are given a type I reductive dual pair:
\[(G,G')=(G(V),G(V'))\subset \Sp(W), \quad W= \Hom_{\rD} (V,V').\]

To ease notation, we will skip $\rD$ in $\Hom_{\rD} (V,V')$ from now on. Recall the moment maps: \cite{KP,DKPC}
\[
\begin{diagram}
 & &  \Hom (V,V') & &   \\
&  \ldTo^{\varphi} & &\rdTo^{{\varphi}'} &  \\
\g&  & & &{\g}'
\end{diagram}
\]
where $\varphi (T)= T^{\ast}T$ and ${\varphi}'(T)=TT^{\ast}$.

Define
\begin{equation}
\label{eq:Gen}
\Hom^{su}(V,{V}') =\{ T\in \Hom(V,V') \mid \mbox{$\Ker (T)$ is a non-degenerate subspace of $V$}\}.
\end{equation}
(The superscript $su$ reflects the author's fondness of the notion of stable/unstable points in geometric invariant theory.)

For a given $\sl_{2}$-triple $\gamma=\{X,H,Y\}\subset \g$, set $V_{k}=\{v\in V\, | \, Hv=kv\}$, for $k\in \Z$. We have by standard $\sl_{2}$-theory,
\begin{equation*}
V=\bigoplus_{k\in \Z} V_{k}.
\end{equation*}
Similar notations apply  for an $\sl_{2}$-triple $\gamma' =\{X',H',Y'\}\subset \g'$.

\begin{defn}
Let $\gamma\subset \g,$ ${\gamma}'\subset {\g}'$ be two $\sl_{2}$-triples, and $T\in \Hom(V,{V}')$. We say that $T$ lifts $\gamma$ to ${\gamma}'$ if
\begin{enumerate}
\item $\varphi (T) =X$ and $\varphi '(T) ={X}'$.
\item $T(V_{k})\subset {V'_{k+1}}$, for all $k$.
\item $T\in \Hom^{su}(V,{V}')$.
\end{enumerate}

Set
\begin{equation}
\label{eq:Ogg'}
\Orb_{\gamma,{\gamma}'}=\{T\in \Hom(V,{V}')\, | \, \mbox{$T$ lifts $\gamma$ to ${\gamma}'$}\}.
\end{equation}
\end{defn}

\begin{prop}
\label{descent}
Let $\gamma'=\{X',H',Y'\}\subset \g'$ be an $\sl_{2}$-triple. If $X'$ is in the image of $\varphi '$, then there is a unique conjugate class of
$\sl_{2}$-triple $\gamma=\{X,H,Y\}\subset \g$, such that $\Orb_{\gamma,{\gamma}'}$ is non-empty. Furthermore $\Orb_{\gamma,{\gamma}'}$ is a single $M_{X}\times M'_{{X'}}$-orbit.
\end{prop}

\begin{proof}
We shift notation temporaily from $X'$ to $\tX$, $V'$ to $\tV$, and $\gamma'$ to $\tgamma$, etc, in order to reduce the burden of double superscript. Write $\tX=\tphi (S)=SS^{\ast}$, for some $S\in \Hom(V,\tV)$.

First observe that $\Ker (S^{\ast}) \subseteq \Ker (\tX)$. By $\sl_2$ theory, we have a direct sum decomposition
\[\tV=\Ker (\tX) \oplus \Im (\tY).\]
This implies that $S^*|_{\Im (\tY)}$ is injective.

Next observe that
\begin{equation}\label{TransForm}
B(S^{\ast}\tv, S^{\ast}\tw)=\tB(\tX \tv, \tw), \quad \tv, \tw\in \tV.
\end{equation}
This implies that
\[\text{the radical of } B|_{S^{\ast}(\tV)}=S^{\ast}(\Ker (\tX)).\]
In particular, $S^{\ast}(\tV)$ is non-degenerate if and only if $\Ker (\tX)\subseteq \Ker (S^{\ast})$. In turn, this is equivalent to $\Ker (\tX)=\Ker (S^{\ast})$, in view of the reverse containment.

Define a new element $T\in \Hom(V,\tV)$ by setting
\[
\begin{aligned}
T^*=\begin{cases} S^*, \ \ \text{ on } \Im (\tY), \\
           0, \ \ \ \ \text{ on } \Ker (\tX). \end{cases}
           \end{aligned}
           \]
We have the replacement of \eqref{TransForm}:
\begin{equation}\label{TransFormNew}
B(T^{\ast}\tv,T^{\ast}\tw)=\tB(\tX\tv, \tw), \quad \tv,\tw\in \tV.
\end{equation}
Therefore $\tX=TT^{\ast}$, $T^*|_{\Im (\tY)}$ is injective, and $T^{\ast}(\tV)$ is non-degenerate or equivalently $\Ker (T)$ is non-degenerate. Moreover, we have an orthogonal direct sum decomposition:
\[V=T^{\ast}(\tV)\oplus \Ker (T).\]

Let $X:=T^*T$. Clearly we have $XT^*=T^*\tX$, and $X|_{\Ker (T)}=0$. We will extend $X$ to an $\sl_2$-triple $\gamma =\{X,H,Y\}$, and we start with the definition of $H$ first.

Write $\tV=\oplus_{j} \tV^{\tgamma,t_j+1}$, the $\sl_2$-isotypic decomposition. Since $\tV^{\tgamma,t_j+1}$ is $\tX=TT^*$-stable, we have an orthogonal direct sum decomposition:
\[
T^*(\tV)=\bigoplus _{j} T^*(\tV^{\tgamma,t_{j}+1}).
\]

For each $j$, write $V^{\gamma,t_{j}}= T^*(\tV^{\tgamma,t_{j}+1})$, which is $X$-stable. (At this point, $\gamma$ has no meaning.) We have an obvious orthogonal direct sum decomposition:
\[
T^{\ast}(\tV)=\bigoplus _{j} V^{\gamma,t_{j}}.
\]

Note that $T^*(\tV_{t_j}^{\tgamma,t_{j}+1})=0$ by construction. Set
\[
V_{i}^{\gamma,t_{j}}=T^*(\tV_{i-1}^{\tgamma,t_{j}+1}), \quad i\in [-t_j+1, t_j-1]_2.
\]
Here and after, for a non-negative integer $a$, $[-a,a]_2$ denotes the set of integers from $-a$ to $a$, differing by $2$ each time.
Note that for $i$ as in above, $T^{\ast}|_{\tilde{V}^{\tilde{\gamma},t_{j}+1}_{i-1}}:\tilde{V}^{\tilde{\gamma},t_{j}+1}_{i-1}\rightarrow V^{\gamma,t_{j}}_{i}$ is a linear isomorphism, and $T(V_{i}^{\gamma,t_{j}})\subseteq \tV_{i+1}^{\tgamma,t_{j}+1}$. In particular we have a direct sum decomposition:
\[V^{\gamma,t_{j}}=\oplus _{i\in [-t_j+1, t_j-1]_2} V_{i}^{\gamma,t_{j}}.
\]
Furthermore, $X(V^{\gamma,t_{j}}_{t_j-1})=0$, and $X|_{V^{\gamma,t_{j}}_{i}}: V^{\gamma,t_{j}}_{i}\rightarrow {V}^{\gamma,t_{j}}_{i+2}$ is a linear isomorphism for any other $i\ne t_j-1$.

Suppose \[
(\tV^{\tgamma, t_j+1}, \tB)\simeq (\tV^{\tgamma,t_{j}+1}_{t_{j}},\tB^{\tgamma,t_{j}+1}_{t_{j}}) \otimes (\rS_{t_j+1},(\cdot,\cdot)_{t_j+1}), \quad \text{as $\tepsilon$-Hermitian spaces}.
\]
By using \eqref{TransFormNew}, we have
\[
(V^{\gamma, t_j}, B)\simeq (\tV^{\tgamma,t_{j}+1}_{t_{j}},\tB^{\tgamma,t_{j}+1}_{t_{j}}) \otimes (\rS_{t_j},(\cdot,\cdot)_{t_j}), \quad \text{as $\epsilon$-Hermitian spaces}.
\]

Define a semisimple linear transformation $H$ on $V$ as follows:
\[
Hv=\begin{cases} iv, \,\,\,v\in {V_{i}^{\gamma,t_{j}}}, \text{ for } i\in [-t_j+1, t_j-1]_2; \\
                 0, \quad v\in \Ker(T).
\end{cases}
\]

It is easy to check that $H \in \g$, and $[X,H]=2X$. Obviously $X, H$ preserve $V^{\gamma,t_{j}}$ for each $j$. By standard argument (\cite[Section 3.3]{CM}), we can complete $X,H$ to an $\sl_{2}$-triple $\gamma =\{X,H,Y\}$ in $\g$, with $Y$ preserving each $V^{\gamma,t_{j}}$ and $0$ on $\Ker (T)$. Moreover we have the $\gamma$-isotypic decomposition,
\begin{equation}\label{sl2Decompose}
V=\bigoplus _{j} V^{\gamma,t_{j}}\oplus \Ker (T),
\end{equation}
giving rise to an admissible $\epsilon$-Hermitian Young tableaux:
\[
(V^{\gamma, t_j}, B)\simeq (\tV^{\tgamma,t_{j}+1}_{t_{j}},\tB^{\tgamma,t_{j}+1}_{t_{j}}) \otimes (\rS_{t_j},(\cdot,\cdot)_{t_j}), \quad \text{for $t_j\geq 2$},
\]
with a minor adjustment for the space of $\gamma$-invariants:
\[
(V^{\gamma, 1}, B)\simeq (\tV^{\tgamma,2}_{1},\tB^{\tgamma,2}_{1}) \oplus (\Ker(T), B|_{\Ker(T)}).
\]

We have thus demonstrated the existence of $T \in \Orb_{\gamma,{\gamma}'}$. Clearly $\gamma$ is unique since it is uniquely determined by its $\epsilon$-Hermitian Young tableaux.

Finally we show that $\Orb_{\gamma,{\gamma}'}$ forms a single $M_{X}\times M'_{{X'}}$-orbit. Supposed we are given $T \in \Orb_{\gamma,{\gamma}'}$. We will show that $T|_{V^{\gamma, t_j}}$ is uniquely determined by its restriction on the highest weight space $V^{\gamma, t_j}_{t_j-1}$, for each $j$. Furthermore $T|_{V^{\gamma, t_j}_{t_j-1}}: V^{\gamma, t_j}_{t_j-1}\rightarrow \tV^{\tgamma, t_j+1}_{t_j}$ is a linear isomorphism.

Since $V^{\gamma, t_j}$ lies in the orthogonal complement of $\Ker (T)$, we see $T$ is injective on $V^{\gamma, t_j}$. We consider $T|_{V_{i}^{\gamma,t_{j}}}$, where $i\in [-t_j+1, t_j-1)_2$. Here $[-a,a)_2$ denotes the set $[-a,a]_2\backslash \{a\}$.

Observe that
\[T^*|_{\tV_{i+1}^{\tgamma,t_{j}+1}}\circ T|_{V_{i}^{\gamma,t_{j}}} =X|_{V_{i}^{\gamma,t_{j}}},\]
\[T|_{V_{i+2}^{\gamma,t_{j}}}\circ T^*|_{\tV_{i+1}^{\tgamma,t_{j}+1}} =\tX|_{\tV_{i+1}^{\tgamma,t_{j}+1}}.\]
Since $\tX$ is injective on $\tV_{i+1}^{\tgamma,t_{j}+1}$, the second equation implies that $T^*$ is also injective on $\tV_{i+1}^{\tgamma,t_{j}+1}$. The first equation then implies that $T|_{V_{i}^{\gamma,t_{j}}}$ is determined by $T^*|_{\tV_{i+1}^{\tgamma,t_{j}+1}}$. As noted earlier, $T$ is injective on $V_{i+2}^{\gamma,t_{j}}$. The second equation then implies that $T^*|_{\tV_{i+1}^{\tgamma,t_{j}+1}}$ is determined by $T|_{V_{i+2}^{\gamma,t_{j}}}$. We have thus shown that $T|_{V_{i}^{\gamma,t_{j}}}$ is by $T|_{V_{i+2}^{\gamma,t_{j}}}$. Continuing this way, we conclude that $T|_{V_{i}^{\gamma,t_{j}}}$, where $i\in [-t_j+1, t_j-1]$, is uniquely determined by $T|_{V^{\gamma, t_j}_{t_j-1}}$.
It is instructive to look at the following diagram: (superscript omitted for simplicity)
\[
\begin{array}{c}   V_{i}  \xlongrightarrow{X} V_{i+2} \\
\hspace{10pt} {\scriptstyle T} \searrow {\scriptstyle T^*} \nearrow  \searrow {\scriptstyle T}\\
\hspace{50pt}{\tV}_{i+1}   \xlongrightarrow{\tX} \tV_{i+3}
\end{array}
\]

We already know $T|_{V^{\gamma, t_j}_{t_j-1}}: V^{\gamma, t_j}_{t_j-1}\rightarrow \tV^{\tgamma, t_j+1}_{t_j}$ is injective. Its surjectivity follows from the fact that $\tX$ is surjective onto $\tV^{\tgamma, t_j+1}_{t_j}$, and $\tX= TT^*$. By the description of the forms on $V^{\gamma, t_j}_{t_j-1}$ and $\tV^{\tgamma, t_j+1}_{t_j}$, it is clear that $T|_{V^{\gamma, t_j}_{t_j-1}}$ is an isometry. It follows immediately that all such $T$'s form a single $M_{X}\times M'_{{X'}}$-orbit.
\end{proof}

\begin{defn} \label{def:descent} We are in the setting of Proposition \ref{descent}. We will say that $\gamma $ (resp. $\Orb $) is the descent of $\gamma ' $ (resp. $\Orb '$), and we write
\[\Orb = \nabla^{V'}_{V}(\Orb '), \quad (\text{or $\nabla (\Orb )$ in short}). \]
When $T\in \Orb_{\gamma,{\gamma}'}$ is injective, we will refer to it as the pure descent case. We call $\gamma $ (resp. $\Orb $) the pure descent of $\gamma ' $ (resp. $\Orb '$), and write $\Orb = \nabla_{pure}(\Orb ')$.
\end{defn}

The proof of Proposition \ref{descent} also gives the description of the descent $\Orb ' \mapsto \Orb $ in terms of admissible Young tableaux, which we summarize. Suppose that $\Orb' \subset \g'$ corresponds to the admissible $\epsilon'$-Hermitian Young tableaux, consisting of
\begin{equation}\label{Orbp}
\begin{aligned} & \text{ (i) the Young diagram }\mathbf{d}^{\gamma'}=[(t_{1}+1)^{i_{1}},\ldots,(t_{l}+1)^{i_{l}},1^{s}], \qquad t_{1}> t_{2}> \ldots >t_{l-1}> t_{l}=1; \\
& \text{(ii) the formed spaces $((V')^{\gamma',t_{j}+1}_{t_{j}},B^{\gamma',t_{j}+1}_{t_{j}})$ of dimension $i_j$ ($1\leq j\leq l$), and $(V')^{\gamma',1}$ of dimension $s$.}
\end{aligned}
\end{equation}

If $\Orb'$ is in the image of the moment map, then we have an embedding of $\epsilon$-Hermitian spaces:
\begin{gather}
(T^*(V'), B)=\{\oplus _{1\leq j\leq l}((V')^{\gamma',t_{j}+1}_{t_{j}},B^{{\gamma'},t_{j}+1}_{t_{j}})\otimes (\rS_{t_j},(\cdot,\cdot)_{t_j})\} \notag\\
\hookrightarrow (V,B).\label{eq:embed}
\end{gather}

Denote by $(U,B_U)$ the orthogonal complement of the direct sum space for $1\leq j\leq l-1$, so that
\[(V,B)= \{\oplus _{1\leq j\leq l-1}((V')^{\gamma',t_{j}+1}_{t_{j}},B^{{\gamma'},t_{j}+1}_{t_{j}})\otimes (\rS_{t_j},(\cdot,\cdot)_{t_j})\} \oplus (U,B_U).\]
We have
\[(U,B_U)= (U_1,B_{U_1}) \oplus (\Ker(T), B|_{\Ker(T)}), \]
where
\[(U_1,B_{U_1})
\simeq ((V')^{\gamma',2}_1, B^{\gamma', 2}_1).\]
Note that in our previous notation, $U=V^{\gamma,1}$, the space of $\gamma$-invariants in $V$.

The admissible $\epsilon$-Hermitian Young tableau corresponding to $\Orb\subset \g$ is as follows.
\begin{equation}\label{Orb}
\begin{aligned} & \text{(i) the Young diagram ${\mathbf d}^{\gamma}=[t_{1}^{i_{1}},\ldots,t_{l-1}^{i_{l-1}}, 1^{i_l+b}]$, where $b=\dim \Ker (T)$}; \\
& \text{(ii) the formed spaces $(V^{\gamma,t_{j}}_{t_{j}-1},B^{\gamma,t_{j}}_{t_{j}-1})\simeq ((V')^{\gamma',t_{j}+1}_{t_{j}},B^{{\gamma'},t_{j}+1}_{t_{j}})$, and $V^{\gamma,1}$ of dimension $i_l+b$}.
\end{aligned}
\end{equation}
Or in descriptive words:
\begin{itemize}
\item The Young diagram of $\Orb$ is obtained by erasing the first column off the Young diagram
of $\Orb '$, and then adding $b$ additional boxes in the first column.
\item The forms on the spaces of multiplicities for rows of length $t_j$ in $\Orb$ are inherited from those for rows of length $t_j+1$ in $\Orb'$, except for rows of length $1$.
\end{itemize}

Next we describe the relationship of the stabilizer groups $M_X$ and $M'_{X'}$.

\begin{defn} Let $\gamma$, ${\gamma}'$ be as before, and $T\in \Orb_{\gamma,{\gamma}'}$. Define the following non-degenerate subspaces of $V$ and $V'$:
\[
V_{\gamma,{\gamma}'}:=\Ker (T) \subseteq V^{\gamma,1}, \quad \text{and } \ V'_{\gamma, \gamma'}:=(V')^{\gamma',1}.
\]
Let $L$ and $L'$ be the isometry groups of $V_{\gamma,{\gamma}'}$ and $V'_{\gamma, \gamma'}$, respectively.
\end{defn}

We define a group homomorphism
\begin{equation}
\label{eq:alp}
\alpha=\alpha_{T}:M'_{X'}\longrightarrow M_{X}
\end{equation}
by setting $\alpha _{T}(m')$ to equal to $T^{-1}m'T$ on $T^*(V')$, and the identity map on $\Ker (T)$, for $m'\in M'_{X'}$. The kernel of $\alpha$ is $L'$. Denote the image of $\alpha$ by
\begin{equation}
\label{eq:MXXp}
M_{X, X'}:=\alpha (M')\subseteq M_X.
\end{equation}

We have
\[
M_{X,X'}\simeq \prod_{j=1}^{l-1} G((V')^{\gamma',t_{j}+1}_{t_{j}}) \times G((V')^{\gamma',2}_{1}).\]
Moreover,
\begin{equation}\label{eq:stab}
M'_{X'}\simeq M_{X,X'}\times L', \text{ and } M_{X} \supset M_{X,X'} \times L.
\end{equation}
Note that $M_X$ contains $M_{X, X'}\times L$ as a symmetric subgroup, due to the fact that $U=(V')^{\gamma',2}_{1}\oplus \Ker (T)$ (orthogonal direct sum) and so $G(U)\supset G((V')^{\gamma',2}_{1})\times L$ as a symmetric subgroup. We will also view $M_{X, X'}$ as a subgroup of $M'_{X'}$ by identifying it with the subgroup of $M'_{X'}$ fixing $(V')^{\gamma',1}$. Then $M'_{X'}= M_{X,X'}\times L'$, and $\alpha$ is the identity map on $M_{X,X'}$ and trivial on $L'$. We may view $M_{X, X'}$ as a subgroup of $M_X\times M'_{X'}$ through the diagonal embedding, in which case $M_{X, X'}$ will be written as $\Delta M_{X, X'}$:
\begin{equation}
\label{eq:MXXp}
\Delta M_{X, X'} \hookrightarrow M_X\times M'_{X'}.
\end{equation}

In the pure descent case, $V_{\gamma,{\gamma}'}=0$, $L$ is the trivial group, and \eqref{eq:stab} becomes
\begin{equation}\label{eq:stab2}
M'_{X'}\simeq M_{X,X'}\times L', \text{ and } M_{X} =M_{X,X'}.
\end{equation}
Note that in this case, the group homomorphism
$\alpha: M'_{X'}\longrightarrow M_{X}$ is surjective.

Finally we record the following criterion on the image of the moment map, for completeness.

\begin{prop} In the setting of Proposition \ref{descent}, $X'$ is in the image of $\varphi '$, if and only if
there exists an embedding of $\epsilon$-Hermitian spaces, as in \eqref{eq:embed}:
\begin{gather*}
\{\oplus _{1\leq j\leq l}((V')^{\gamma',t_{j}+1}_{t_{j}},B^{{\gamma'},t_{j}+1}_{t_{j}})\otimes (\rS_{t_j},(\cdot,\cdot)_{t_j})\} \notag\\
\hookrightarrow (V,B).
\end{gather*}
\end{prop}

\begin{proof} We have alreday seen that if $X'$ is in the image of $\varphi '$, then there exists an embedding of $\epsilon$-Hermitian spaces, as claimed. So we just need to prove the reverse implication.

Suppose that
\[
\{\oplus _{1\leq j\leq l}((V')^{\gamma',t_{j}+1}_{t_{j}},B^{{\gamma'},t_{j}+1}_{t_{j}})\otimes (\rS_{t_j},(\cdot,\cdot)_{t_j})\}
\hookrightarrow (V,B).
\]
Here $t_1> t_2 > \cdots > t_{l-1}> t_l=1$. Write  $U_j=(\rS_{t_j},(\cdot,\cdot)_{t_j})$, and
$U_j'=(\rS_{t_j+1},(\cdot,\cdot)_{t_j+1})$, and consider the moment maps:
\[
\begin{diagram}
 & &  \Hom (U_j,U_j') & &   \\
&  \ldTo^{\varphi} & &\rdTo^{{\varphi}'} &  \\
\g_j&  & & &{\g}_j'
\end{diagram}
\]
where $\g_j$ (resp. $\g_j'$) is the Lie algebra of $G(U_j)$ (resp. $G(U_j')$), and $\varphi (Z)= Z^{\ast}Z$, and ${\varphi}'(Z)=ZZ^{\ast}$. If $t_j$ is even, $G(U_j)$ is the symplectic group $\Sp_{t_j}(\rF)$, and $G(U_j')$ is the split orthogonal group in dimension $t_j+1$. If $t_j$ is odd, $G(U_j)$ is the split orthogonal group in dimension $t_j$, and $G(U_j')$ is the symplectic group $\Sp_{t_j+1}$.
From the proof of Proposition \ref{descent}, we see that there exists an injective $Z_j\in  \Hom (U_j,U_j')$ such that  $\varphi (Z_j)$ is a principal nilpotent element in $\g_j$ and $\varphi '(Z_j)$ is a principal nilpotent element in $\g_j'$. (One constructs $Z_j$ as in the proof of Proposition \ref{descent} by fixing a linear isomorphism between the spaces of highest weight vectors, both of dimension $1$. It is also easy to see that $Z_j$ is unique up to a scalar.)

We have an injective map
\begin{gather*}
\oplus_{1\leq j\leq l} (\text{Id}\otimes Z_j): \,\,
\{\oplus _{1\leq j\leq l}((V')^{\gamma',t_{j}+1}_{t_{j}},B^{{\gamma'},t_{j}+1}_{t_{j}})\otimes (\rS_{t_j},(\cdot,\cdot)_{t_j})\} \\
\longmapsto
\{\oplus _{1\leq j\leq l}((V')^{\gamma',t_{j}+1}_{t_{j}},B^{{\gamma'},t_{j}+1}_{t_{j}})\otimes (\rS_{t_j+1},(\cdot,\cdot)_{t_j+1})\}.
\end{gather*}
By making it zero on its orthogonal complement in $V$, we get an element $Z\in \Hom (V,V')$, and it is easy to see $\varphi '(Z)=X'$.
\end{proof}

\subsection{A canonical symplectic subspace and the associated oscillator representation}
\label{subsec:subspace}

Let $\gamma\subset \g$ and ${\gamma}'\subset {\g}'$ be two $\sl_{2}$-triples of type $\Orb$ and $\Orb '$, where $\Orb = \nabla^{V'}_{V}(\Orb ')$. Pick any $T \in \Orb_{\gamma,{\gamma}'}$. We may assume that
\begin{equation}
\label{eq:VV'}
V=\bigoplus_{k=-r}^{r} V_{k} \qquad \mbox{and} \qquad {V'}=\bigoplus_{k=-r-1}^{r+1} V'_{k},
\end{equation}
for some $r$. (In our previous notation, $r=\max_j\{r_j\}$.)

We introduce the following symplectic subspace of $W$:
\begin{equation}
{W}_{0}:=\bigoplus_{k=-r}^{r}
\Hom(V_{k},{V}'_{k})\subset W=\Hom(V,{V}').
\end{equation}
(The subscript $0$ in $W_0$ is meant to indicate the space of degree $0$ maps in $W$, with respect to $(H,H')$.)

Recall $\g_{-1}$ is a symplectic space associated to the $\sl_2$-triple $\gamma\subset \g$, whose symplectic structure is defined through an $\Ad\,G$-invariant non-degenerate bilinear form $\kappa$ on $\g$, as in \eqref{defsymg-1}.
We now fix the bilinear form $\kappa$ as $\frac{1}{2}\Tr(T^{\ast}S)$ (for $T, S \in \g$), and likewise for ${\kappa'}$. Recall also that for any formed $U$, $U^-$ denotes the space $U$ equipped with the form scaled by $-1$.

The following lemma is easy, by direct computation and dimension matching.
\begin{lem} [\cite{Zh2}] \label{lem:embedding}
 Given $T\in \Orb_{\gamma,{\gamma'}}$, the map
\[
\begin{array}{rcl}
J_{T}: \ \ -\g_{-1}\oplus \g'_{-1} & \longrightarrow &  {W}_{0}, \\
(R,{R'})  & \mapsto &  TR+{R'}T
\end{array}
\]
is a symplectic imbedding, with $(\Im J_{T})^{\perp}=\Hom (V_{\gamma,{\gamma}'},  V'_{\gamma, \gamma'})$ (i.e. $W_{\gamma,{\gamma}'}$).
\end{lem}

To visualize the above result, it is instructive to look at the following diagram:

\[
\begin{array}{c}   V_{i-1}  \xlongleftarrow{R} V_{i} \\
\hspace{50pt} {\scriptstyle T} \searrow \hspace{30pt}\downarrow \hspace{30pt}\searrow {\scriptstyle T}\\
\hspace{100pt}V'_{i}   \xlongleftarrow{R'} V'_{i+1}
\end{array}
\]

Associated to the symplectic space $W_0$, we have the Heisenberg group $\rH(W_0)$, and the
smooth oscillator representation $(\omega _0, \SY_0)$ of $\cover{\Sp}(W_0)\ltimes \rH(W_0)$ (attached to the central character $\psi$).

For any $\sl_2$-triple ${\gamma} =\{X,H,Y\}$, we have a Hisenberg group $\rH_{\gamma}:=\rH (\g_{-1})$. We introduce a modified $\sl_2$-triple $\check{\gamma} =\{-X,H,-Y\}$, with the associated Heisenberg group $\rH_{\check{\gamma}}$. Note that the corresponding symplectic structure on $\g_{-1}$ will then be $-\g _{-1}$.  With this modification, we may extend $J_T$ to a group homomorphism
\begin{equation*}
J_{T}:{\rH}_{\check{\gamma}}\times \rH_{\gamma'}  \longrightarrow   \rH(W_{0}).
\end{equation*}

By composing with the canonical surjective group homomorphism $N\times N'\rightarrow {\rH}_{\check{\gamma}} \times \rH_{\gamma '}$, we obtain a group homomorphism
\begin{equation}
\label{eq:extendJT2}
\phi _{T}: N\times N' \longrightarrow \rH(W_{0}).
\end{equation}

Observe that the action of $\Delta M_{X,X'}$ on $-\g_{-1}\oplus \g'_{-1}$, when transported to $W_0$, preserves the symplectic structure on $W_0$, and as such we may extend $\phi_{T}$ to a group homomorphism (denoted by the same symbol)
\begin{equation*}
\phi _{T}:\Delta M_{X,X'} \ltimes (N\times {N'})\longrightarrow \Sp(W_{0})\ltimes \rH(W_0).
\end{equation*}

We have the reductive dual pair $(L,L')\subseteq \Sp(W_{\gamma, \gamma'})$, where $L$ (resp. $L'$) is the isometry group of $V_{\gamma,{\gamma}'}$ (resp. $V'_{\gamma, \gamma'}$), and $W_{\gamma, \gamma'}:=\Hom (V_{\gamma, \gamma '}, V'_{\gamma, \gamma '})$.
As a subgroup of $M_X$ (resp. $M'_{X'}$), $L$ (resp. $L'$) acts trivially on $N$ (resp. $N'$). On the other hand, the groups $L$ and $L'$ are commuting subgroups of $\Sp(W_{0})$ in a natural way. It is then easy to check that we may extend $\phi_T$ to a group homomorphism, still denoted by $\phi_T$, from the subgroup $(\Delta M_{X,X'} (L\times L')) \ltimes (N\times {N'}) $ of $(M_X\times M'_{X'})\ltimes (N\times {N'})$, by making $\phi_T$ to be the identity map on $L$ and $L'$ (i.e., the natural inclusion of $L$ and $L'$ into $\Sp(W_{0})$):
\begin{equation}
\phi _{T}:(\Delta M_{X,X'} (L\times L')) \ltimes (N\times {N'})\longrightarrow \Sp(W_{0})\ltimes \rH(W_0). \label{eq:extendedalphaTmap}
\end{equation}
By pulling back, this yields a representation $\omega _0^T$ of $(\cover{\Delta M_{X,X'}}\ltimes (N\times {N'})) \times (\cover{L}\times \cover{L'})$ on $\SY_0$:
\begin{equation}
\label{eq:toscillator}
\omega_0^{T}:=\omega _0\circ \phi_{T}.
\end{equation}

Denote by $(\omega_{\gamma, \gamma'}, \SY_{\gamma,{\gamma}'})$ the smooth oscillator representation associated to the symplectic subspace $W_{\gamma,\gamma'}$. The following lemma is straightforward, by using the embedding $J_T$ in Lemma \ref{lem:embedding} and checking the central character.

\begin{lem}\label{lemma:Tisomorphism}  Given $T\in \Orb_{\gamma,{\gamma'}}$, we have an isomorphism of $(\cover{\Delta M_{X,X'}}\ltimes (N\times {N'})) \times (\cover{L}\times \cover{L'})$-modules:
\[
\omega _0^{T}\simeq \SS_{\check{\gamma}}\otimes
\SS_{\gamma'} \otimes \omega _{\gamma,{\gamma}'},
\]
where $\omega_0^T=\omega _0\circ \phi_{T}$, and on the right hand side, $\cover{\Delta M_{X,X'}}\ltimes (N\times {N'})$ acts on $\SS_{\check{\gamma}}\otimes \SS_{\gamma'}$, and $\cover{L}\times \cover{L'}$ acts by the oscillator representation $\omega_{\gamma,\gamma'}$.
\end{lem}

\subsection{The twisted Jaquet modules of the oscillator representation}

Recall that $(G,G')\subseteq \Sp (W)$ is a type I reductive dual pair, and $(\omega , \SY )$ is the smooth oscillator representation of $\cover{\Sp}(W)\ltimes \rH(W)$ (attached to $\psi$). Associated to an $\sl_2$-triple $\gamma'$, we have the unipotet subgroup $U'$ of $G'$ and the non-degenerate character $\chi_{\gamma'}$ of $U'$. Denote by $\SY_{{U'},\chi_{{\gamma}'}}$ the space of $({U'},\chi_{{\gamma}'})$-covariants of the representation $(\omega , \SY )$. We will compute $\SY_{{U'},\chi_{{\gamma}'}}$ explicitly, in terms of $T \in \Orb_{\gamma,{\gamma}'}$, which will be the key technical preparation for the transition of generalized Whittaker models (Theorem \ref{thm:whittaker} in the next subsection). This computation may be compared to the well-known computation of the Jacquet modules of the oscillator representation by Kudla \cite{KuInv}.

We recall from \cite{AG} the notion of of Schwartz functions of a Nash (i.e, smooth semi-algebraic) manifold, and related concepts. For a Nash manifold $M$ and a Fr\'{e}chet space $\SE$, denote by $\SS (M; \SE)$ the space of $\SE$-valued Schwartz functions on $M$. For a Nash group $G$ and a Nash subgroup $H$, the Nash structure on $H\backslash G$ is unique if $G$ is almost linear (i.e., it  has a Nash representation with finite kernel). See \cite[Proposition 3.8]{Su3}. Given a smooth representation $(\rho, \SE)$ of $H$, one has a representation of $G$ on $\SS (H\backslash G; \SE)$ by right translation, known as the Schwartz induction. See \cite[Section 2]{Cl}.

\begin{prop}[\cite{GZ,Zh2}] \label{prop:mainproposition} Let $\gamma\subset \g$ and ${\gamma}'\subset {\g}'$ be two $\sl_{2}$-triples of type $\Orb$ and $\Orb '$, where $\Orb = \nabla_{V',V}(\Orb ')$. Then, given any $T\in \Orb_{\gamma,{\gamma}'}$, there exists an (explicit)
$ \cover{G}\times \cover{M'_{X'}}{N'}$-intertwining isomorphism
\begin{equation}
\label{covariants}
\Psi_{T}:\SY_{{U'},\chi_{{\gamma}'}}\longrightarrow \SS (N\backslash \cover{G};\SY_0).
\end{equation}
On the right hand side, the action of $\cover{G}$ is by right translation, and the action of $\cover{M'_{X'}}{N'}$ is induced by a (explicitly determined) representation of $\cover{M_X}N \times \cover{M'_{X'}}{N'}$ on $\SY_0$ (depending on $T$), whose restriction to $(\cover{\Delta M_{X,X'}}\ltimes (N\times {N'})) \times (\cover{L}\times \cover{L'})$ is  $\omega_0^T$ defined in \eqref{eq:toscillator}.
 \end{prop}

\begin{rem}
\begin{enumerate}
    \item [(a)] In our cases, the Schwartz space $\SS (N\backslash \cover{G};\SY_0)$ can be concretely defined by an appropriate collection of semi-norms. We refer the interested reader to \cite[Section 4.3]{GZ} for details.
    \item [(b)] Proposition \ref{prop:mainproposition} is a more precise version of \cite[Proposition 3.8]{Zh2}. We remark that in the statement of \cite[Proposition 3.8]{Zh2}, $\cover{L}N\backslash \cover{G}$ should just be $N\backslash \cover{G}$ (since $L$ is in fact in $N$).
    \end{enumerate}
    \end{rem}

The rest of this subsection is devoted to the proof of Proposition \ref{prop:mainproposition}. Near the end, the isomorphism $\Phi_T$ will be given explicitly. See Equation \eqref{PsiT}.

We fix $T\in \Orb_{\gamma,{\gamma}'}$, and write $T=\oplus _{k=-r}^{r}T_k$, where $T_k=T|_{V_k}$. We have the following diagram:
\begin{equation}
\label{eq:figT}
\begin{array}{c}   V_{-r}   \oplus    V_{-r+1}    \oplus    \cdots    \oplus    V_{r-1}   \oplus   V_{r}   \\
  \hspace{45pt} \searrow \hspace{-3pt} {\scriptstyle T_{-r}}  \hspace{10pt}  \searrow \hspace{-3pt} {\scriptstyle T_{-r+1}} \hspace{5pt}  \cdots\hspace{5pt}    \searrow \hspace{-3pt} {\scriptstyle T_{r-1}} \hspace{-5pt}   \searrow \hspace{-3pt} {\scriptstyle T_{r}}    \\
V'_{-r-1}    \oplus     V'_{-r}     \oplus     V'_{-r+1}    \oplus    \cdots    \oplus    V'_{r-1}   \oplus   V'_{r}    \oplus   V'_{r+1}.
\end{array}
\end{equation}
As in the proof of Proposition \ref{descent}, we have $\Ker (T)=\Ker (T_0)\subseteq V^{\gamma, 1}\subseteq V_0$, and $T_k$ is injective for $k\ne 0$.

The proof of Proposition \ref{prop:mainproposition} is by induction on $r$, and we proceed to give the main ingredients of the inductive step. To facilitate the discussion, we introduce some notations. For $l\leq r$ and $m\leq r+1$, set
\[
V_{(l)}=\bigoplus_{k=-l}^{l}V_{k}, \qquad \mbox{and} \qquad V'_{(m)}=\bigoplus_{k=-m}^{m}V'_{k}.
\]

Under these notations, we have
\[V=V_{(r)}, \qquad \mbox{and} \qquad V'=V'_{(r+1)};\]
\[G=G_{(r)}, \qquad \mbox{and} \qquad G'=G'_{(r+1)};\]
\[(\omega ,\SY)=(\omega_{(r),(r+1)},\SY_{(r),(r+1)}).\]

Let $P'_{m}$ be the stabilizer of $V'_{m}$ in $G'_{(m)}$. We have $P'_{m}=M'_{(m)}N'_{m}$, where $N'_{m}$ is the
unipotent radical of $P'_{m}$, $M'_{(m)}=M'_{m}\times G'_{(m-1)}$, and $M'_{m}\cong \GL (V'_{m})$.
Set
\begin{equation*}
N'_{(m)}=N'\cap G'_{(m)}, \qquad U'_{(m)}=U'\cap N'_{(m)}, \qquad \mbox{and} \qquad U'_{m}=U'\cap N'_{m}.
\end{equation*}

Recall that $\chi'=:\chi_{\gamma'}$ is the character of $U'$ associated to the $\sl_{2}$-triple $\gamma'$, as in \eqref{defchi}. We set
\[\chi'_{(m)}=\chi'|_{U_{(m)}}, \ \ \text{ and } \ \ \chi'_{m}=\chi'|_{U'_{m}}.\]
Using these notations, we then have
\begin{equation}
\label{GandC}
U'=U'_{(r+1)}=U'_{(r)}U'_{r+1}, \ \ \text{ and } \ \ \chi '=\chi '_{(r+1)}=\chi '_{(r)}\chi '_{r+1}.
\end{equation}

To carry out the induction, we will apply the decomposition
\begin{equation}
\label{decom1}
V'_{(r+1)} = V'_{r+1}\oplus V'_{-r-1}\oplus V'_{(r)},
\end{equation}
followed by the decomposition
\begin{equation}
\label{decom2}
V_{(r)}=V_{-r}\oplus V_{r}\oplus V_{(r-1)}, \ \ \ r>0.
\end{equation}

Recall the following basic fact on the so-called mixed models \cite[Section 3]{HoPre2}: Suppose that $X\subseteq W$ is a totally isotropic subspace. Choose another totally isotropic subspace $Y$ supplementary to $X^{\perp}$, so that $X$ and $Y$ are paired non-degenerately by the symplectic form.
Let $W_0=(X\oplus Y)^{\perp}$. Then the smooth oscillator representation of $\cover{\Sp}(W)$ has a realization on the  Schwartz space of $X$ with values in the smooth oscillator representation of $\cover{\Sp}(W_0)$. This is known as the mixed model. For the case at hand, we will have the identification
\begin{equation}
\label{eq:polarization}
\begin{aligned}
\SY_{(r),(r+1)}&\cong \SS_{(r),r+1}\otimes \SY_{(r),(r)}\\
&\cong [\SS_{(r),r+1}\otimes \SS_{-r,(r)}]\otimes \SY_{(r-1),(r)}, \ \ \ \ r>0.
\end{aligned}
\end{equation}Here $\SS_{(r),r+1}$ denotes the Schwartz space on $\Hom (V_{(r)}, V'_{r+1})$, and $\SS_{-r,(r)}$ denotes the Schwartz space on $\Hom (V_{-r}, V'_{(r)})$. (Similar notations apply in the sequel.) This identification allows us to do the induction on $r$.

Instead of the covariant space $\SY_{{U'},\chi' }$, it is more convenient to work dually with the space of quasi-invariant distributions $(\SY')^{{U'},{\chi'}}$.  Our goal is to understand the space $(\SY_{(r),(r+1)}')^{U'_{(r+1)},\chi'_{(r+1)}}$. In view of \eqref{GandC}, we start with a distribution $\lambda \in (\SY_{(r),(r+1)}')^{U'_{r+1},\chi'_{r+1}}$. The basic idea is to use the realization of $\SY_{(r),(r+1)}$ given in \eqref{eq:polarization}, and solve the equations satisfied by $\lambda$ to find all such $\lambda$.

{\bf Step 1} (Searching for $T_r$):
We use the explicit realization of $\SY_{(r),(r+1)}$ given by the first isomorphism in \eqref{eq:polarization}. Thus $\lambda$ is identified as a $\SY_{(r),(r)}'$-valued distribution on $\Hom (V_{(r)},V'_{r+1})$. We refer the reader to \cite{KVVD} for generalities on vector valued distributions.
By using (derived) action of elements in the center of $\u'_{r+1}$ and the fact that $\Im (X')$ contains $V'_{r+1}$, one shows that $\lambda$ may be identified with a $\SY_{(r),(r)}'$-valued distribution that \emph{lives on} $\Hom_{\GNM}(V_{(r)},V'_{r+1})$ (i.e., supported on and having transverse order zero at all points of $\Hom_{\GNM}(V_{(r)},V'_{r+1})$), where
\begin{equation*}
\begin{aligned}
&\Hom_{\GNM}(V_{(r)},V'_{r+1})= \\
&\{T\in \Hom(V_{(r)},V'_{r+1})\, | \, \mbox{$TT^{\ast}=0$ and $T$
has maximal rank $\dim V'_{r+1}$} \}.
\end{aligned}
\end{equation*}
\[
\text{Assume that $\Hom_{\GNM}(V_{(r)},V'_{r+1})$ is non-empty.}\tag{A1}
\]
Then $\Hom_{\GNM}(V_{(r)},V'_{r+1})$ is a single orbit under the action of $G=G_{(r)}$. We may pick one representative of $\Hom_{\GNM}(V_{(r)},V'_{r+1})$
called $T_r$, with the following property: there exists a totally isotropic subspace $V_{r}$ of $V=V_{(r)}$ having the same dimension as $V'_{r+1}$, and
 \begin{equation}
 \label{eq:Tr}
 T_r|_{V_{(r-1)}\oplus V_{-r}}=0, \ \text{ and } \ T_r: V_r\rightarrow V'_{r+1} \, \text{is a linear isomorphism}.
 \end{equation}
 Here $V_{-r}$ is an isotropic subspace dual to $V_{r}$, and $V_{(r-1)} = (V_r\oplus V_{-r})^{\perp}$. Via this choice of $T_r$, $\lambda$ is identified as a $\SY_{(r),(r)}'$-valued distribution on $G$ in the standard way.

Let $P_{r}$ be the stabilizer of $V_{-r}$
in $G=G_{(r)}$. Then $P_{r}=M_{(r)}N_{r}$, where $N_{r}$ is the
unipotent radical of $P_{r}$, $M_{(r)}= M_{r}\times G_{(r-1)}$, and $M_{r}\cong \GL(V_{-r})$. Note that $N_{r}G_{(r-1)}$ is the stabilizer of $T_r$ in $G_{(r)}$.

{\bf Step 2} (Searching for $T_{-r}$):
Since we have the decomposition $V=V_{-r}\oplus V_{(r-1)}\oplus V_r$, we could now use the explicit realization of $\SY_{(r),(r+1)}$ given by the second isomorphism in \eqref{eq:polarization}. Thus $\lambda$ is identified as a $\SY_{(r-1),(r)}'$-valued distribution on $G\times \Hom (V_{-r},V'_{(r)})$. By the requirement under the (derived) action of all other elements (from a complementary subspace of the center of $\u'_{r+1}$) in $\u'_{r+1}$, one finds that $\lambda$ may be identified with a distribution
that \emph{lives on} $G\times \SA$, where $\SA \subseteq \Hom(V_{-r}, V'_{(r)})$ is the solution space of the following linear system $\SL$:
\[
\{S \in \Hom(V_{-r},V'_{(r)})\, | \,
\Tr ({R'}T_{r}S^{\ast})=\Tr({R'}{X'}),\, \forall \, {R'}\in \Hom(V'_{r+1},V'_{(r-1)}\oplus
V'_{-r})\},
\]
whose corresponding homogeneous system has $\Hom(V_{-r},V'_{-r})$ as the general solution.
\[
\text{Assume that the linear system $\SL$ has a non-empty solution}. \tag{A2}
\]
The assumption (A2) amounts to the existence of a particular solution $S=T_{-r}$ of $\SL$ so that
\begin{equation}
\label{eq:T-r}
T_{-r}\in \Hom (V_{-r},V'_{-r+1}), \text{ and } T_r T_{-r}^{\ast}=X' \text { on } V'_{r-1}.
\end{equation}
If this is the case, then the general solution of $\SL$ is given by
\[
\SA=T_{-r}+\Hom (V_{-r},V'_{r}).
\]
Figuratively, this means that we should be able to complete a commutative triangle:
\[
\begin{array}{c}   V_{r}   \\
\hspace{-1pt} {\scriptstyle T^{\ast}_{-r}} \nearrow  \searrow \hspace{1pt} {\scriptstyle T_{r}}\\
V'_{r-1}   \xrightarrow{X'} V'_{r+1}
\end{array}
\]

\begin{rem} We have presented the above discussions to be suitable for the situation where we are only given the nilpotent orbit $\Orb'$, but not $\Orb$. Thus we only have the decomposition \eqref{decom1}, but not the decomposition \eqref{decom2}. Nevertheless, the just concluded discussions allow us, first under the hypothesis (A1), to find $V_r$, $V_{-r}$, $V_{(r-1)}$ so that the decomposition \eqref{decom2} holds, and most crucially $T_r$. Then, under the hypothesis (A2), we further find $T_{-r}$. At the end of the induction, we will have $T$ and $X=T^*T$, and the various $V_i$'s will then carry their meanings as the eigenspaces of the neutral element of an $\sl_2$-triple $\gamma$.
\end{rem}

To summarize, if we assume both (A1) and (A2), then an element $\lambda \in (\SY_{(r),(r+1)}')^{U'_{r+1},\chi'_{r+1}}$ is identified as a distribution that \emph{lives on} $G\times [T_{-r}+\Hom(V_{-r},V'_{-r})]$.

Given a function $f\in \SY_{(r),(r+1)}\cong [\SS_{(r),r+1}\otimes
\SS_{-r,(r)}]\otimes \SY_{(r-1),(r)}$, define a new function
\[
f_{r}\in C^{\infty}(\cover{G_{(r)}};\SS_{-r,-r}\otimes \SY_{(r-1),(r)}), \]
given by
\[
f_{r}(g)(S)=[\omega_{(r),(r+1)}(g)f](T_{r},T_{-r}+S),\qquad g\in \cover{G_{(r)}}, S\in\Hom(V_{-r},V'_{-r}).
\]
Note that $T_r\in \Hom(V_r,V'_{r+1}) \subseteq \Hom(V_{(r)},V'_{r+1})$, and $T_{-r}+S\in \Hom(V_{-r}, V'_{(r)})$. This yields a map
\begin{equation*}
\begin{aligned}
\Psi_r: (\SY_{(r),(r+1)})_{U'_{r+1},\chi'_{r+1}}&\longrightarrow C^{\infty}(N_r\cover{G_{(r-1)}}\backslash \cover{G_{(r)}};\SS_{-r,-r}\otimes \SY_{(r-1),(r)}), \\
\qquad \qquad f&\longmapsto f_r.
\end{aligned}
\end{equation*}
Our discussions (with a little extra work) implies that the map $f\mapsto f_{r}$ induces a $\cover{G_{(r)}}\times N'_{r+1}\cover{G'_{(r)}}$-intertwining isomorphism
\begin{equation}
\label{eq:Psi_r}
\Psi_{r}:(\SY_{(r),(r+1)})_{U'_{r+1},\chi'_{r+1}} \longrightarrow
\SS(N_{r} \cover{G_{(r-1)}}\backslash \cover{G_{(r)}};\SS_{-r,-r}\otimes\SY_{(r-1),(r)}),
\end{equation}
where the right hand side is the Schwartz induction of an (explicitly determined) representation of $N_{r} \cover{G_{(r-1)}}$ on $\SS_{-r,-r}\otimes\SY_{(r-1),(r)}$, with an (explicitly determined) commuting action of $N'_{r+1}\cover{G'_{(r)}}$.

When $r=0$, a modification needs to be made. In this case, we still have $\SY_{(0),(1)}\simeq \SS_{(0),1}\otimes \SY_{(0),(0)}$. Instead of (A1) (or \eqref{eq:Tr}) and (A2) (or \eqref{eq:T-r}), we have the following replacement:
 \begin{equation}
 \label{eq:T0}
 \, \text{there exists a surjective $T_0: V_0\rightarrow V'_{1}$ such that } T_0T_0^*=X'|_{V'_{-1}},
 \end{equation}
as depicted in the following diagram:
\[
\begin{array}{c}   V_{0}   \\
\hspace{-1pt} {\scriptstyle T^{\ast}_{0}} \nearrow  \searrow \hspace{1pt} {\scriptstyle T_{0}}\\
V'_{-1}   \xrightarrow{X'} V'_{1}
\end{array}
\]
We may also choose $T_0$ so that $\Ker (T_0)\subseteq V_0$ is non-degenerate. Write $L$ for the isometry group of $\Ker (T_0)$. We will take \eqref{eq:T0} as the replacement of (A1) and (A2), for $r=0$.

Given a function $f\in \SY_{(0),(1)}\simeq \SS_{(0),1}\otimes \SY_{(0),(0)}$, define a new function $f_{0}\in
C^{\infty}(\cover{G_{(0)}}; \SY_{(0),(0)})$ by
\[
f_{0}(g)=[\omega_{(0),(1)}(g)f](T_{0}),\qquad g\in \cover{G_{(0)}}.
\]
A similar conclusion holds: the map $f\mapsto f_{0}$ induces a $\cover{G_{(0)}}$-intertwining isomorphism
\[
\Psi_{0}:(\SY_{(0),(1)})_{U'_{1},\chi'_{1}} \longrightarrow
\SS(\cover{L}\backslash \cover{G_{(0)}};\SY_{(0),(0)}),
\]
where the right hand side is the Schwartz induction of an (explicitly determined) representation $\cover{L}$ on $\SY_{(0),(0)}$
with an (explicitly determined) commuting action of $\cover{L'}$ (the isometry group of $(V')^{\gamma',1}$).

We summarize our discussions by the following
\begin{lem}\label{Model-NC}
\begin{enumerate}
    \item [(a).] For $(\SY_{(r),(r+1)})_{U'_{(r+1)},\chi'_{(r+1)}}$ to be non-zero, both (A1) and (A2) must be satisfied, for all $r$.
     \item[(b).] Assume both (A1) and (A2), for all $r$. Let $T:=\oplus _{k=-r}^{r}T_k$, where $T_i$, $T_{-i}$ ($1\leq i\leq r$) and $T_0$ are as in \eqref{eq:Tr}, \eqref{eq:T-r} and \eqref{eq:T0}. Then $TT^*=X'$. Moreover $T\in \Orb_{\gamma,{\gamma}'}$, where $\gamma $ is an $\sl_2$-triple containing $X:=T^*T$.
\end{enumerate}
\end{lem}

We shall apply the map $\Phi_r$ repeatedly. Recall from Section \ref{subsec:subspace} that $(\omega _0, \SY_0)$ is the smooth oscillator representation associated to the symplectic space $W_{0}=\bigoplus_{k=-r}^{r}
\Hom(V_{k},{V}'_{k})$. By a repeated use of the mixed models, we have the following tensor product decomposition:
\begin{equation*}
\label{Schitchi}
\SY_0 \simeq \SS_{-r,-r}\otimes
\SS_{-r+1,-r+1}\cdots\otimes \SS_{-1,-1}\otimes \SY_{(0),(0)}.
\end{equation*}
As before, we identify the space $\SY_0$ with
the space of Schwartz functions on
$\Hom(V_{-r},V'_{-r})\oplus \cdots \oplus
\Hom(V_{-1},V'_{-1})$ with values in $\SY_{(0),(0)}$. Now,
given $f\in \SY=\SY_{(r),(r+1)}$, set
\begin{equation}
\label{def:fr0}
f_{(r)}(g)(S_{-r},\ldots,S_{-1})=(\omega_{(r),(r+1)}(g) f)(T_{r},T_{-r}+S_{-r},\ldots,T_{1},T_{-1}+S_{-1},T_{0}),
\end{equation}
for all $g\in \cover{G}$, $S_{-k}\in \Hom(V_{-k},V'_{-k})$, $k=1,\ldots,r$.  Then the map

\begin{align}\label{PsiT}
\Psi_T=\Psi_0\circ \Psi_1\circ \cdots \circ \Psi _{r-1}\circ \Psi_{r}: \,\, \SY_{{U'},\chi_{{\gamma}'}}&\longrightarrow \SS(N\backslash \cover{G};\SY_0), \\
f&\longmapsto f_{(r)},\notag
\end{align}
induces a $\cover{G} \times N'$-intertwining isomorphism, where the right hand side of \eqref{PsiT} is the Schwartz induction of an (explicitly determined) representation of $N$ on $\SY_0$, with an (explicitly determined) commuting action of $N'$. By tracking the action of $N_{r} \cover{G_{(r-1)}} \times N'_{r+1}\cover{G'_{(r)}}$ on $\SS_{-r,-r}\otimes\SY_{(r-1),(r)}$ in the isomorphism $\Psi_r$ (for each $r$), it is routine to check that the action of $N\times N'$ on $\SY_0$ coincides with $\omega _0^T=\omega_0 \circ \phi _T$, where $\phi_T$ is given in \eqref{eq:extendJT2}.

We have thus proved Proposition \ref{prop:mainproposition}, with the explicitly defined map
$\Psi_T$ as in \eqref{PsiT}.

\subsection{Transition of generalized Whittaker models}
\label{subsec:Trans}

We are given a type I reductive dual pair $(G,G')\subset \Sp(W)$, and the smooth oscillator representation $(\omega, \SY)$ of $\cover{\Sp}(W)\ltimes \rH(W)$ attached to $\psi$. We are interested in knowing all generalized Whittaker models of $\Theta (\pi)$, where $\pi \in \Irr(\cover{G})$.

We have the following result on the transition of generalized Whittaker models. This is an extension of the main result of \cite{GZ}, which is the pure descent case (i.e., $\Ker (T)=0$, in which case $L$ is trivial and $M_X=M_{X,X'}$).

\begin{thm} [Gomez-Zhu \cite{Zh2}]
\label{thm:whittaker}
Let $\pi \in \Irr (\cover{G})$, and $\Orb' \subset \g'$ be a nilpotent orbit.
\begin{enumerate}
\item[(a).]
If $\Orb'$ is not in the image of the moment map $\varphi'$, then
\[\Wh_{\Orb'}(\Theta (\pi))=0.\]
\item[(b).]
If $\Orb'$ is in the image of the moment map $\varphi'$, let $\Orb = \nabla^{V'}_{V}(\Orb ')$ be its descent.  We have (as in \eqref{eq:alp} and \eqref{eq:stab})
\[
\begin{aligned}
&M_{X}\supset M_{{X}, {X}'}\times L, \qquad (\text{as a symmetric subgroup})\\
&M'_{X'}=M_{{X}, {X}'}\times L',
\end{aligned}
\]
and $(L,L')$ forms a dual pair of the same type as $(G,G')$. Then for any $\tau '\in \Irr(\cover{L'})$, we have
\[\Wh_{\Orb',\tau '}(\Theta (\pi))\simeq \Wh_{\Orb, \Theta (\tau')^{\vee}}(\pi ^{\vee})\]
as $\cover{M_{X, X'}}$-modules.  Here $\Theta (\tau')$ is the full theta lift of $\tau '$ with respect to the dual pair $(L,L')$, and the symbol $\vee$ indicates the dual in the category of Casselman-Wallach representations.
\end{enumerate}
\end{thm}

\begin{cor} \label{nonvanish1} We are in the setting of Theorem \ref{thm:whittaker}.
If $\Wh_{\Orb, \Theta (\tau')^{\vee}}(\pi ^{\vee})\ne 0$, for some $\tau '\in \Irr(\cover{L'})$, then
\[\Theta (\pi)\ne 0.\]
\end{cor}

\begin{rem} Using the formulation of Section
\ref{dual pair} for type II dual pairs, i.e., by considering
 \[
 (\rD \times \rD, ((x,y)\mapsto (y^{\natural},x^{\natural})),
\]
where $\rD$ is a division algebra with involution $\natural$, and $\rD\times \rD$-valued Hermitian or skew-Hermitian forms, Theorem \ref{thm:whittaker} (statement and proof) should carry over verbatim for type II dual pairs. \end{rem}

\begin{rem} When the dual pair $(G, G')$ is in the stable ranger with $G$ the smaller member, every nilpotent orbit $\Orb \subset \g$ is the pure descent (Definition \ref{def:descent}) of a nilpotent orbit $\Orb' \subset \g'$. Applying Corollary \ref{nonvanish1} to the zero orbit of $\g$ and the trivial representation of $\cover{L'}$, we immediately conclude the nonvanishing of theta liftings in the stable range. Namely if $\pi $ is a smooth irreducible genuine representation of $G$, then $\Theta (\pi)\ne 0$. Of course this fact (non-vanishing of theta liftings in stable range) was known earlier
 (\cite[Section III.4]{KuNot} and \cite[Theorem 1]{PP}).

\end{rem}

\begin{proof}[Proof of Theorem \ref{thm:whittaker}] Part (a) follows from Lemma \ref{Model-NC}, and so we just prove Part (b).

By Proposition \ref{prop:mainproposition} and Lemma \ref{lemma:Tisomorphism}, we have
\[
\SY_{{U'},\chi_{{\gamma}'}} \simeq \SS(N\backslash \cover{G}; \SS_{\check{\gamma}}\otimes \SS_{\gamma'}\otimes \omega_{\gamma, \gamma'}) \simeq \SS(N\backslash \cover{G}; \SS_{\check{\gamma}})\otimes \SS_{\gamma'}\otimes \omega_{\gamma, \gamma'},
\]
as $\cover{G}\times \cover{M_{X,X'}}N' \times \cover{L'}$-modules. Here the action of $\cover{M_{X,X'}}N' \times \cover{L'}$ is induced by the representation of $(\cover{\Delta M_{X,X'}}\ltimes (N\times {N'})) \times (\cover{L}\times \cover{L'})$ on $\SS_{\check{\gamma}}\otimes \SS_{\gamma'}\otimes \omega_{\gamma, \gamma'}$.

We take the $(L',\tau')$-isotypic quotient on both sides to get
\begin{equation}
    \label{refinecovariant}
\SY_{({L'U'},\tau'\otimes \chi_{\gamma'})}\simeq \SY_{({U'},\chi_{\gamma'}),(L',\tau')}\simeq \SS(N\backslash \cover{G};\Theta (\tau')\otimes \SS_{\check{\gamma}})\otimes (\tau' \otimes \SS_{\gamma'}),
\end{equation}

Now we take the $(G, \pi)$-isotypic quotient on both sides of \eqref{refinecovariant}.
On the left hand, we get
\begin{equation*}
\label{2qq2}
\SY_{(G,\pi), ({L'U'},\tau'\otimes \chi_{\gamma'})}\simeq \pi \otimes \Theta(\pi)_{(L'U',\tau'\otimes \chi_{\gamma'})}.
\end{equation*}

On the right hand side, we get (c.f. \cite[Proposition 4.9]{GZ})
\[\SS(N\backslash \cover{G};\Theta (\tau')\otimes \SS_{\check{\gamma}})_{(G, \pi)}\otimes (\tau' \otimes \SS_{\gamma'})\simeq \pi \otimes \Wh_{\Orb, {\Theta (\tau')}^{\vee}}(\pi^{\vee})\otimes (\tau' \otimes \SS_{\gamma'}).
\]

Noting that
\[\Wh_{\Orb', \tau'}(\Theta(\pi))=\Hom_{L'N'}(\Theta(\pi),\tau'\otimes \SS_{\gamma'})\simeq \Hom_{L'N'}(\Theta(\pi)_{(L'U',\tau'\otimes \chi_{\gamma'})}, \tau'\otimes \SS_{\gamma'}),\]
we obtain the required isomorphism of $\cover{M_{X,X'}}$-modules:
\[\Wh_{\Orb',\tau '}(\Theta (\pi))\simeq \Wh_{\Orb, \Theta (\tau')^{\vee}}(\pi ^{\vee}).\]
\end{proof}

\section{Correspondence of invariants (II): associated cycles}
\label{sec:AC}

For this section, the local field $\rF$ will be $\R$. Thus we will concern ourselves with representations of reductive Lie groups. A fundamental invariant of representations is the associated cycle (defined by Vogan in \cite{Vo89}). We remark that, by the well-known result of Schmid-Vilonen  \cite{SV}, the weak associated cycle (for its definition, see \cite[Section 4]{AV} or \eqref{weak}) agrees with the wavefront cycle under the Kostant-Sekiguchi correspondence. (The latter is defined by Barbasch-Vogan \cite{BVwf} or Howe \cite{HoWav}; they are shown to be the same by Rossmann \cite{Ro}.) The main focus of this section is a recent result of Barbasch, Ma, Sun and the author \cite{BMSZ}, which bounds associated cycles of the full theta lift.
Related earlier works on transition of associated cycles include \cite{NOTYK,NZ,LM}.

\subsection{The associated cycle map}\label{subsec:AC}

We first recall some basic notions from the theory of associated varieties (\cite{Vo89}).

Let $\bfG$ be a complex reductive Lie group and $\g$ be its Lie algebra. Let $G$ be a real form of $\bfG$. We will be concerned with Casselman-Wallach representations of $G$ and their algebraic counterparts, namely finite length $(\g,\bfK)$-modules. Here $\bfK$ is the complexification of a maximal compact subgroup $K$ of $G$. For the theory of dual pairs, we need to work with genuine Casselman-Wallach representations of $\cover{G}$, and finite length $(\g,\cover{K})$-modules. Here $G$ is a member of the reductive dual pair $(G,G')\subseteq \Sp (W)$, as in earlier sections. We will need to allow this minor extension in our discussion of $(\g,\bfK)$-modules,
by taking $\bfK$ to be the complexification of $\cover{K}$.

We have  decompositions
\[
\g=\kk \oplus \p
\quad \text{ and } \quad
\g^*=\kk^*\oplus \p^*,
\]
where $\mathfrak k$ is the Lie algebra of $\bfK$, and $\p$ is the orthogonal complement of $\mathfrak k$ in $\g$ under the trace form.
We identify $\g^*$ with $\g$ and $\p^*$ with $\p$ by using the trace form.
Denote by $\Nil_{\bfG}(\g)=\Nil_{\bfG}(\g^*)$ (or $\Nil(\g)$ in short) the set of nilpotent $\bfG$-orbits in $\g=\g^*$, and by
$\Nil_{\bfK}(\p)=\Nil_{\bfK}(\p^*)$ (or $\Nil(\p)$ in short) the set of nilpotent $\bfK$-orbits in $\p=\p^*$.

Suppose that $\CO \in \mathrm{Nil}_{\bfG}(\g)$.
It is well-known \cite{KoR} that under the adjoint action of $\bfK$, the complex variety $\CO\cap \p$ is a union of finitely many orbits, each of dimension $\frac{\dim_\C \CO}{2}$.  For any $\bfK$-orbit
$\SO\subset \CO\cap \p$, let $\CK (\SO)$ denote the Grothendieck group of the category of ${\bfK}$-equivariant algebraic vector bundles on $\SO$,
and $\CK^+ (\SO)$ the submonoid generated by the ${\bfK}$-equivariant algebraic vector bundles.
Taking the isotropy representation at a point $\mathbf e\in \SO$ yields an identification
\begin{equation*}
  \CK (\SO) =\CR({\bfK}_{\mathbf e}),
\end{equation*}
where the right hand side denotes the Grothendieck group of the category of algebraic representations of the stabilizer group ${\bfK}_{\mathbf e}$.

Put
\begin{equation*}
  \CK (\CO):=\bigoplus_{\SO\textrm{ is a $\bfK$-orbit in $\CO\cap \p$}}
  \mathrm \CK (\SO),
\end{equation*}
and
\begin{equation*}
  \CK^+(\CO):=\bigoplus_{\SO\textrm{ is a $\bfK$-orbit in $\CO\cap \p$}}
  \mathrm \CK^+(\SO).
  \end{equation*}

 There is a partial order $\preceq $ on $\CK (\CO)$ defined by
\[
  \CE_1\preceq \CE_2\Leftrightarrow \CE_2-\CE_1\in \CK^+(\CO),
  \qquad \CE_1, \CE_2\in \CK (\CO).
\]

We now recall the definition of associated cycles (\cite{Vo89}). We introduce a general notation: for an affine complex algebraic variety $Z$, write $\C[Z]$ for the algebra of regular functions on $Z$.

Denote by $\overline \CO$ the Zariski closure of $\CO$. Given a finitely generated $(\C[\overline \CO\cap \p], \bfK)$-module $\sigma^\circ$, for each $\bfK$-orbit $\SO\subset \CO\cap \p$, the set
\[
 \CE_{\sigma^\circ, \SO}:= \bigsqcup_{\mathbf e\in \SO} \C_\mathbf e\otimes_{\C[\overline \CO\cap \p]}\sigma^\circ
\]
is naturally a $\bfK$-equivariant algebraic vector bundle over $\SO$, where $\C_\mathbf e$  denotes the complex number field $\C$ viewed as a $\C[\overline \CO\cap \p]$-algebra via the evaluation map at $\mathbf e$. We define $\mathrm{AC}_\SO(\sigma^\circ)\in \CK(\SO)$ to be the Grothendieck group element associated to $\CE_{\sigma^\circ, \SO}$. Define the  associated cycle of $\sigma^\circ$ to be
\[
  \mathrm{AC}_\CO(\sigma^\circ):=\sum_{\SO\textrm{ is a   $\bfK$-orbit in $ \CO\cap \p$}} \mathrm{AC}_\SO(\sigma^\circ)\in \CK (\CO).
\]
For later use, we also define the weak associated cycle of $\sigma^\circ$ (see \cite{AV}) to be the formal direct sum:
\begin{equation}\label{weak}
\mathrm{AC}_\CO(\sigma^\circ)^{\mathrm{weak}}:=\sum_{\SO\textrm{ is a   $\bfK$-orbit in $ \CO\cap \p$}} m_{\sigma^\circ, \SO}\SO,
\end{equation}
where $m_{\sigma^\circ, \SO}$ is the rank of the vector bundle $\CE_{\sigma^\circ, \SO}$.

Write $I_{\overline \CO\cap \p}$ for the radical ideal of $\C[\p]$ corresponding to the closed subvariety $\overline \CO\cap \p$.
Every $(\C[\overline \CO\cap \p], \bfK)$-module is clearly a $(\C[\p], \bfK)$-module.
Conversely, for each  $(\C[\p], \bfK)$-module $\sigma$,
\[
  \frac{(I_{\overline \CO\cap \p})^i\cdot \sigma}{(I_{\overline \CO\cap \p})^{i+1}\cdot \sigma}\qquad (i\in \NN)
\]
is naturally a $(\C[\overline \CO\cap \p], \bfK)$-module.
We say that $\sigma$ is $\CO$-bounded if
\begin{equation}\label{eq:Obound}
(I_{\overline \CO\cap \p})^i \cdot\sigma=0\qquad
\textrm{
for some $i\in \NN$.}
\end{equation}
When $\sigma$ is finitely generated, this is equivalent to saying that  the support of $\sigma$ is contained in $\overline{\CO}\cap \p$.

\begin{defn}
Let $\sigma$ be a finitely generated $(\C[\p], \bfK)$-module, which is $\CO$-bounded. The associated cycle of $\sigma$ is defined to be
\[
  \mathrm{AC}_\CO(\sigma):=\sum_{i\in \NN}\mathrm{AC}_\CO\left(\frac{(I_{\overline \CO\cap \p})^i\cdot \sigma}{(I_{\overline \CO\cap \p})^{i+1}\cdot \sigma}\right)\in  \CK (\CO).
\]
The weak associated cycle of $\sigma$ is defined similarly, in the obvious way.
\end{defn}

The assignment $\mathrm{AC}_\CO$ is additive in the following sense: the equality
\[
 \mathrm{AC}_\CO(\sigma)= \mathrm{AC}_\CO(\sigma')+ \mathrm{AC}_\CO(\sigma'')
\]
holds, for any exact sequence $0\rightarrow \sigma'\rightarrow \sigma\rightarrow \sigma''\rightarrow 0$ of finitely generated $\CO$-bounded $(\C[\p], \bfK)$-modules.

Let $\Pi$ be a $(\g,{\bfK})$-module of finite length. We say $\Pi$ is \emph{$\CO$-bounded} if the associated variety of the annihilator ideal $\Ann(\Pi)$ in $\CU(\g)$ (the universal enveloping algebra of $\g$) is contained in $\overline{\CO}$.

Pick a filtration
\begin{equation}\label{goodf}
\CF:\quad  \cdots\subset \Pi_{-1}\subset \Pi_0\subset \Pi_1\subset \Pi_2\subset \cdots
\end{equation}
of $\Pi$, which is good in the following sense:
\begin{itemize}
\item for each $i\in \Z$, $\Pi_i$ is a finite-dimensional  $\bfK$-stable subspace of $\Pi$;
\item $\g \cdot \Pi_i\subset \Pi_{i+1}$ for all $i\in \Z$, and $\g \cdot \Pi_i=\Pi_{i+1}$ when  $i\in \Z$ is sufficiently large;
\item $\bigcup_{i\in \Z} \Pi_i=\Pi$, and $\Pi_i=0$ for some $i\in \Z$.
\end{itemize}
Then the grading
\[
  \mathrm{Gr}(\Pi ):= \mathrm{Gr}(\Pi,\CF):=\bigoplus_{i\in \Z} \Pi_i/\Pi_{i+1}
\]
is naturally a finitely generated  $(\rS(\p), \bfK)$-module, where $\rS(\p)$ denotes the symmetric algebra of $\p$. Recall that we have identified $\p$ with its dual space $\p^*$, and consequently $\rS(\p)$ with $\C[\p]$. Therefore $\mathrm{Gr}(\Pi)$ is a $(\C[\p], \bfK)$-module. It is finitely generated since $\Pi$ is finitely generated, being of finite length.

We record the following lemma, which is a direct consequence of   \cite[Theorem 8.4]{Vo89}.

\begin{lem}\label{Boundness}
Let $\Pi$ be a $(\g, \bfK)$-module  of finite length,  and $\CF$ a  good filtration of $\Pi$. Then $\Pi $ is $\CO$-bounded if and only if $\mathrm{Gr}(\Pi, \CF)$ is $\CO$-bounded as a $(\C[\p], \bfK)$-module.
\end{lem}

\begin{defn}
Let $\Pi$ be a $(\g, \bfK)$-module of finite length, which is $\CO$-bounded. The associated cycle of $\Pi$ is defined to be
\[
   \mathrm{AC}_\CO(\Pi):= \mathrm{AC}_\CO(\mathrm{Gr}(\Pi, \CF))\in   \CK (\CO).
\]
This is independent of the good filtration $\CF$ in \eqref{goodf}.
\end{defn}

The associated cycle of $\Pi$ is a fundamental invariant attached to $\Pi$. It is additive in the following sense: the equality
\[
 \mathrm{AC}_\CO(\Pi)= \mathrm{AC}_\CO(\Pi')+ \mathrm{AC}_\CO(\Pi'')
\]
holds for any exact sequence
$0\rightarrow \Pi'\rightarrow \Pi\rightarrow \Pi''\rightarrow 0$ of  $\CO$-bounded  $(\g, \bfK)$-modules of finite length.

Denote $\CM_{\CO}(\g, {\bfK})$ the category of $\CO$-bounded finite length $(\g,{\bfK})$-modules, and write
$\CK_{\CO}(\g,{\bfK})$ for its Grothendieck group. The associated cycle map on $\CM_{\CO}(\g, {\bfK})$ thus induces a canonical homomorphism
\[\label{def:Ch}
  \AC_{\CO}: \CK_{\CO}(\g, \bfK)\rightarrow \CK (\CO).
\].

For a Casselman-Wallach representation of ${G}$ (which is $\CO$-bounded), we define its associated cycle as the associated cycle of its Harish-Chandra module.

\subsection{Algebraic theta lifting and moment maps}
\label{subsec:MMM}

We continue with the notation of Sections \ref{dual pair} and \ref{duality}. We shall assume that the homomorphisms $\widetilde {G}\rightarrow G$ and $\widetilde {G'}\rightarrow G'$ are finite fold covering maps.
Fix a smooth oscillator representation $(\omega, \SY)$ of $(\cover{G}\times \cover{G'})\ltimes \rH(W)$, where $G=G(V)$ and $G'=G(V')$.

We fix a Cartan involution of $\Sp(W)$ which preserves both $G$ and $G'$. (Such a Cartan involution exists and is unique up to conjugation by $G\times G'$.) We then have a compatible choice of maximal compact subgroups of $\Sp(W)$, and $G$ and $G'$, denoted by $U$, and $K$ and $K'$, respectively. Let $(\Omega, \SP)$ be the Harish-Chandra module of $(\omega, \SY)$ (with respect to $U$), which is naturally a $(\g\times \g', \widetilde{K}\times \widetilde{K'})$-module. Here and elsewhere in this section, $\g$ and $\g'$ denote the complexified Lie algebras of $G$ and $G'$, respectively, and $\widetilde {K}\subset \widetilde G$ and $\widetilde{K'}\subset \widetilde{G'}$ are respectively the preimages of $K$ and $K'$.

\begin{defn}
For a genuine $(\g,\cover{K})$-module $\Pi$ of finite length,
define
\[
  \check{\Theta}_{V}^{V'}(\Pi ):= \left(\Omega \otimes \Pi \right)_{(\g, {K})},\qquad
  \text{(the coinvariant space).}
\]
(Recall that the $({\g, {K}})$-coinvariant space is the maximal quotient on which $({\g, {K}})$ acts trivially.) The $(\g',\wt{K'})$-module $\check{\Theta}_{V}^{V'}(\Pi)$, or $\check{\Theta}(\Pi)$ in short,
is genuine and of finite length \cite{Ho89}.
\end{defn}

\begin{rem} If $\pi$ is a smooth irreducible genuine representation of
$\cover{G}$, and let $\Pi$ denote the Harish-Chandra module of $\pi $. Then
\[
  \check{\Theta}(\Pi )\simeq \check \Theta (\pi )^{\text{alg}},\]
  where $ \check \Theta (\pi )= \Theta (\pi ^{\vee})$, the full theta lift of $\pi ^{\vee}$ (the contragredient representation of $\pi$), and the superscript alg denotes the operation of taking the underlying Harish-Chandra module.
  \end{rem}

Write $\SV:=V\otimes_\R \C$, which is a right $\rD\otimes_\R \C$-module. The $\R$-bilinear map $
\langle\cdot , \cdot  \rangle_{V}: V\times V\rightarrow \rD
$ extends to a $\C$-bilinear map
$
\langle\cdot , \cdot  \rangle_{\SV}: \SV\times \SV \rightarrow \rD\otimes_\R \C.
$
Write $\bfG$ for the isometry group of $(\SV, \langle\cdot , \cdot  \rangle_{\SV})$, which is the complexification of $G$. (As an elementary example, $V=\C^{p}$, equipped with the standard positive definite Hermitian form $\langle\cdot , \cdot  \rangle_{V}$, so that $G=\rU_{p}$. Then $\SV$, equipped with the form $\langle\cdot , \cdot  \rangle_{\SV}$ (the complex extension) may be viewed as a Hermitian $\rD$-space for $(\rD, \natural):= (\C\times \C, ((x,y)\mapsto (y,x))$. In this way, the complexification $\bfG$ is naturally identified with $\GL_{p}(\C)$.) Write $\bfK$ and $\widetilde{\bfK}$ for the complexifications of the compact groups $K$ and $\widetilde K$ respectively.
The space $\SV'$ and the groups $\bfG'$, $\bfK'$ and $\widetilde{\bfK'}$ are similarly defined.
We identify $\g$ with its dual space $\g^*$ by using the trace form
\[
  \g\times \g\rightarrow \C, \quad (x,y)\mapsto \textrm{the trace of the $\C$-linear endomorphism $xy: \SV\rightarrow \SV$}.
\]
Likewise, $\g'$ is identified with ${\g'}^*$.

Put $\SW=W\otimes_\R \C$, where $W=\Hom_{\rD}(V,V')$. Then $\SW=\Hom_{\rD\otimes_\R \C}(\SV,\SV')$, to be viewed as right $\rD\otimes_\R \C$-modules. As complex vector spaces, $\SW=\Hom (\SV,\SV')$, where $\SV$ (resp. $\SV'$) is an $\epsilon $-symmetric  (resp. $\epsilon '$-symmetric) bilinear space, with $\epsilon \epsilon '=-1$, and so we have a complex dual pair $(\bfG,\bfG')\subseteq \Sp(\SW)$.

We have the moment maps in this setting: (\cite{KP,DKPC})
\begin{equation}\label{MMG}
\begin{diagram}
 & &  \SW & &   \\
&  \ldTo^{\SM} & &\rdTo^{\SM'} &  \\
\g&  & & &{\g}'
\end{diagram}
\end{equation}
that are given by
  \[\SM (\phi)= \phi^{*}\phi, \ \ \text{ and } \ \ \SM '(\phi)=\phi \phi ^{*}.\]
  Here $\phi ^{*}$ denotes the adjoint map defined similarly as in \eqref{adj}.

  We recall the notion of check theta lift of nilpotent orbits for a complex dual pair. By \cite[Theorem 1.1]{DKPC}, for any $\CO\in \Nil_{\bfG}(\g)$, $\SM'(\SM^{-1}(\bCO))$ equals the closure of a
 unique nilpotent orbit $\CO' \in \Nil_{\bfG'}(\g')$. We call $\CO'$ the \emph{check theta lift} of $\CO$, and we write
 \begin{equation}
 \label{def:LC}
  \CO'=\check \bftheta_{\SV}^{\SV'}(\CO).
 \end{equation}

In the case at hand, the complexified spaces $\SV$, $\SV'$, and $\SW$ are complex formed spaces with certain additional structures, called classical spaces of signatures in \cite[Section 4]{BMSZ}), which we will briefly review. (The identification of these additional structures will facilitate a more uniform treatment, for example, on the definition of the double fiberation of moment maps in the symmetric space setting later in this section, and on the growth of Casselman-Wallach representations in Section \ref{CassWall}.)

A classical space with signature consists of $(\SV, \la\,,\,\ra, J,L)$, where $\SV$ is a complex vector space with a non-degenerate bilinear form $\la\,,\,\ra$ (either symmetric or skew-symmetric), and so determines a complex classical group $\bfG=\bfG(\SV)$, $J$ is a conjugate linear automorphism of $\SV$ which determines a real or quaternionic structure on $\SV$, and $L$ a linear automorphism of $\SV$ which determines choice of a maximal compact subgroup of the real classical group $\bfG^J$ (the centralizer of $J$ in $\bfG$), satisfying certain compatibility conditions. See \cite[Lemma 4.1]{BMSZ}. A classical space with signature may be classified by the classical signature, which is a triple  $\mathsf s=(\star, p,q)\in  \{B,C,D, \widetilde C, C^*, D^*\}\times \NN\times \NN$ such that
\[
\begin{cases}
  p+q\textrm{ is odd },&\quad \textrm{ if  $\star=B$;} \smallskip \\
   p+q\textrm{ is even },&\quad \textrm{ if  $\star=D$;} \smallskip \\
 p=q,&\quad \textrm{ if  $\star\in \{C,\widetilde C, D^*\}$;} \smallskip \\
\textrm{both $p$ and $q$ are even} ,&\quad \textrm{ if  $\star=C^*$.} \smallskip \\
 \end{cases}
\]
The real classical group $\bfG^J$ corresponding to $\mathsf s=(\star, p,q)$ is isomorphic with
\[
 \left\{
     \begin{array}{ll}
         \rO(p, q), &\hbox{if $\star=B$ or $D$}; \smallskip\\
            \Sp_{2p}(\R), &\hbox{if $\star\in \{C,\widetilde C\}$}; \smallskip\\
                   \rO^*(2p), &\hbox{if $\star=D^*$}; \smallskip \\
          \Sp(\frac{p}{2}, \frac{q}{2}), &\hbox{if $\star=C^*$}.\\
            \end{array}
   \right.
\]

Set
\begin{equation}\label{defg}
G= \left\{
     \begin{array}{ll}
       \textrm{the metaplectic double cover of $\bfG^{J}$}, \quad &\hbox{if $\star=\widetilde C$}; \smallskip\\
            \bfG^{J},  \quad  &\hbox{otherwise}.\\
            \end{array}
   \right.
\end{equation}
The linear automorphism of $L$ of $\bfV$ satisfies $L^2=\dot \epsilon$, where
\begin{equation*} \label{epsilond}
\dot \epsilon:=\begin{cases}
  1,&\quad \textrm{ if  $\star\in \{B,D,C^*\}$;} \smallskip \\
 -1,&\quad \textrm{ if  $\star\in \{C,\widetilde C, D^*\}$ }.
 \end{cases}
\end{equation*}
Denote by $\bfG^{L}$ the centralizer of $L$ in  $\bfG$, and $K$ the inverse image of  $\bfG^{L}$  under the natural homomorphism $G\rightarrow \bfG$, which is a maximal compact subgroup of $G$. Write $\bfK$ for the complexification of  $K$, which is a reductive complex linear algebraic group. The conjugation action of $L$ decomposes $\g$ into $(\pm 1)$-eigenspaces:
\[\g ={\mathfrak k}\oplus \p, \qquad (\text{the complex Cartan decomposition}).
\]

As in \cite[Section 4]{BMSZ}, the classical space structures of $\SV$ and $\SV'$ give rise to a classical space structure on $\SW$, which is of type $\star =C, \widetilde C$, In this way, one obtains in particular a compatible choice of datum $L$, $L'$ and $\SL$ (resp. for $\SV$, $\SV'$ and $\SW$), with $\SL ^2=-1$. The $\lambda $-eigenspace of $\SL$, written as $\SW_{\lambda}$, is characterized by the following:
\begin{equation*}\label{eq:eigen}
 \phi \in \SW_{\lambda} \smallskip \text{ if and only if } \smallskip L\phi ^{*}=\lambda \dot \epsilon \phi ^* L', \qquad (\lambda ^2=-1).
 \end{equation*}

We fix a choice of $\sqrt{-1}$, and decompose
\begin{equation}\label{SX}
\SW= \SX\oplus \SY
\end{equation}
where $\SX$ and $\SY$ are $\sqrt{-1}$ and $-\sqrt{-1}$ eigenspaces of $\SL$, respectively. Thus we have
\begin{equation}\label{eq:eigenX}
 \phi \in \SX \smallskip \text{ if and only if } \smallskip L\phi ^{*}=\sqrt{-1}\dot \epsilon \phi ^* L'.
 \end{equation}
It is easy to check that $\SM (\SX)\subset \p$ and $\SM' (\SX)\subset \p'$. Restriction on $\SX$ thus induces a pair of algebraic maps (\cite{DKPR1, NOZ}):
\begin{equation}\label{MMP}
\begin{diagram}
 & &  \SX & &   \\
&  \ldTo^{M:=\SM|_{\SX}} & &\rdTo^{M':=\SM'|_{\SX}} &  \\
\p&  & & &{\p}'
\end{diagram}
\end{equation}
which are also called moment maps. They are both $\bfK\times \bfK'$-equivariant. Here $\bfK'$ acts trivially on $\p$, $\bfK$ acts trivially on $\p'$, and all the other actions are the obvious ones.

\begin{rem} Daszkiewicz, Kraskiewicz and Przebinda have made detailed investigations of the moment maps in \eqref{MMG} and \eqref{MMP} and the resulting nilpotent orbit correspondence, for which we refer the readers to \cite{DKPC, DKPR1,DKPR2}.
\end{rem}

We assume that the choice of $\SX$ in \eqref{SX} is compatible with $\Omega$, as follows.
As a module for the Lie algebra of $\rH(W)$, $\Omega$ is the submodule of $\omega$ generated by $ \omega^{\SX} $ (the invariant space of $\SX$).
We have the following bound of the associated variety of $\check{\Theta}(\Pi)$. A more refined statement on the bound of the associated cycles will be discussed in Section \ref{BoundAC}.

\begin{thm}[{\cite[Theorem B and Corollary E]{LM}}]\label{cor:Cbound}
For any $(\g,K)$-module $\Pi$ of finite length, we have
\[
\AV(\check{\Theta}(\Pi)) \subset M'(M^{-1}(\AV(\Pi))).
\]
In particular, if $\Pi$ is $\CO$-bounded for a nilpotent
orbit $\CO\in \Nil_{\bfG}(\g)$, then $\check{\Theta}(\Pi)$ is
$\check \bftheta_{\SV}^{\SV'}(\CO)$-bounded.
\end{thm}

\subsection{Descent of $\bfK$-orbits in $\p$}
\label{DesNil}

We start with the complex setting, so we have $\SW=\Hom (\SV,\SV')$, and a complex dual pair $(\bfG,\bfG')\subseteq \Sp(\SW)$.
Suppose that we have a nilpotent $\bfG'$-orbit $\CO'$ in $\g'$, which is contained in the image of the moment map $\SM'$. As a special case of the results proved in Section \ref{subsec:GD}, we can define the descent of $\CO'$, which is a nilpotent $\bfG$-orbit in $\g$. To recall, denote
\begin{equation*}
\SW^{su} := \{\phi \in \SW | \  \, \mbox{$\Im (\phi^{\star})$ is non-degenerate with respect to $\la\,,\,\ra_{\SV}$ }\}.
\end{equation*}
Then $(\SM')^{-1}(\CO')\cap \SW^{su}$ is a single $\bfG\times \bfG'$-orbit. This is in fact true for any $\bfG'$-orbit $\CO'$ in the image of the moment map $\SM'$, without requiring $\CO'$ to be nilpotent. C.f. Proposition \ref{prop:descent} (which is proved below). Write
\begin{equation}\label{not:descent}
 {\nabla } (\CO'):=\textrm{the image of the set $(\SM')^{-1}(\CO')\cap \SW^{su}$ under the moment map  $\SM$.}
\end{equation}
This is a $ \bfG$-orbit in $\g$, which is called the descent of $\CO'$. When there exists a $\phi \in (\SM')^{-1}(\CO')\cap \SW^{su}$ such that $\phi^{\star}$ is surjective (equivalently $\phi$ is injective), we will refer to it as the pure descent case (similar to Definition \ref{def:descent} in the Lie algebra setting). We call $ {\nabla } (\CO')$ the pure descent of $\Orb '$, and we write
\begin{equation}\label{not:descent2}
\CO=\nabla_{\text{pure}}(\CO').
\end{equation}

Now we are in the $(\g,\bfK)$-setting of Section \ref{subsec:MMM}. Put
\[\SX^{su} := \SW^{su}\cap \SX.\]

The following proposition is \cite[Lemma 4.3]{BMSZ}. We include a proof for the reader's convenience.

\begin{prop}
  \label{prop:descent}
Suppose that $\SO'$ is a $\bfK'$-orbit in $\p'$, which is in the image of the moment map $M'$. Then the set
\begin{equation}\label{kkp}
  (M')^{-1}(\SO')\cap \SX^{su}
\end{equation}
is a single $\bfK\times \bfK'$-orbit. Moreover, for every $\phi$ in $(M')^{-1}(\SO')\cap \SX^{su}$, there is an exact sequence of algebraic groups:
  \begin{equation}
    \label{exact}
  1\rightarrow  \bfK_{\phi}\rightarrow (\bfK\times \bfK')_{\phi}\xrightarrow{\textrm{the projection to the second factor}} \bfK'_{\bfe'}\rightarrow 1,
\end{equation}
where $e'=M'(\phi)$, and a subscript indicates the stabilizer group.
\end{prop}

\begin{rem} For any $\phi$ in $(M')^{-1}(\SO')\cap \SX^{su}$, the stabilier group $\bfK_{\phi}$ is reductive.
\end{rem}

Write
\[
  \SO:=\nabla (\SO'):=\textrm{the image of the set \eqref{kkp} under the moment map  $M$.}
\]
This is a $ \bfK$-orbit in $\p$, which is called the descent of $\SO'$. It is clear that $\SO\in \Nil_{\bfK}(\p)$ if $\SO'\in \Nil_{\bfK'}(\p')$.

\begin{proof}
 When the rank of some (and hence every) element in $\SO'\subset \Hom(\SV', \SV')$
 is equal to $\dim \SV$,
  every element $\phi\in (M')^{-1}(\SO')\subset \Hom (\SV, \SV')$ has the property that $\phi^{\star}$ is surjective and hence belongs to  $\SX^{su}$.
In this case, the proposition is proved in \cite[Lemma 13]{Oh}.

 In general, take an element  $\psi \in (M')^{-1}(\SO')$. Then $\Im (\psi^{\star})$ is a $L$-stable subspace of $\SV$.
 Fix an $L$-stable decomposition
 \[
  \Im(\psi ^{\star}) = \SV_1 \oplus \SN_1
 \]
 such that the form $\inn{}{}_{\SV}$ is non-degenerate on $\SV_1$ and zero on
 $\SN_1$. Define $\phi$ by the requirement $\phi^{\star}=p_1\circ \psi^{\star}$, where $p_1$ is the projection map to $\SV_1$.
 Then $\phi$ belongs to $(M')^{-1}(\SO')\cap \SX^{su}$. This proves that the set
 $(M')^{-1}(\SO')\cap \SX^{su}$ is nonempty.

For an element $\phi\in (M')^{-1}(\SO')\cap \SX^{su}$, write $e' :=M'(\phi)=\phi \phi^* $.
As $\Im(\phi^*)$ is non-degenerate, we see that $\Ker ({e}')\subset \Ker (\phi^*)$. Obviously $\Ker (\phi^*)\subset \Ker ({e}')$, and so $\Ker (\phi^*)=\Ker ({e}')$, which is $L'$-stable, and $\phi |_{\Im(\phi^*)}$ is a linear bijection onto $\Im({e}')$.

Let $\SV(\SO')$ denote the set of $L$-stable
   non-degenerate subspaces $\SV''$ of $\SV$ such that
  \[
     \dim (\SV'')_{\lambda} = \dim \left(\Im(e')\right)_{\sqrt{-1} \dot \epsilon \lambda}
  \qquad \textrm{for all $\lambda \in \C$}.
\]
Here a subscript complex number on the left (resp. right) indicates the corresponding eigenspace of $L$ (resp. $L'$). Note that these eigenspaces are zero unless $\lambda^2=\dot \epsilon$.

 In view of \eqref{eq:eigen}, $\Im(\phi^*)\in \SV(\SO')$ for every $\phi\in (M')^{-1}(\SO')\cap \SX^{su}$. Since $\SV (\SO')$ is a single $\bfK$-orbit, the lemma is reduced to the previously discussed case.
\end{proof}

\subsection{Bounding associated cycle via geometric theta lift}\label{BoundAC}

We are in the setting of Section \ref{DesNil}.
Recall that we have made the choice of $\SX$ in \eqref{SX} so that, as a module for the Lie algebra of $\rH(W)$, $\Omega$ (the Harish-Chandra module of $\omega$) is the submodule of $\omega$ generated by $ \omega^{\SX} $. It is well-known \cite[Theorem 4.1]{Ho79} that $\omega^{\SX}$ is one-dimensional.

\begin{defn}\label{zeta} Denote by
\[\zeta = \zeta _{V,V'}: \,\, \bfK \times \bfK'\rightarrow \C^{\times},
\]
the character of $\bfK \times \bfK'$, via which it acts on $ \omega^{\SX}$.
\end{defn}

We refer the reader to \cite[Section 4.3]{BMSZ} for an explicit description of $\zeta_{V,V'}$.

\vsp

We are now back in the setting of Proposition \ref{prop:descent}, with $\SO'\in \mathrm{Nil}_{\bfK'}(\p')$ and $\SO=\nabla(\SO')$. Let $e:=M(\phi)$, where $\phi \in (M')^{-1}(\SO')\cap \SX^{su}$.

Let $\CE$ be a $\bfK$-equivariant algebraic vector bundle over $\SO$. Its fibre
$\CE_{e}$ at $e$ is an algebraic representation of the stabilizer group $\bfK_{e}$. We also view it as a representation of the group
$(\bfK \times \bfK')_\phi$ via the pull-back through the homomorphism
\[
  (\bfK\times \bfK')_\phi\xrightarrow{\textrm{the projection to the first factor}} \bfK_{e}.
\]
We may thus view $\CE_{e} \otimes \zeta $ as a representation of $(\bfK\times \bfK')_\phi$ and by taking the coinvariant space
$(\CE_{e} \otimes \zeta )_{ \bfK_\phi}$, we get an algebraic representation of $\bfK'_{e'}$, via the exact sequence \eqref{exact}.
Write
\[\CE':= \check \vartheta_{\SO}^{\SO'}(\mathcal E)\] for the  $\bfK'$-equivariant algebraic vector bundle over $\SO'$ whose fibre at $e'$ equals this coinvariant space representation. In this way, we get an exact functor $\check \vartheta_{\SO}^{\SO'}$ from the category of $\bfK$-equivariant algebraic vector bundle over $\SO$ to the category of $\bfK'$-equivariant algebraic vector bundle over $\SO'$. This exact functor induces a  homomorphism of the Grothendieck groups:
\[
   \check \vartheta_{\SO}^{\SO'}:  \CK (\SO)\rightarrow  \CK (\SO').
\]
The above homomorphism is independent of the choice of $\phi$ in Proposition \ref{prop:descent}.

Now assume that $\CO'\in \mathrm{Nil}_{\bfG'}(\g')$ is in the image of the moment map $\SM'$, and let $\CO:= {\nabla }(\CO')$ be its descent. We define the geometric theta lift to be the homomorphism
\[
 \check \vartheta_{\CO}^{\CO'}: \CK(\CO)\rightarrow \CK(\CO')
\]
such that
\begin{equation}\label{checkthetaO}
 \check\vartheta_{\CO}^{\CO'}(\CE)= \sum_{\SO'\textrm{ is a $\bfK'$-orbit in $\CO'\cap \p'$,  $\, \nabla (\SO')=\SO$}}
 \check\vartheta_{\SO}^{\SO'}(\CE),
\end{equation}
for any $\bfK$-orbit $\SO $ in $\CO\cap \p$, and any $\bfK$-equivariant algebraic vector bundle $\CE$ over $\SO$.

Recall that $\CO'\in \mathrm{Nil}_{\bfG'}(\g')$ is parametrized by a partition of size $\dim \bfV'$. We refer to $\CO'$ as a pure nilpotent if its partition contains no 1's, or $\mathbf c_1(\CO')=\mathbf c_2(\CO')$, where $\mathbf c_1(\CO')$ (resp. $\mathbf c_2(\CO')$)
is the length of the first (resp. second) column of the Young diagram of $\CO'$.

\begin{defn} \label{RegDes}
Given $\CO ={\nabla }(\CO')$, we say $\CO'$ is regular for $ {\nabla }$ if
\[
\text{either $\CO$ is the pure descent of $\CO'$ (see \eqref{not:descent2}), or $\CO'$ is a pure nilpotent}.
\]
\end{defn}

The relevance of the notion just defined is clear by the following lemma, which follows easily from \cite[Theorems 5.2 and 5.6]{DKPC}.

\begin{lem}\label{lifreg} Suppose that $\CO ={\nabla } (\CO')$ and $\CO'$ is regular for $ {\nabla } $. Then
\[
  \check \bftheta_{\SV}^{\SV'}(\CO)=\CO', \text { i.e., } \SM'(\SM^{-1}(\overline{\CO}))=\overline{\CO'}.
\]
\end{lem}

In the regular descent setting, we have the following upper bound of associated cycles.

\begin{thm}[{\cite[Theorem 8.17]{BMSZ}}]
\label{GDS.AC}
Let $\CO'\in \Nil_{\bfG'}(\g')$ and $\CO ={\nabla}(\CO') \in \Nil_{\bfG}(\g)$.
Suppose that $\CO'$ is regular for $\nabla $. Then for any $\CO$-bounded $(\g,\bfK)$-module $\Pi$ of finite length,  $\check{\Theta}_{V}^{V'}(\Pi)$ is $\CO'$-bounded and
    \[
    \AC_{\CO'}(\check \Theta_V^{V'}(\Pi))\preceq \check \vartheta_{\CO}^{\CO'}(\AC_{\CO}(\Pi)),
    \qquad \text{in $\CK(\CO')$}.
  \]
\end{thm}

\begin{rem} When $(G,G')$ is in the stable range with $G$ the smaller member, Loke and Ma proves that the above is an equality (\cite[Theorems C and D]{LM}).
\end{rem}

\begin{proof}[Proof of Theorem \ref{GDS.AC} (outline)]
By the technique of \cite[Sections 3.2 and 3.3]{LM}, there exists a good filtration $\CF$ on $\Pi$, a good filtration $\CF'$ on $\check \Theta (\Pi)$, and a surjective $(\C[\p'], \bfK')$-module homomorphism
\begin{equation}\label{surpkm}
  \check \Theta (\mathrm{Gr}(\Pi,\CF)) \rightarrow \mathrm{Gr}(\check \Theta(\Pi), \CF').
\end{equation}
Here $\check \Theta(\sigma)$, for a $(\C[\p], \bfK)$-module $\sigma $, is the commutative (check) theta lift of $\sigma$ defined by
\[
   \check \Theta(\sigma):=(\C[\SX]\otimes_{\C[\p]} \sigma \otimes \zeta )_{\bfK} \qquad (\textrm{the  coinvariant space}),
\]
which is naturally a  $(\C[\p'], \bfK')$-module. Here $\zeta$ is as in Definition \ref{zeta}.

Put $\sigma:=\mathrm{Gr}(\Pi,\CF)$, which is an $\CO$-bounded finitely generated $(\C[\p], \bfK)$-module. By the regularity assumption of the orbit descent and Lemma \ref{lifreg}, we clearly have
\[\SM^{-1}(\overline{\CO}))\subset \SM^{-1}(\overline{\CO'}).\]
This implies  that
\[
 ( I_{\overline{\CO'}\cap \p'})^i\cdot \C[\SX] \subset I_{\overline{\CO}\cap \p}\cdot  \C[\SX]\qquad \textrm{for some } i\in \NN.
\]
Here and as before, $I_{\overline{\CO'}\cap \p'}$ denotes the radical ideal of $\C[\p']$ corresponding to the closed subvariety $\overline{\CO'}\cap \p'$, and similarly for $I_{\overline{\CO}\cap \p}$. Since $\sigma$ is $\CO$-bounded, we conclude that
$\check \Theta (\sigma )$ is $\CO'$-bounded. (See the definition of $\CO$-boundedness in \eqref{eq:Obound}.) By \eqref{surpkm}, $\mathrm{Gr}(\check \Theta(\Pi), \CF')$ is $\CO'$-bounded, and therefore
$\check \Theta(\Pi)$ is $\CO'$-bounded.

For the assertion on the associated cycle, we need a closer examination of the commutative theta lift versus geometric theta lift. Here is the main assertion: (\cite[Proposition 8.16]{BMSZ})
\begin{equation}\label{surAC}
\mathrm{AC}_{\CO'}( \check \Theta (\sigma))\preceq  \check \vartheta_{\CO}^{\CO'}(\mathrm{AC}_{\CO}(\sigma)), \qquad \text{in $\CK(\CO')$}.
\end{equation}
This is shown in two steps. As the first step, one shows that \eqref{surAC} is in fact an equality, if $\sigma$ is a finitely generated $(\C[\overline{\CO}\cap\p], \bfK)$-module. This follows from an algebraic version of the Frobenius reciprocity, together with some geometric facts of regular descent (see \cite[Section 8.3]{BMSZ}). For the second step, the right exactness of the functor
$\sigma \mapsto \check \Theta (\sigma )$ implies the subadditivity of the functor $\sigma \mapsto \mathrm{AC}_{\CO'}( \check \Theta (\sigma))$. The inequality in \eqref{surAC} follows from this subadditivity and the assertion in the first step.

Finally,
\begin{eqnarray*}
      && \mathrm{AC}_{\CO'}(\check \Theta (\Pi))\\
      &=&  \mathrm{AC}_{\CO'}(\mathrm{Gr}(\check \Theta (\Pi),\CF'))\\
             &\preceq& \mathrm{AC}_{\CO'}(\check \Theta (\sigma)) \quad\  \qquad (\textrm{by \eqref{surpkm}})\\
          &  \preceq & \check \vartheta_{\CO}^{\CO'}(\mathrm{AC}_{\CO}(\sigma))\quad \qquad (\textrm{by \eqref{surAC}})\\
               &=&\check \vartheta_{\CO}^{\CO'}(\mathrm{AC}_{\CO}(\Pi)).
               \end{eqnarray*}
\end{proof}

\section{Unitarity preservation and applications}
\label{sec:Integrals}

In this section, the local field $\rF$ will be $\R$. We will discuss a general approach on unitarity preservation (due to Barbasch, Ma, Sun and the author \cite{BMSZ}), which refines the earlier approach of Li via the $L^p$ estimate \cite{Li89}. Other previous work related to the theme of this section include \cite{Li90} and \cite{He1}. We remark that the important role of growth of matrix coefficients in unitary representation theory of classical groups was first elucidated by Howe in the context of his theory of rank in \cite{HoRank}.

We shall shift the notation, to facilitate uniform treatment. After complexfication, the standard module $V$ (of $G$) and $V'$ (of $G'$) will become classical spaces with signature $\mathsf s=(\star, p,q)$ and $ \mathsf s'=(\star', p',q')$, where $\star$ and $\star'$ are Howe dual of each other. Recall from Section \ref{subsec:MMM} that a classical space with signature $\mathsf s$ comes with the structure data $(V_{\mathsf s}, \la\,,\,\ra_{\mathsf s}, J_{\mathsf s},L_{\mathsf s})$. Also the Howe dual
$\star'$ of the label $\star$ is defined as follows:
\[
\star':=\widetilde{C}, \ D, \  C, \ B, \ D^*,\  \textrm{ or } \ C^*
\]
respectively if
\[
\star=B,\  C, \ D, \ \widetilde{C}, \ C^*, \ \textrm{ or }\  D^*.
\]
We will use the subscript $\mathsf s$ to reflect its dependence on the classical signature $\mathsf s$, likewise for  $\mathsf s'$.

For a Casselman-Wallach representation $\pi$ of $G_{\mathsf s}$, put
\[
   \check \Theta_{\mathsf s}^{\mathsf s'}(\pi):=(\omega_{\mathsf s, \mathsf s'}\widehat \otimes \pi)_{G_{\mathsf s}} \qquad (\textrm{the  Hausdorff coinvariant space}).
\]
This is a Casselman-Wallach representation of $G_{\mathsf s'}$.

\subsection{Growth of Casselman-Wallach representations}\label{CassWall}
We have the complex Cartan decompsition:
\[
  \g_{\sfs}=\mathfrak{k}_{\sfs}\oplus \p_{\sfs},
\]
where $\g_{\sfs}$ is the complexified Lie algebra of $G_{\sfs}$. Write $\p_\mathsf s^{J_\mathsf s}$ for the centralizer of $J_\mathsf s$ in $\p_\mathsf s$, which is a real form of $\p_\mathsf s$. For the group $G_\mathsf s$, we also have the Cartan decomposition:
\[
  G_\mathsf s=K_\mathsf s\cdot \exp(\p_\mathsf s^{J_\mathsf s}).
\]

\begin{defn}
Denote by  $\Psi_\mathsf s$ the function on $G_\mathsf s$ satisfying the following conditions:
\begin{itemize}
\item it is both left and right $K_\mathsf s$-invariant;
\item for all $g\in \exp(\p_\mathsf s^{J_\mathsf s})$,
\[
  \Psi_\mathsf s(g)=\prod_{a} \left(\frac{1+a}{2}\right)^{-\frac{1}{2}},
\]
 where $a$ runs over all eigenvalues of $g: V_\mathsf s\rightarrow V_\mathsf s$, counted with multiplicities.
\end{itemize}
\end{defn}
Note that all the  eigenvalues of $g \in \exp(\p_\mathsf s^{J_\mathsf s})$ are positive real numbers, and come in pairs of the form $\{a, a^{-1}\}$. Thus $0<\Psi_\mathsf s(g)\leq 1$ for all $g\in G_\mathsf s$.

\begin{rem} For a classical group, its standard module is in a sense more basic than the group. From this perspective, the bi-$K_\sfs$-invariant function $\Psi_\sfs$ may be viewed as the most basic bi-$K_\sfs$-invariant function on the classical group $G_{\mathsf s}$.
\end{rem}

Denote by $\Xi_\sfs$ the bi-$K_\sfs$-invariant Harish-Chandra's $\Xi$ function on $G_\sfs$ (\cite[Section 4.5]{Wa1}).

Set
\begin{equation}\label{def:nus}
  \nu_\mathsf s:=\begin{cases}
    \abs{\mathsf s},\quad &\textrm{if }\star\in \{C, \widetilde C\};\\
     \abs{\mathsf s}-1,\quad & \textrm{if }\star= C^*;\\
      \abs{\mathsf s}-2,\quad & \textrm{if }\star\in \{B,D\};\\
       \abs{\mathsf s}-3,\quad & \textrm{if }\star= D^*,
  \end{cases}
\end{equation}
where
\begin{equation}\label{eq:abss}
 \abs{\mathsf s}:=p+q, \qquad \qquad \text{for } \mathsf s=(\star, p,q).
\end{equation}

By the well-known  estimate of Harish-Chandra's $\Xi$ function (\cite[Theorem 4.5.3]{Wa1}), we have

\begin{lem}\label{boundpsi}
There exists a real number $C_\sfs>0$ such that
\[
  \Psi_\sfs^{\nu_\sfs}(g)\leq C_\sfs\cdot \Xi_\sfs (g)\quad\textrm{ for all $g\in G_\sfs$}.
\]
\end{lem}

For any Casselman-Wallach representation $\pi$ of $G_\mathsf s$, write $\pi^\vee$ for its contragredient representation, which is a  Casselman-Wallach representation  of $G_\mathsf s$ equipped with a $G_{\mathsf s}$-invariant, non-degenerate, continuous bilinear map
 \[
   \la\,,\,\ra: \pi \times\pi^\vee\rightarrow \C.
 \]

Recall that a bi-$K_\sfs$-invariant positive function  $f$ on $G_\sfs$ is said to be of logarithmic growth if the function $x\mapsto f(\exp(x))$ on $\p_\mathsf s^{J_\mathsf s}$ is of polynomial growth.

\begin{defn}\label{defn:growth} Let $\nu\in \R$.
\begin{enumerate}
\item[(1).] A  positive function $\Psi$ on $G_\mathsf s$ is said to be $\nu$-bounded if there is
a bi-$K_\sfs$-invariant positive function $f$ of logarithmic growth such that
\[
  \Psi(g)\leq f(g)\cdot \Psi_{\bfs}^\nu(g)\cdot \Xi_\sfs(g)\qquad \textrm{for all }g\in G_\sfs.
\]
\item[(2).] A Casselman-Wallach representation $\pi$ of $G_\mathsf s$ is said to be $\nu$-bounded if there
exist a $\nu$-bounded positive function $\Psi$ on $G_\sfs$, and continuous seminorms $\abs{\,\cdot\,}_{\pi}$ and $\abs{\,\cdot\,}_{\pi^\vee}$ on  $ \pi$ and $\pi^\vee$ respectively such that
\[
 \abs{ \la g \cdot u, v\ra}\leq \Psi(g)\cdot \abs{u}_{\pi}\cdot \abs{v}_{\pi^\vee}
\]
for all $u\in \pi$, $v\in \pi^\vee$, and $g\in G_{\mathsf s}$.
\end{enumerate}
\end{defn}

\begin{rem} Since $0<\Psi_\mathsf s(g)\leq 1$, it is evident that a  positive function $\Psi$ on $G_\mathsf s$ which is $\nu$-bounded is automatically $\mu$-bounded for every $\mu \leq \nu$. A similar statement holds for a Casselman-Wallach representation $\pi$ of $G_\mathsf s$.
\end{rem}

\begin{lem}\label{int} For every real number $\nu>0$, and every bi-$K_\sfs$-invariant positive function $f$ of logarithmic growth, the function $f\cdot \Psi_\mathsf s^\nu\cdot \Xi_\sfs^2$ is integrable with respect to a Haar measure in $G_\mathsf s$. Consequently, if
a Casselman-Wallach representation $\pi$ of $G_\mathsf s$ is $\nu$-bounded for some $\nu >0$, then the integral
\[
 \int_{G_{\mathsf s}} \la g\cdot u,v\ra \cdot \Xi_{{\mathsf s}}(g)\rd\!g
\]
is absolutely convergent for all $u,v\in \pi $.
\end{lem}
\begin{proof}
  This follows from the integral formula for $G_\mathsf s$ under the Cartan decomposition (see
 \cite[Lemma 2.4.2]{Wa1}), as well as the  estimate of Harish-Chandra's $\Xi$ function mentioned earlier.
\end{proof}

\subsection{Theta lifting by integration}
The idea of studying theta lifting by matrix coefficient integrals first appeared in Li's work \cite{Li89,Li90}. We begin with the following observation on the growth of $\omega_{\mathsf s, \mathsf s'}|_{G_{\mathsf s}\times G_{\mathsf s'}}$.

\begin{lem}\label{matrico}
 There exist continuous seminorms $\abs{\,\cdot\,}_{\mathsf s, \mathsf s'}$ and $\abs{\,\cdot\,}_{\mathsf s, \mathsf s'}^\vee$ on  $ \omega_{\mathsf s, \mathsf s'}$ and $\omega_{\mathsf s, \mathsf s'}^\vee$ respectively such that
\[
 \abs{ \la (g,g')\cdot u, v\ra}\leq \Psi_{\mathsf s}^{\abs{\mathsf s'}}(g)\cdot \Psi_{\mathsf s'}^{\abs{\mathsf s}}(g')\cdot \abs{u}_{\mathsf s, \mathsf s'}\cdot \abs{v}_{\mathsf s, \mathsf s'}^\vee
\]
for all $u\in \omega_{\mathsf s, \mathsf s'}$, $v\in \omega_{\mathsf s, \mathsf s'}^\vee$, and $(g,g')\in G_{\mathsf s}\times G_{\mathsf s'}$.
\end{lem}
\begin{proof}
  This follows from the proof of \cite[Theorem 3.2]{Li89}.
\end{proof}

 \begin{defn}\label{defn:CRcov}
A Casselman-Wallach representation of $G_{\mathsf s}$ is convergent for $\check \Theta_{\mathsf s}^{\mathsf s'}$
if it is $\nu$-bounded for some $\nu>\nu_{\mathsf s}-\abs{\mathsf s'}$.
\end{defn}

Let $\pi$ be a  Casselman-Wallach representation of $G_{\mathsf s}$ that is convergent for $\check \Theta_{\mathsf s}^{\mathsf s'}$.
Consider the integrals
\begin{equation}\label{convint00}
\begin{array}{rcl}
 (\pi \times \omega_{\mathsf s, \mathsf s'})\times (\pi^\vee \times \omega_{\mathsf s, \mathsf s'}^\vee )&\rightarrow &\C, \smallskip \\
   ((u,v),(u',v')) &\mapsto &\int_{G_{\mathsf s}} \la g\cdot u, u'\ra\cdot \la g\cdot v, v'\ra \rd\! g.
   \end{array}
 \end{equation}

\begin{lem}\label{lemconv}
The integrals in \eqref{convint00} are absolutely convergent and the map \eqref{convint00} is   continuous and multi-linear.
\end{lem}
\begin{proof}
This is a direct consequence of Lemmas \ref{boundpsi}, \ref{int} and \ref{matrico}.
\end{proof}

By Lemma \ref{lemconv}, the integrals in \eqref{convint00} yield a continuous bilinear form
\begin{equation}
\label{convint01}
 (\pi \widehat \otimes \omega_{\mathsf s, \mathsf s'})\times (\pi^\vee \widehat \otimes \omega_{\mathsf s, \mathsf s'}^\vee )\rightarrow \C.
 \end{equation}
Set
\begin{equation}\label{thetab0}
  \bar{\Theta}_{\mathbf s}^{\mathbf s'}(\pi):=\frac{\pi \widehat \otimes \omega_{\mathsf s, \mathsf s'}}{\textrm{the left kernel of \eqref{convint01}}}.
\end{equation}

The operation: $\pi \mapsto \bar{\Theta}_{\mathbf s}^{\mathbf s'}(\pi)$ will be (informally) referred to as theta lifting by integration (when suitable growth conditions of representations are met).

\begin{prop}\label{boundm}
The representation $\bar{\Theta}_{\mathbf s}^{\mathbf s'}(\pi)$ of $G_{\mathsf s'}$ is a quotient of $\check \Theta_{\mathbf s}^{\mathbf s'}(\pi)$, and is  $(\abs{\mathsf s}-\nu_{\sfs'})$-bounded.
\end{prop}
\begin{proof}
Note that the bilinear form \eqref{convint01} is  $(G_{\mathsf s}\times G_{\mathsf s})$-invariant, as well as $G_{\mathsf s'}$-invariant. Thus $\bar{\Theta}_{\mathbf s}^{\mathbf s'}(\pi)$ is a quotient of  $\check \Theta_{\mathbf s}^{\mathbf s'}(\pi)$, and is therefore a Casselman-Wallach representation. Its contragedient representation
is identified with
\[
(\bar{\Theta}_{\mathbf s}^{\mathbf s'}(\pi))^\vee:=\frac{\pi^\vee \widehat \otimes \omega^\vee_{\mathsf s, \mathsf s'}}{\textrm{the right kernel of \eqref{convint01}}}.
\]
 Lemmas \ref{int} and \ref{matrico} implies that there are continuous seminorms $\abs{\,\cdot\,}_{\pi, \mathsf s, \mathsf s'}$ and $\abs{\,\cdot\,}_{\pi^\vee, \mathsf s, \mathsf s'}$ on  $\pi \widehat \otimes \omega_{\mathsf s, \mathsf s'}$ and $\pi^\vee \widehat \otimes \omega^\vee_{\mathsf s, \mathsf s'}$ respectively such that
\[
 \abs{ \la g'\cdot u, v\ra}\leq \Psi_{\mathsf s'}^{\abs{\mathsf s}}(g')\cdot \abs{u}_{\pi, \mathsf s, \mathsf s'}\cdot \abs{v}_{\pi^\vee, \mathsf s, \mathsf s'}
\]
for all $u\in \pi \widehat \otimes \omega_{\mathsf s, \mathsf s'}$, $v\in \pi^\vee \widehat \otimes \omega^\vee_{\mathsf s, \mathsf s'}$, and $g'\in G_{\mathsf s'}$.
The proposition  then easily follows in view of  Lemma \ref{boundpsi}.
 \end{proof}

\subsection{Temperedness as source of unitarity}\label{unitarity}

For the notion of weakly containment of unitary representations, see \cite{CHH} for example.

\begin{lem}\label{weaklycont}
Suppose that $\abs{\mathsf s'}\geq  \nu_{\mathsf s}$. Then as a unitary representation of $G_{\mathsf s}$, $\hat \omega_{\mathsf s, \mathsf s'}$ is weakly contained
in the regular representation.
\end{lem}

\begin{proof}
This has been known to experts (see   \cite[Theorem 3.2]{Li89}).  Lemmas \ref{boundpsi} 
and \ref{matrico} imply that for a dense subspace of $\hat \omega_{\mathsf s, \mathsf s'}|_{G_{\mathsf s}}$, the diagonal  matrix coefficients belong to $L^{2+\epsilon}(G_{\sfs})$, for every positive real number $\epsilon $.
 Thus the lemma follows from \cite[Theorem 1]{CHH}.
\end{proof}

Note that if $\star \in\{B, C^*, D^*\}$, $\nu_{\mathsf s}$ is odd, and $\abs{\mathsf s'}$ is even. Thus the smallest $\abs{\mathsf s'}$ such that $\hat \omega_{\mathsf s, \mathsf s'}$ is weakly contained
in the regular representation is actually $\nu_{\mathsf s}+1$. Note that if $\star =\widetilde C$, a similar mismatch of parity occurs:  $\nu_{\mathsf s}$ is even and $\abs{\mathsf s'}$ is odd. Nevertheless, when $\abs{\mathsf s'}\geq  \nu_{\mathsf s}$, we may decompose $\mathsf s'$ into a direct sum  $\sfs'_1\oplus \sfs'_2$ as classical spaces, where $ \sfs'_1$ is of type $B$, and $\sfs'_2$ is of type $D$ such that $\abs{\mathsf s'_2} = \nu_{\mathsf s}$, and thus $\hat \omega_{\mathsf s, \mathsf s'_2}$ is weakly contained
in the regular representation of $G_{\mathsf s}$. This turns out to be sufficient for our unitarity preservation argument (to follow shortly). If $\star \in \{C, D\}$, both $\nu_{\mathsf s}$ and $\abs{\mathsf s'}$ are even. For these reasons, we introduce
\begin{equation}\label{{def:nus0}}
   \nu_{\mathsf s}^+:=\begin{cases}
    \nu_{\mathsf s},\quad &\textrm{if }\star \in \{C, D, \widetilde C\};\\
     \nu_{\mathsf s}+1,\quad &\textrm{if }\star \in \{B, C^*, D^*\}.\\
  \end{cases}
\end{equation}
Note that $\nu_{\mathsf s}^+$ is the smallest even integer bigger than or equal to $\nu_{\mathsf s}$, in all cases.

The following definition is a slight variation of \ref{defn:CRcov}.
 \begin{defn}\label{defn:CR33}
A Casselman-Wallach representation  of $G_{\mathsf s}$ is overconvergent  for $\check \Theta_{\mathsf s}^{\mathsf s'}$ if
 it is  $\nu$-bounded  for some $\nu>\nu_{\mathsf s}^+ -\abs{\mathsf s'}$.
\end{defn}

We will neeed the following positivity result for matrix coefficient integrals, due to Harris-Li-Sun.

\begin{prop}[{\cite[Theorem A. 5]{HLS}}] \label{positivity}
Let $G$ be a real reductive group with a maximal compact subgroup $K$. Let $\pi_1$ and $\pi_2$ be two unitary representations of $G$ such that $\pi_2$ is weakly
contained in the regular representation. Let $u_1, u_2, \cdots, u_r$ ($r\in \NN$) be vectors in $\pi_1$ such that for all $i,j=1,2, \cdots, r$,
the integral
\[
  \int_G \la g\cdot u_i, u_j\ra\,\Xi_G (g) \rd\!g 
\]
is absolutely convergent, where  $\Xi_G$ is the bi-$K$-invariant Harish-Chandra's $\Xi$ function on $G$, and $\rd\!g$ is a Haar measure on $G$.   Let $v_1,v_2,\cdots, v_r$ be  $K$-finite vectors in $\pi_2$.
Put
\[
u:=\sum_{i=1}^r u_i\otimes v_i\in \pi_1\otimes \pi_2.
\]
Then the integral
\[
\int_G \la g \cdot u,u \rangle\,\rd\! g
\]
absolutely converges to a nonnegative real number.
\end{prop}

\begin{thm}[{\cite[Theorem 5.3.5]{BMSZ}}]\label{positivity000}
Assume that $\abs{\mathsf s'}\geq \nu^+_{\mathsf s}$. Let $\pi$ be a Casselman-Wallach representation of $G_{\mathsf s}$ that is overconvergent  for $\check \Theta_{\mathsf s}^{\mathsf s'}$. If $\pi$ is unitarizable, then so is $\bar{\Theta}_{\mathsf s}^{\mathsf s'}(\pi)$.
\end{thm}

\begin{proof}
Fix an invariant continuous Hermitian inner product on $\pi$, and  write $\hat \pi$ for the completion of $\pi$ with respect to this Hermitian inner product.
The space $\pi \widehat \otimes \omega_{\mathsf s, \mathsf s'}$ is equipped with the  inner product $\la\,,\,\ra$  that is the tensor product of the ones on $\pi$ and $\omega_{\mathsf s, \mathsf s'} $. It suffices to show that
\[
  \int_{G_{\mathsf s}}\la g \cdot u,u\ra\rd\! g\geq 0
\]
for all $u$ in a dense subspace of $\pi \widehat \otimes \omega_{\mathsf s, \mathsf s'}$.

If $\nu^+_{\mathsf s}<0$, then $\star \in \{C, C^*\}$ and $\abs{\mathsf s}=0$. The theorem is trivially true. Thus we assume that $\nu^+_{\mathsf s}\geq 0$.

Since $\abs{\mathsf s'}\geq \nu^+_{\mathsf s}$, we can decompose the classical space $\sfs'$ into the direct sum $\sfs'_1\oplus \sfs'_2$, where $\abs{\mathsf s'_2}=\nu^+_{\mathsf s}$. Note that when $\star =\wt{C}$, the classical space $\mathsf s'$ is of type $B$, then $\sfs'_1$ is of type $B$ and $\sfs'_2$ is of type $D$. In all other cases, $\sfs'_1$ and $\sfs'_2$ are of the same type as
$\sfs'$.

Let $\pi_2:=\hat \omega_{\mathsf s, \mathsf s'_2}|_{G_{\mathsf s}}$, viewed as a unitary representation of $G_{\mathsf s}$. (When $\star =\wt{C}$, $G_{\mathsf s}$ is a metaplectic group, and $\pi_2$ factors through the underlying real symplectic group.)
By Lemma \ref{weaklycont}, the representation $\pi_2$ is weakly contained in the regular representation of $G_{\mathsf s}$.

Denote by $\pi_1$ the unitary representation of $G_{\mathsf s}$ given by
\[
  {\pi}_1:=\hat \pi\widehat \otimes_{\mathrm h} (\hat \omega_{\mathsf s, \mathsf s'_1}|_{G_{\mathsf s}})\qquad (\textrm{$\widehat \otimes_{\mathrm h}$ indicates the Hilbert space tensor product}).
\]
By our assumption, $\pi $ is $\nu$-bounded  for some $\nu>\nu_{\mathsf s}^+ -\abs{\mathsf s'}= -\abs{\mathsf s'_1}$, and so by Lemma \ref{matrico},  $\pi \otimes \omega_{\mathsf s, \mathsf s'_1}|_{G_{\mathsf s}}$ is $\nu$-bounded for some $\nu >0$.
 Lemma  \ref{int} then implies that the integral
\[
 \int_{G_{\mathsf s}} \la g\cdot u,v\ra \cdot \Xi_{G_{\mathsf s}}(g)\rd\!g
\]
is absolutely convergent for all $u,v\in \pi \otimes \omega_{\mathsf s, \mathsf s'_1}$. In view of the functorial property of the oscillator representation (\cite[Section 2]{HoPre2}):
\[\omega_{\mathsf s, \mathsf s'}= \omega_{\mathsf s,\mathsf s'_1} \widehat \otimes \omega_{\mathsf s, \mathsf s'_2},
\]
the theorem follows by Proposition \ref{positivity}.
\end{proof}

\subsection{Application: associated cycles in the convergent range}

The goal of this subsection is to show that the upper bound of associate cycles in Theorem \ref{GDS.AC} is actually an equality, under a suitable hypothesis on the growth of the representation $\pi$ (which depends on the dual pair under consideration). The underlying mechanism to achieve this equality is a variant of the doubling method, suitably adjusted for theta lifting by integration. Philosophically this may be compared to the proof of the conservation relations, which consists of two steps reflecting two complementary aspects of theta correspondence (non-occurrence and occurrence).

\begin{defn}
An nilpotent orbit $\CO'\in \Nil(\g_{\sfs'})$ is good for $\nabla_{\sfs}^{\sfs'}$ if it is regular for $\nabla_{\sfs}^{\sfs'}$ and satisfies the following additional condition:
\[
\begin{cases}
   \mathbf c_1(\CO')>\mathbf c_2(\CO'), \qquad  &\textrm{if $\star' \in\{B,D\}$}; \\
      \abs{\sfs}-\abs{\nabla_\mathrm{pure}(\CO')}\in\{0,1\},\qquad  &\textrm{if $\star'\in\{C, \widetilde C\}$}; \\
\abs{\sfs}=\abs{\nabla_\mathrm{pure}(\CO')}, \qquad &\textrm{if $\star' \in \{C^*, D^*\}$}.
  \end{cases}
\]
Here and as before, $\mathbf c_1(\CO')$ and $\mathbf c_2(\CO')$ denote the lengths of the first and the second columns of the Young diagram of $\CO'$, respectively. Also  $\abs{\sfs}$ is an in \eqref{eq:abss}.
\end{defn}

\begin{rem}
Suppose that $\CO'\in \Nil(\g_{\sfs'})$ is good for $\nabla_{\sfs}^{\sfs'}$. If $\star' \in\{B,D\}$, then we must have $\abs{\sfs}=\abs{\nabla_\mathrm{pure}(\CO')}$; if
$\star'\in\{C, \widetilde C\}$, and if $\abs{\sfs}-\abs{\nabla_\mathrm{pure}(\CO')}=1$, we must also have $\mathbf c_1(\CO')=\mathbf c_2(\CO')$. This is due to the requirements of the regular descent in Definition \ref{RegDes}.
\end{rem}

In the rest of this section we suppose that $\CO'\in \Nil(\g_{\sfs'})$ is good for $\nabla_{\mathsf s}^{\mathsf s'}$. Put $\CO:=\nabla_{\mathsf s}^{\mathsf s'}(\CO')$ as before. Write
 \[
  \kappa:=\begin{cases}
    \abs{\sfs'}-\abs{\sfs}-1, \qquad  &\textrm{if $\star '\in\{B,D\}$}; \\
       \abs{\sfs'}-\abs{\sfs}+1, \qquad  &\textrm{if $\star '\in\{C, \widetilde C\}$}; \\
 \abs{\sfs'}-\abs{\sfs},\qquad  &\textrm{if $\star '\in \{C^*, D^*\}$},
  \end{cases}
\]
which is a non-negative integer. Set
\begin{equation}\label{p1234}
\nu_{\sfs,\sfs'}:=-(\kappa+1)=
  \begin{cases}
 \nu_{\sfs}-\abs{\sfs'} +2,\quad& \textrm{if $\star' =C^*$ (i.e., $\star=D^*$});\\
  \nu_{\sfs}-\abs{\sfs'},\quad& \textrm{otherwise}.
   \end{cases}
   \end{equation}
Here $\nu_{\sfs}$ is defined in \eqref{def:nus}. Thus $\pi$ is convergent for $\check \Theta_{\mathsf s}^{\mathsf s'}$ whenever $\pi$ is a Casselman-Wallach representation of $G_{\sfs}$  that is  $\nu$-bounded for some
$\nu> \nu_{\sfs,\sfs'}$.

\begin{thm}[{\cite[Theorem 10.2]{BMSZ}}]\label{thm:GDS.AC} Suppose that $\CO'\in \Nil(\g_{\sfs'})$ is good for ${\nabla}_{\sfs}^{\sfs'}$, and let $\CO:={\nabla}_{\mathsf s}^{\mathsf s'}(\CO')$.
  Let $\pi$ be an $\CO$-bounded Casselman-Wallach representation of $G_{\sfs}$  that is  $\nu$-bounded for some $\nu >\nu_{\sfs,\sfs'}$. Then $\bar{\Theta}_{\mathsf s}^{\mathsf s'}(\pi)$, $\check \Theta_{\mathsf s}^{\mathsf s'}(\pi)$ and $\check \Theta_{\mathsf s}^{\mathsf s'}(\pi^{\mathrm{alg}})$ are all $\CO'$-bounded, and the following equalities in $\CK_{\sfs'}(\CO')$ hold:
      \begin{equation*}\label{ACEQ}
  \mathrm{AC}_{\CO'}(\Thetab_{\mathsf s}^{\mathsf s'}(\pi))=    \mathrm{AC}_{\CO'}(\check \Theta_{\mathsf s}^{\mathsf s'}(\pi))= \mathrm{AC}_{\CO'}(\check \Theta_{\mathsf s}^{\mathsf s'}(\pi^{\mathrm{alg}}))=   \check \vartheta_{\CO}^{\CO'}(\mathrm{AC}_{\CO}(\pi)).
  \end{equation*}
\end{thm}

The above theorem has the following consequence on nonvanishing, similar to Corollary \ref{nonvanish1}.

\begin{cor} \label{nonvanish2} We are in the setting of Theorem \ref{thm:GDS.AC}. If $\check \vartheta_{\CO}^{\CO'}(\mathrm{AC}_{\CO}(\pi))\ne 0$, then
\[\check \Theta_{\mathsf s}^{\mathsf s'}(\pi)\ne 0.\]
\end{cor}

\vsp

For the proof of Theorem \ref{thm:GDS.AC}, we introduce three classical signatures (see Section
\ref{subsec:MMM}):
\[
  \sfs_0:=(\dot \star, \kappa, \kappa):=\begin{cases}
    (\star, \kappa,  \kappa ), \qquad  &\textrm{if $\star \neq B$}; \\
        (D, \kappa,  \kappa),\qquad  &\textrm{if $\star =B$},
  \end{cases}
\]
\[
  \sfs'':=(\star, p+\kappa, q+\kappa)\qquad\textrm{and}\qquad \dot{\sfs}:=(\dot \star, \abs{\sfs}+\kappa, \abs{\sfs}+\kappa). \]
Note that
\begin{enumerate}
\item[(1).] $\sfs_0$ and $\dot{\sfs}$ are both split classical spaces, and are of the same type $\star$, except when $\star =B$;
\item[(2).] $\sfs''$ is a classical space in the same Witt towel as $\sfs$.
\end{enumerate}

We have the decompositions as classical spaces:
\[
  V_{\sfs''}= V_{\sfs}\oplus  V_{\sfs_0},
\]
\[
  V_{\dot{\sfs}}= V_{\sfs''}\oplus V_{\sfs ^-}= V_{\sfs}\oplus  V_{\sfs_0}\oplus V_{\sfs ^-}.
\]

We pick a polarization
\[
  V_{\sfs_0}=X_{\sfs_0}\oplus Y_{\sfs_0}.
\]

This leads to a polarization
\[
  V_{\dot{\sfs}}= (V_{\sfs}^\triangle \oplus X_{\sfs_0})\oplus (V_{\sfs}^\nabla \oplus Y_{\sfs_0}), \quad \text{where}
\]
\[V_{\sfs}^\triangle =\{(v,v) \in  V_{\sfs ^-}\oplus V_{\sfs}\}, \text{ and }V_{\sfs}^\nabla =\{(v,-v) \in  V_{\sfs ^-}\oplus V_{\sfs}\}.\]

The Lagrangian spaces $X_{\dot {\sfs}}:= (V_{\sfs}^\triangle \oplus X_{\sfs_0})$ (of $V_{\dot{\sfs}}$) and $X_{\sfs_0}$ (of
 $V_{\sfs_0}$) yield the Siegel parabolic subgroups
 \[
  P_{\dot {\sfs}}=R_{\dot {\sfs}}\ltimes N_{\dot {\sfs}}\subset G_{\dot {\sfs}} \qquad \textrm{and}\qquad P_{\sfs_0}=R_{\sfs_0}\ltimes N_{\sfs_0}\subset G_{\sfs_0}.
 \]

Let
\[
  P_{\sfs'',\sfs_0}=R_{\sfs'',\sfs_0}\ltimes N_{\sfs'',\sfs_0}=(G_{\sfs}\cdot R_{\sfs_0})\ltimes N_{\sfs'',\sfs_0}\subset G_{\sfs''}
\]
be the parabolic subgroup of $G_{\sfs''}$ stabilizing $X_{\sfs_0}$, where $R_{\sfs'',\sfs_0}$ the Levi subgroup stabilizing both $X_{\sfs_0}$ and $Y_{\sfs_0}$, and $N_{\sfs'',\sfs_0}$ is the unipotent radical. Also $R_{\sfs_0}\subseteq G_{\sfs_0}$ is the Levi subgroup stabilizing both $X_{\sfs_0}$ and $Y_{\sfs_0}$. Write $\mathfrak r_{\sfs_0}$ and $\mathfrak n_{\sfs'',\sfs_0}$ for the complexified Lie algebras of $R_{\sfs_0}$ and $N_{\sfs'',\sfs_0}$ respectively so that the complexified Lie algebra of $P_{\sfs'', \sfs_0}$ equals $(\g_{\sfs}\times \mathfrak r_{\sfs_0})\ltimes \mathfrak n_{\sfs'',\sfs_0}$.
Write
\[
  \CO'':= \Ind_{\sfs}^{\sfs''}  (\CO)\in \Nil(\g_{\sfs''})
\]
for the induced orbit, namely the unique nilpotent orbit in $\Nil(\g_{\sfs ''})$ that contains a non-empty Zariski open subset of $\CO + \mathfrak n_{\sfs'',\sfs_0}$. One checks that $\CO''$ is good for ${\nabla}_{\mathsf s'}^{\sfs''}$ and ${\nabla}_{\mathsf s'}^{\sfs''}(\CO'')=\CO'$. Thus we have the following double descents of nilpotent orbits:
\[\CO''\xrightarrow{\nabla_{\mathsf s'}^{\sfs''}}\CO'\xrightarrow{\nabla_{\mathsf s}^{\sfs'}}\CO.\]

\vsp

The proof of Theorem \ref{thm:GDS.AC} will involve performing another theta lifting from $\sfs'$ to $\sfs''$ and then comparing the composition $\Thetab^{\sfs''}_{\sfs'}(\Thetab^{\sfs'}_{\sfs}(\pi))$ (the double theta lifting by integration) with certain unitary degenerate principal series representation. The degenerate principal series enters the picture, since $\dot{\sfs}=\sfs''\oplus \sfs^{-}$, and so by a variant of the doubling method, we can obtain the double theta lift $\Thetab^{\sfs''}_{\sfs'}(\Thetab^{\sfs'}_{\sfs}(\pi))$ of $\pi$ from $\sfs$ to $\sfs''$, by integrating the full theta lift of the trivial representation of $\sfs'$ to $\dot{\sfs}$ against $\pi$ (viewed as a representation of $G_{\sfs^-}$). Note that since the split rank of $\dot{\sfs}$ is
\[\abs{\sfs}+\kappa = \begin{cases}
    \abs{\sfs'}-1, \qquad  &\textrm{if $\star '\in\{B,D\}$}; \\
       \abs{\sfs'}+1, \qquad  &\textrm{if $\star '\in\{C, \widetilde C\}$}; \\
 \abs{\sfs'},\qquad  &\textrm{if $\star '\in \{C^*, D^*\}$},
  \end{cases}
\]
the full theta lift of the trivial representation of $\sfs'$ to $\dot{\sfs}$, written earlier as $\Omega_{\sfs'}^{\dot{\sfs}} (\one) $, is in the degenerate principal series of $G_{\dot {\sfs}}$, induced by a unitary character of the Siegel parabolic subgroup $P_{\dot {\sfs}}$. See Theorem \ref{embedding}, for the case $\star '\in\{B,D\}$. The key idea is that since we know the associated cycle of the unitary degenerate principal series (it suffices to know its weak associated cycle), we will then obtain a lower bound of the associated cycle for each step of the theta lifting. Together with the upper bound of the associated cycle proved in Section \ref{sec:AC}, we will achieve the equality of the associated cycles. The growth requirement of $\pi$ and the goodness requirement of the descent ensure that no hiccups will occur during the entire operation.

The following is the key technical lemma, which computes the weak associate cycle of $\Thetab^{\mathsf s''}_{\mathsf s'}(\Thetab_{\mathsf s}^{\mathsf s'}(\pi))$. In fact we will compute the weak associate cycle of  $\Thetab^{\mathsf s''}_{\mathsf s_1'}(\Thetab_{\mathsf s}^{\mathsf s_1'}(\pi))$, for every $\sfs_1'$ in the set  $S_{\star', \abs{\sfs'}}$, where
\[\label{sstark}
 S_{\star', j}:=\{\textrm{classical signature $\sfs_1'$ of the form $(\star', p_1,q_1)$ with $p_1+q_1=j$}\}, \quad j\in \NN.
\]

\begin{lem}[{\cite[Lemma 10.6]{BMSZ}}]
\label{lem:GDS.AC2}
  Let $\pi$ be an $\CO$-bounded Casselman-Wallach representation of $G_{\sfs}$  that is  $\nu$-bounded for some
$\nu>\nu_{\sfs,\sfs'}$. Then we have an equality of weak associated cycles (see \eqref{weak} for the definition):
 \begin{eqnarray*}
  && \left( \bigoplus_{\sfs'_1\in S_{\star', \abs{\sfs'}} } \mathrm{AC}_{\CO''}( \Thetab^{\mathsf s''}_{\mathsf s'_1}(\Thetab_{\mathsf s}^{\mathsf s'_1}(\pi)))\right)^{\mathrm{weak}}
\\
&=& \left( \bigoplus_{\sfs'_1\in S_{\star', \abs{\sfs'}}}   \check \vartheta^{\sfs'', \CO''}_{\sfs'_1, \CO'}(\check \vartheta^{\sfs'_1, \CO'}_{\sfs,\CO}(\mathrm{AC}_{\CO}(\pi)))\right)^{\mathrm{weak}}.
\end{eqnarray*}
Here we have added $\sfs, \sfs'$ in the notation $\check \vartheta_{\sfs, \CO}^{\sfs', \CO'}$ to emphasize the classical spaces involved.
\end{lem}

\begin{proof}[Sketch of proof] This follows from the structure of degenerate principal series on the unitary axis (\cite{LZ1,LZ2,Yamana}). For $\star\in \{C,\tilde{C}\}$, we have the following isomorphism:
\[
    \bigoplus_{\sfs'_1\in S_{\star', |\sfs'|}}   \Thetab^{\sfs''}_{\sfs'_1}(\Thetab^{\sfs'_1}_{\sfs}(\pi))\cong \Ind_{P_{\sfs'',\sfs_0}}^{G_{\sfs''}} (\pi\otimes \chi_0)
    \oplus  \Ind_{P_{\sfs'',\sfs_0}}^{G_{\sfs''}} (\pi \otimes \chi' _0).
    \]
Here $P_{\sfs'',\sfs_0}=(G_{\sfs}\cdot R_{\sfs_0})\ltimes N_{\sfs'',\sfs_0}$, $\chi _0$ and $\chi_0'$ are the two characters (resp. genuine characters) of $R_{\sfs_0}$ of order $2$ (resp. $4$), for $\star =C$ (resp. $\widetilde{C}$).
\end{proof}

\begin{proof} [Proof of Theorem \ref{thm:GDS.AC}]

For every $\sfs_1'\in S_{\star', \abs{s'}}$ we have
surjective $(\g_{\sfs'_1}, K'_{\sfs'_1})$-module homomorphisms
\[
\check \Theta_{\mathsf s}^{\mathsf s'_1}(\pi ^{\mathrm{alg}})\rightarrow \left (\check \Theta_{\mathsf s}^{\mathsf s'_1}(\pi )\right )^{\mathrm{alg}}\rightarrow
 \left (\Thetab_{\mathsf s}^{\mathsf s'_1}(\pi )\right )^{\mathrm{alg}}.
 \]
 Thus Theorem \ref{GDS.AC} implies that  $\check \Theta_{\mathsf s}^{\mathsf s'_1}(\pi^{\mathrm{alg}})$, $\check \Theta_{\mathsf s}^{\mathsf s'_1}(\pi)$ and $\Thetab_{\mathsf s}^{\mathsf s'_1}(\pi)$ are all $\CO'$-bounded,  and
\begin{equation*}\label{ineq000}
   \mathrm{AC}_{\CO'}( \Thetab_{\mathsf s}^{\mathsf s'_1}(\pi))\preceq  \mathrm{AC}_{\CO'}(\check \Theta_{\mathsf s}^{\mathsf s'_1}(\pi))\preceq  \mathrm{AC}_{\CO'}(\check \Theta_{\mathsf s}^{\mathsf s'_1}(\pi ^{\mathrm{alg}}))\preceq \check \vartheta_{\sfs, \CO}^{\sfs_1, \CO'}(\mathrm{AC}_{\CO}(\pi^{\mathrm{alg}})).
  \end{equation*}
  Consequently,
\begin{equation}\label{acco}
   \mathrm{AC}_{\CO'}( \Thetab_{\mathsf s}^{\mathsf s'_1}(\pi))\preceq   \check \vartheta_{\sfs, \CO}^{\sfs'_1, \CO'}(\mathrm{AC}_{\CO}(\pi)).
  \end{equation}

One checks that $\Thetab_{\mathsf s}^{\mathsf s'_1}(\pi )$ is convergent (actually overconvergent) for $ \check \Theta_{\mathsf s'_1}^{\mathsf s''}$ so that $\Thetab_{\mathsf s'_1}^{\mathsf s''}( \Thetab_{\mathsf s}^{\mathsf s'_1}(\pi))$ is defined.  Similar to \eqref{acco}, we have that
\begin{equation}\label{acco2}
   \mathrm{AC}_{\CO''}( \Thetab_{\mathsf s'_1}^{\mathsf s''}(\Thetab_{\mathsf s}^{\mathsf s'_1}(\pi)))\preceq   \check \vartheta_{\sfs'_1, \CO'}^{\sfs'', \CO''}(\mathrm{AC}_{\CO'}( \Thetab_{\mathsf s}^{\mathsf s'_1}(\pi))).
  \end{equation}
Combining \eqref{acco} and \eqref{acco2}, we obtain that
\begin{equation}\label{acineq}
   \mathrm{AC}_{\CO''}( \Thetab_{\mathsf s'_1}^{\mathsf s''}(\Thetab_{\mathsf s}^{\mathsf s'_1}(\pi )))\preceq   \check \vartheta_{\sfs'_1, \CO'}^{\sfs'', \CO''}(\check \vartheta_{\sfs, \CO}^{\sfs'_1, \CO'}(\mathrm{AC}_{\CO}(\pi ))).
\end{equation}
Then Lemma \ref{lem:GDS.AC2} implies that the inequality in \eqref{acineq} is in fact an equality. In particular,
\begin{equation}\label{acineq2}
   \mathrm{AC}_{\CO''}( \Thetab_{\mathsf s'}^{\mathsf s''}(\Thetab_{\mathsf s}^{\mathsf s'}(\pi)))=  \check \vartheta_{\sfs', \CO'}^{\sfs'', \CO''}(\check \vartheta_{\sfs, \CO}^{\sfs', \CO'}(\mathrm{AC}_{\CO}(\pi))).
\end{equation}

In view of \eqref{acco}, write
\begin{equation}\label{extra}
  \check \vartheta_{\sfs, \CO}^{\sfs', \CO'}(\mathrm{AC}_{\CO}(\pi ))=\mathrm{AC}_{\CO'}( \Thetab_{\mathsf s}^{\mathsf s'}(\pi ))+\CE',\qquad\textrm{where $\ \CE'\in \CK^+_{\sfs'}(\CO')$}.
\end{equation}
Then
\begin{eqnarray}
\nonumber \check \vartheta_{\sfs', \CO'}^{\sfs'', \CO''}(\check \vartheta_{\sfs, \CO}^{\sfs', \CO'}(\mathrm{AC}_{\CO}(\pi)))
\nonumber &=&\mathrm{AC}_{\CO''}( \Thetab_{\mathsf s'}^{\mathsf s''}(\Thetab_{\mathsf s}^{\mathsf s'}(\pi))) \quad \quad \quad (\textrm{by \eqref{acineq2}})\\
\nonumber &\preceq &  \check \vartheta_{\sfs', \CO'}^{\sfs'', \CO''}(\mathrm{AC}_{\CO'}( \Thetab_{\mathsf s}^{\mathsf s'}(\pi)))\, \quad \quad (\textrm{by \eqref{acco2}})\\
\label{prece}  &\preceq &  \check \vartheta_{\sfs', \CO'}^{\sfs'', \CO''}(\mathrm{AC}_{\CO'}( \Thetab_{\mathsf s}^{\mathsf s'}(\pi))+\CE')   \\
\nonumber &=&\check \vartheta_{\sfs', \CO'}^{\sfs'', \CO''}(\check \vartheta_{\sfs, \CO}^{\sfs', \CO'}(\mathrm{AC}_{\CO}(\pi))).
\, \quad (\textrm{by \eqref{extra}})
\end{eqnarray}
In particular,  the equality holds in \eqref{prece}. Therefore  $\check \vartheta_{\sfs', \CO'}^{\sfs'', \CO''}(\CE')=0$. Using
$0\preceq \CE'\preceq \check \vartheta_{\sfs, \CO}^{\sfs', \CO'}(\CE)$ for some $\CE\in \CK^+_{\sfs}(\CO)$, one checks that $\CE'=0$. This completes the proof of Theorem \ref{thm:GDS.AC}.
\end{proof}

\vsp

\begin{center} {\bf Acknowledgements}
\end{center}

The author has learned much from Roger Howe, Steve Kudla, and Jian-Shu Li. He also profited much from talking with Kyo Nishiyama, Dihua Jiang, Wee Teck Gan, and Binyong Sun. The author thanks all of them for their mathematics and friendship. Thanks are also due to the referee for careful reading and helpful suggestions.

\end{document}